%% file: Prism_revision_3.tex
\documentclass[11pt, a4paper]{amsart}

\input{include}

\def\lSm{\mathbf{lSm}}
\def\SmlSm{\mathbf{SmlSm}}
\def\Sm{\mathbf{Sm}}

\newcounter{elno}

\begin{document}

\input{defin}

\author{Federico Binda}
\address{Dipartimento di Matematica ``Federigo Enriques'',  Universit\`a degli Studi di Milano\\ Via Cesare Saldini 50, 20133 Milano, Italy}
\email[F. Binda]{federico.binda@unimi.it}

\author{Tommy Lundemo}
\address{Mathematisch Instituut, Universiteit Utrecht, Budapestlaan 6, 3584 CD Utrecht, Netherlands}
\email[T. Lundemo]{t.lundemo@uu.nl}

\author{Alberto Merici}
\address{Institut f\"ur Mathematik, Universit\"at Heidelberg\\ MATHEMATIKON, Im Neuenheimer Feld 205, 69120  Heidelberg, Germany}
\email[A. Merici]{merici@mathi.uni-heidelberg.de}

\author{Doosung Park}
\address{Department of Mathematics and Informatics, University of Wuppertal, Germany}
\email[D. Park]{dpark@uni-wuppertal.de}

\thanks{F.B. was partially supported by the PRIN 2022 ``The arithmetic of motives and L-functions''. T.L.\ is funded by the NWO-grant VI.Veni.242.129; A.M. was supported by Horizon Europe's Marie
Sk{\l}odowska-Curie Action PF 101103309 ``MIPAC'' and is funded by the Deutsche Forschungsgemeinschaft (DFG, German Research 
Foundation)
TRR 326 \textit{Geometry and Arithmetic of Uniformized Structures}, 
project number 444845124. D.P.~and T.L.~were partially supported by the research training group GRK 2240 ``Algebro-Geometric methods in Algebra, Arithmetic and Topology.''}
\title[Log prisms and motivic sheaves]{Logarithmic prismatic cohomology, motivic sheaves, and comparison theorems}

\begin{abstract}
    We prove that (logarithmic) prismatic and (logarithmic) syntomic cohomology are representable in the category of logarithmic motives. As an application, we obtain Gysin maps for prismatic and syntomic cohomology, and we explicitly identify their cofibers. We also prove a smooth blow-up formula and we compute prismatic and syntomic cohomology of Grassmannians.
    In the second part of the paper, we develop a descent technique inspired by the work of Nizio\l~ on log $K$-theory. Using the resulting \emph{saturated descent}, we prove de Rham and crystalline comparison theorems for log prismatic cohomology, and the existence of Gysin maps for $A_{\inf}$-cohomology. 
\end{abstract}
\maketitle

\tableofcontents

\section{Introduction}
Let us fix a prime $p$. Prismatic cohomology is a cohomology theory for $p$-adic formal schemes introduced by Bhatt and Scholze \cite{BSPrism}, building on the work of Bhatt, Morrow, and Scholze \cite{BMS1}, \cite{BMS2}. Classical $p$-adic cohomology theories, such as \'etale, de Rham, and crystalline cohomology, are all specializations of prismatic cohomology.

There are several approaches to prismatic cohomology. The construction in \cite{BSPrism} uses a mixed-characteristic refinement of the crystalline site, based on the notion of \emph{prism}; a sort of ``deperfection'' of (mixed characteristic) perfectoid rings. Given a prism $(A,I)$, the prismatic cohomology of a $p$-adic smooth formal scheme $\kX$ over $A/I$ is then defined as $R\Gamma_{\Prism} (\kX /A) := R\Gamma( (\kX/A)_{\Prism}, \cO_{\Prism})$, where $(\kX/A)_{\Prism}$ is the prismatic site of $\kX$ relative to $(A,I)$, and $\cO_{\Prism}$ is the structure sheaf. It comes equipped with a  decreasing filtration $\Fil_\rN^{\geq *}$, the \emph{Nygaard filtration}, and a Frobenius endomorphism $\phi$.  In general, prismatic cohomology is not complete with respect to the Nygaard filtration. The Nygaard completion of prismatic cohomology can be described in purely homotopy-theoretic terms, using trace invariants such as topological Hochschild homology. This is the approach of \cite{BMS2}, and it is compared with that of Bhatt--Scholze in \cite[Theorem 13.1]{BSPrism}. This point of view has lent itself to numerous striking results in homotopy theory: let us highlight the computations in algebraic $K$-theory of Antieau--Krause--Nikolaus \cite{AntieauKrauseNikolaus} and the generalization of prismatic cohomology to commutative ring spectra of Hahn--Raksit--Wilson \cite{HahnRaksitWilson}.

The  theory of prismatic cohomology in the sense of \cite{BSPrism} admits a natural extension to formal schemes $\kX$ that are not necessarily smooth over $A/I$, but merely smooth in the sense of \emph{logarithmic geometry}. A logarithmic version of the prismatic site has been defined by Koshikawa \cite{Koshikawa} and further developed by Koshikawa and Yao \cite{Koshikawa-Yao}, generalizing the comparison theorems of \cite{BSPrism} to the semi-stable setting.

In \cite{BLPO-BMS2}, the first, second, and fourth named authors, together with {\O}stv{\ae}r, defined a variant of logarithmic prismatic cohomology based on the approach of \cite{BMS2}. The role of topological Hochschild homology is played by Rognes' \emph{logarithmic} topological Hochschild homology \cite{rognes}. The resulting theory $R\Gamma_{\cPrism}(-)$ is called the (absolute) Nygaard-complete logarithmic prismatic cohomology. It is a (strict) \'etale sheaf on the category of quasisyntomic log formal schemes and comes equipped with a Nygaard filtration and a Frobenius endomorphism $\phi$. We may then define log syntomic cohomology analogously to the non-log variant of \cite{BMS2}, namely as the equalizer
\[
\begin{tikzcd}
R\Gamma_{\rm Fsyn}(-,\Z_p(i)):=\mathrm{eq}\Bigl(\Fil_\rN^{\geq i}R\Gamma_{\cPrism}(-)\{i\} \ar[r,shift left = 1,"\mathrm{can}"]\ar[r,shift right = .5,"\phi"'] &R\Gamma_{\cPrism}(-)\{i\}\Bigr)
\end{tikzcd}\] of the Frobenius and the canonical map.

This paper is dedicated to further study of log syntomic cohomology and Nygaard-complete log prismatic cohomology. We exhibit these as theories defined on a suitable category of mixed motives and prove comparison theorems analogous to those of \cite{BMS2}.   Our results are independent of \cite{Koshikawa}, \cite{Koshikawa-Yao}, but we remark  that, as in \cite{BMS2}, it is possible to construct   (depending on the choice of a perfectoid base $R$) a non-Nygaard completed version of log prismatic cohomology out of the completed version, by left Kan extension from the log smooth case (see \cite[Section 5]{LogBeilinson} for a comparison with the site-theoretic definition). Our results also apply to the resulting theory $\cPrism^{\rm nc}_{-/R}$.
Let us also point out that one main difference from \cite{Koshikawa}, \cite{Koshikawa-Yao} at the level
 of statements is that the log structure on the coefficient prism is in general trivial in
 this paper, with the exception of the Breuil-Kisin case (see Section \ref{sec:BK_cohomology}).

\

We will now give an overview of the main results of this work and discuss the methods used to achieve them. 

\subsection{Motivic realizations} Let us fix a base (formal) quasisyntomic $p$-adic scheme $S$. Crystalline cohomology, Hodge cohomology, prismatic cohomology, and syntomic cohomology are all examples of cohomology theories that fail to be ${\A}^1$-homotopy invariant (for prismatic cohomology this indeed follows from the crystalline comparison theorem, and for syntomic cohomology follows from e.g. \cite[Theorem 1.12(5)]{BMS2}, together with the fact that $\TC$ is not $\A^1$-invariant). 
In particular, the functors 
\[
R\Gamma_{\cPrism}(-) \text{ and } R\Gamma_{\Fsyn}(-, \Z_p)\in \Fun(\QSyn_S, \cD(\Z_p))\] 
do not factor through the $\infty$-category $\mathcal{FDA}^{\rm eff}(S;\Z_p)$ of (effective, formal, \'etale) motivic sheaves of Ayoub, Morel, Voevodsky and others, defined as the subcategory of $\A^1$-local objects of the category of \'etale sheaves of complexes of $\cD(\Z_p)$-modules on smooth formal $S$-schemes. This category is a powerful tool for studying \emph{rational} $p$-adic cohomology theories.  For example, it can be used to deduce finiteness results \cite{VezzMonsky}, to define categories of coefficients \cite{DegNiz}, Gysin triangles \cite{Deg08}, and to study in large generality monodromy operators and weight filtrations \cite{BKV}, \cite{BGV}. 

By enlarging the framework to the category of logarithmic schemes, we gain an extra level of flexibility that not only admits more general objects (such as formal models with semi-stable reduction) but also allows a different motivic formalism to be applied. Namely, we can work in the category of logarithmic motivic sheaves introduced in \cite{BPO}, \cite{BPOCras}, \cite{BPO-SH}. Its construction mimics Voevodsky's definition; one starts with the category of strict \'etale or Nisnevich sheaves of $\Lambda$-modules $\Shv_{\set}(\SmlSm_S, \cD(\Lambda))$ or $\Shv_{\sNis}(\SmlSm_S, \cD(\Lambda))$  on the category of fs log smooth log schemes separated of finite type over a base $S$, for a ring of coefficients $\Lambda$. One then considers the localization with respect to the projections $(\P^n_S, \P^{n-1}_S)\times_S X \to X$ for every $X\in \SmlSm_S$ and $n\geq 1$. 
The effect of inverting this class of maps is striking, as it recovers several results that were previously only available in the $\A^1$-local setting, notably the construction of Gysin sequences and a Thom space isomorphism for Chern-oriented theories. By passing to $\P^1$-spectra one obtains a non-effective version, $\logDA(S, \Lambda)$, which is our preferred framework in this text. A verbatim translation of the above construction allows for $p$-adic formal log schemes as input. We denote the resulting category  $\logFDA(S, \Lambda)$ or $\logFDA(S, \cD(\Lambda))$.

In the first part of the paper, our main goal is to prove the representability of the aforementioned cohomology theories in the log motivic categories of formal schemes. Theorem \ref{thm:spectra-prism} gives oriented motivic spectra
\[
\E^{\cPrism}  \in \logFDA(S, \cD(\Z_p)), \quad \E^{\Fil \cPrism} \in \logFDA(S, \widehat{\cD\cF}(\Z_p))
\]
for any quasi-compact and quasi-separated $p$-adic formal base scheme $S$ (with trivial log structure). If $S\in \mathrm{QSyn}_R$ for a perfectoid ring $R$, we also obtain oriented motivic spectra $\E^{\cPrism^{\rm nc}}$ and $\E^{\Fil \cPrism^{\rm nc}}$ from the non-Nygaard complete versions. Both $\E^{\cPrism}$ and $\E^{\Fil \cPrism}$ are $E_\infty$-rings, and there are functorial ring maps induced by the divided Frobenius and the canonical morphism in prismatic cohomology. Forming their equalizer, we then obtain a motivic ring spectrum $\E^{\Fsyn}$, representing log syntomic cohomology. The orientation is induced by the existence of Chern classes, due to the work of Bhatt--Lurie \cite{BhattLurie}. 
Similarly, building on our previous work \cite{BLPO-BMS2}, \cite{BLPO-HKR}, \cite{mericicrys}, we can construct motivic ring spectra $\E^{\rm Hdg}$, $\E^{\rm dR}$, representing Hodge and de Rham cohomology for log smooth formal $S$-schemes, as well as $\E^{\rm dRW}$, $\E^{\rm crys}$ representing de Rham-Witt and log crystalline cohomology over a perfect field $k$ of characteristic $p$ (with trivial log structure). See \S \ref{sec:deRham_Hodge_Crys_spectra}.

This has as an immediate application the following 
result (see Theorem \ref{thm:gysin-prism}):
\begin{thm} 
Let $S\in \QSyn$. Let $\kZ\to \kX$ be a morphism of $p$-adic smooth formal schemes over $S$ such that it is locally the $p$-completion of a pure codimension $d$ closed immersion of smooth schemes over $S$.
Let $\Bl_\kZ(\kX)$ denote the blow-up of $\kX$ in $\kZ$ and $E$ be the exceptional divisor, so that $(\Bl_\kZ(\kX),E)$ is log smooth over $S$. Then for all $j$ there are Gysin maps in $\cD(\Z_p)$, functorial in $(\kX,\kZ)$, 
\begin{align*}
\mathrm{gys}_{\kZ/\kX}^{\cPrism}\colon &R\Gamma_{\cPrism}(\kZ)\{j-d\}[-2d]\to R\Gamma_{\cPrism}(\kX)\{j\}\\ 
\mathrm{gys}_{\kZ/\kX}^{\Fil  \cPrism}\colon  &\Fil_\rN^{\geq j-d}R\Gamma_{\cPrism}(\kZ)\{j-d\}[-2d]\to \Fil_\rN^{\geq j} R\Gamma_{\cPrism}(\kX)\{j\}\\
\mathrm{gys}_{\kZ/\kX}^{\Fsyn}\colon &R\Gamma_{\Fsyn}(\kZ,\Z_p(j-d))[-2d]\to R\Gamma_{\Fsyn}(\kX,\Z_p(j))
\end{align*}
whose homotopy cofibers are respectively given as 
\begin{align*}
    &R\Gamma_{\cPrism}(\Bl_\kZ(\kX),E)\{j\}\\
    &\Fil_\rN^{\geq j} R\Gamma_{\cPrism}(\Bl_\kZ(\kX),E)\{j\}\\
    & R\Gamma_{\Fsyn}((\Bl_\kZ(\kX),E),\Z_p(j)).
\end{align*}
If $S\in \mathrm{QSyn}_R$ with $R$ perfectoid, we have similar Gysin maps for the non-completed prismatic cohomology of \cite{BSPrism}, relative to the perfect prism $(A_{\inf}(R),\ker(\theta))$:
\begin{align*}
\mathrm{gys}_{\kZ/\kX}^{\Prism}\colon &R\Gamma_{\Prism}(\kZ/A_{\inf}(R))\{j-d\}[-2d]\to R\Gamma_{\Prism}(\kX/A_{\inf}(R))\{j\}\\ 
\mathrm{gys}_{\kZ/\kX}^{\Fil\Prism}\colon  &\Fil_\rN^{\geq j-d}R\Gamma_{\Prism}(\kZ/A_{\inf}(R))\{j-d\}[-2d]\to \Fil_\rN^{\geq j} R\Gamma_{\Prism}(\kX/A_{\inf}(R))\{j\}.
\end{align*}
The homotopy cofibers are respectively given as 
\begin{align*}
    &R\Gamma_{\cPrism^{\rm nc}}((\Bl_\kZ(\kX),E)/R)\{j\}\\
    &\Fil_\rN^{\geq j} R\Gamma_{\cPrism^{\rm nc}}((\Bl_\kZ(\kX),E)/R)\{j\}
\end{align*}
\end{thm}
In particular, not only  we can define a push-forward for (non log) prismatic or syntomic cohomology along the inclusion of a smooth subscheme $\kZ\subset \kX$, but we can also identify the cofiber of this map in terms of an explicit \emph{logarithmic} cohomology theory. Coherently with the philosophy sketched in \cite{BLPO-HKR}, we might call the corresponding bifiber sequence the \emph{residue sequence} for prismatic (resp.\ syntomic) cohomology.

The Gysin map described above appears after a trivialization of the motivic Thom space of the normal bundle of $\kZ$ in $\kX$, which is possible in light of the orientation on prismatic cohomology. This is an analog of the homotopy purity theorem of Morel and Voevodsky \cite[Theorem 2.23]{MV}, proved in \cite[\S 7.4]{BPO}, \cite{BPO-SH}. In the $\A^1$-local setting, the logarithmic term would be replaced by the open complement $\mathfrak{U}=\kX-\kZ$: we interpret the pair $(\Bl_\kZ(\kX), E)$ as a relative compactification of $\mathfrak{U}$, which preserves the homotopy type. Note that more general Gysin maps (where also $\kX$ and $\kZ$ are allowed to have a non-trivial log structure) can be constructed, using \cite[Theorem 7.5.4]{BPO}. See Remark \ref{rmk:generalGysin} for more details.

Another byproduct of the representability result of prismatic and syntomic cohomology is a computation of the prismatic and syntomic cohomology of the Grassmannian $\mathrm{Gr}(r,n)$ for every $r,n$. See Theorem \ref{thm:coho_Grassmannian} for the precise  notation.

\begin{thm} Let $S\in \QSyn$. Let $\kX$ be a $p$-adic log smooth formal scheme over $S$.
Then there are isomorphisms of bigraded rings, functorial in $X$:
    \begin{align*}
 \phi_{r,n}^{\cPrism} \colon  H^*_{\cPrism}(\kX)\{\bullet\} \otimes_{\mathbb{Z}} Z_{r,n} &\xrightarrow{\sim} H^*_{\cPrism}(\mathrm{Gr}(r, n) \times \kX)\{\bullet\}\\
 \phi_{r,n}^{\Fil\cPrism}\colon \Fil^{\geq \bullet}  H^*_{\cPrism}(\kX)\{\bullet\} \otimes_\mathbb{Z} Z_{r,n} &\xrightarrow{\sim} \Fil^{\geq \bullet}H^*_{\cPrism}(\mathrm{Gr}(r, n)\times \kX)\{\bullet\} \\
\phi_{r,n}^{\Fsyn}\colon H^*_{\Fsyn}(\kX,\Z_p(\bullet))\otimes_\Z Z_{r,n} &\xrightarrow{\sim} H^*_{\Fsyn}(\mathrm{Gr}(r, n)\times \kX,\Z_p(\bullet)),
\end{align*}
where $Z_{r,n}$ is an explicit ring in \eqref{Znd}.
If $S\in \QSyn_R$ with $R$ perfectoid, a similar statement holds for the non-completed prismatic cohomology relative to $R$. 
\end{thm}

Yet another application of our general motivic formalism, we obtain blow-up formulas for syntomic and prismatic cohomology, see Theorem \ref{thm:blow-up_main}.
\begin{thm}  Let $S\in \QSyn$.
Let $\kZ\to \kX$ be a morphism of $p$-adic smooth formal schemes over $S$ such that it is locally the $p$-completion of a pure codimension $d$ closed immersion of smooth schemes over $S$.
Let $\kX' = \Bl_\kZ(\kX)$. Then there are equivalences, functorial in $(\kX,\kZ)$,
    \begin{align*}
 R\Gamma_{\cPrism}(\kX)\{j\} \oplus \bigoplus_{0<i< d} R\Gamma_{\cPrism}(\kZ)\{j-i\} [-2i] &\xrightarrow{\sim} R\Gamma_{\cPrism}(\kX')\{j\}\\
  \Fil^{\geq j} R\Gamma_{\cPrism}(\kX)\{j\} \oplus  \bigoplus_{0<i< d} \Fil^{\geq j-i} R\Gamma_{\cPrism}(\kZ)\{j-i\}[-2i]&\xrightarrow{\sim} \Fil^{\geq j} R\Gamma_{\cPrism} \{j\}(\kX') \\
R\Gamma_{\Fsyn}(\kX, \Z_p(j)) \oplus \bigoplus_{0<i< d} R\Gamma_{\Fsyn}(\kZ, \Z_p(j-i)) [-2i] &\xrightarrow{\sim} R\Gamma_{\Fsyn}(\kX', \Z_p(j))
    \end{align*}
If $S\in \QSyn_R$ with $R$ perfectoid, a similar statement holds for the non-completed prismatic cohomology relative to $R$. 
\end{thm}
The previous result was first proved in \cite[Corollary 9.4.2]{BhattLurie} for syntomic cohomology, with the extra assumption (in the non-formal setting) that $Z$ is the image of a section $s\colon Y\to X$ of a smooth morphism of schemes $f\colon X\to Y$.

\

Let us briefly sketch how one establishes the representability of prismatic cohomology. The first step is to prove that it is $(\P^n,\P^{n-1})$-invariant for every $n\geq1$. Using quasisyntomic descent, we can reduce to the case where the base is a perfectoid ring $R$: in this case, the HKR-filtration \cite[Theorem 5.15]{BLPO-HKR} induces as in \cite{BMS2} a finite \emph{conjugate} filtration on the graded pieces with respect to the Nygaard filtration $\gr^i_N$, with graded quotients given by exterior powers of the Gabber's logarithmic cotangent complex. It is explained in \cite[\S 8.3]{BLPO-HKR} that this is  $(\P^n,\P^{n-1})$-invariant. 
The second step is $\P^1$-stability: we prove this by showing that the graded commutative monoid $(R\Gamma_{\cPrism}(-)\{i\})_{i\in \Z}$ admits a Chern orientation, and this will follow from the construction of the prismatic first Chern class of \cite{BhattLurie} in the non-log situation. 

\subsection{Saturated descent and comparison theorems}
In the second part of the paper, we concentrate on more explicit descriptions of log prismatic cohomology. Namely, we prove that it can be computed via descent of (non-logarithmic) prismatic cohomology in many cases of interest. The main application of this is a comparison result with log crystalline cohomology: 

\begin{thm}[Theorems \ref{thm:crys_comp_triv} and \ref{thm:crys_comp_logpoint}]\label{thm:intro-crys-comp}
Let $k$ be a perfect field of characteristic $p$. For all log smooth $X\in \lSm_{(k,\triv)}$ and $Y\in \lSm_{(k,\N)}$ with $\ul{Y}$   reduced, we have equivalences of filtered $E_\infty$-rings that depends functorially on $X$ and $Y$:
\[
R\Gamma_{\cPrism}(X/k)\simeq R\Gamma_{\rm crys}(X/W(k))\quad 
R\Gamma_{\cPrism}((\ul{Y},\partial Y\oplus_\N\N_{\rm perf})/k)\simeq R\Gamma_{\rm crys}(Y/W(k,\N)).
\]
\end{thm}

Note that the inclusion $k^\times\to k$ (resp.\ the map $k^\times\oplus \N \to k$ given by $(x,0)\mapsto x$ and $(x,n)\mapsto 0$ for $x\in k^\times$ and $n\geq 1$) is the log structure on the point $\Spec(k,\mathrm{triv})=\Spec(k)$ (resp.\ standard log point $\Spec(k,\N)$).

We will reduce both statements to the ordinary, non-logarithmic crystalline comparison of \cite{BMS2}. The central tool for this is \emph{saturated descent}, a technique first considered in \cite{Niziol}. For this, we form the \v Cech nerve in the category of saturated monoids and get for a monoid map $M\to N$ a descent datum\[
\begin{tikzcd}
(A,M)\ar[r] &(A\otimes_{\Z[M]}\Z[N],N)\ar[r,shift left=1]\ar[r,shift right=.5]  &(A\otimes_{\Z[M]}\Z[N\oplus_M^{\rm sat}N],N\oplus_M^{\rm sat}N)\ar[r,shift left = .5]\ar[r,shift left=1.5]\ar[r,shift right=.5]   &\ldots
\end{tikzcd}
\]
We prove that this datum induces the following descent property:
\begin{thm}[Theorem \ref{thm:main-sat-descent-prism} and Example \ref{ex:log-smooth-descent}]
    Let $B$ be a $p$-complete ring with bounded $p^\infty$-torsion. For $p$-adic formal log smooth scheme $\kX$ over $B$,
    the log prismatic cohomology $R\Gamma_{\cPrism}(\kX)$ can be computed \'etale locally as limit of (non-log) prismatic cohomology\[
    \varlim_{\Delta}
     (\cPrism_{A\cotimes_{B\langle M\rangle}B\langle M_{\rm perf}\oplus_M^{\rm sat}\ldots \oplus_M^{\rm sat}M_{\rm perf}\rangle})
    \] 
    where $(A,M)$ is a strict chart on $\kX$. 
\end{thm}
We stress that the rings $A\cotimes_{B\langle M\rangle}B\langle N\oplus_M^{\rm sat}\ldots \oplus_M^{\rm sat}N\rangle$ are not obtained by tensor products in any category of rings: rather, these \emph{saturated} tensor products actively extract data from the log structure. 
By the same method, one obtains similar results for the log crystalline cohomology (see Theorems \ref{thm:main-sat-descent-dRW} and  \ref{thm:main-sat-descent-dRW2}), which combined with the crystalline comparison of \cite{BMS2} gives Theorem \ref{thm:intro-crys-comp}. 

This descent datum is considered also in \cite[\S 6]{Koshikawa-Yao} applied to the Kummer \'etale cohomology with $p$ inverted, and in \emph{loc.\ cit.} it is used to deduce a Kummer \'etale comparison isomorphism \cite[Theorem 6.1 and Proposition 6.13]{Koshikawa-Yao}. 

\subsection{A log version of Breuil-Kisin cohomology}
In the last part of the paper, we use the above results to construct a logarithmic version of the Breuil--Kisin cohomology of \cite[Theorem 1.2]{BMS2}
(see \cite[Example 1.6]{Koshikawa} for a construction based on site-theoretic log prismatic cohomology).
For $K$ a discrete valuation field with ring of integers $\cO_K$, uniformizer $\varpi$ and residue field $k$, let $\mathfrak{S} = W(k)\llbracket z\rrbracket$. There is a surjective map $\tilde \theta: \mathfrak{S} \to \cO_K$ determined by the inclusion $W(k)\subseteq \cO_K$ and $z \mapsto \varpi$. The kernel of this map is generated by an Eisenstein polynomial $E=E(z) \in \mathfrak{S}$ for $\varpi$. Let $\phi$ be the endomorphism of $\mathfrak{S}$ determined by the Frobenius on $W(k)$ and $z\mapsto z^p$. The following makes reference to the notion of \emph{log Cartier type} introduced in \cite[Definition 4.8]{katolog}. 

\begin{thm}[Theorems \ref{thm:BK-horizontal} and \ref{thm:BK-semistable}]\label{thm:intro-BK}
Let $\lSm_{\cO_K}$ (resp. $\lSm_{(\cO_K,\varpi)}^{\rm cart}$) be the category of fs log schemes of finite type that are log smooth over $\cO_K$ (resp. log smooth over $(\cO_K,\varpi)$ of Cartier type). There are cohomology theories
\begin{align*}
R\Gamma_{\cPrism}(-/\mathfrak{S})&\colon \lSm_{(\cO_K,\triv)}^{\rm op}\to \cD(\mathfrak{S})\\
R\Gamma_{\cPrism}(-/(\mathfrak{S},\varpi))&\colon \lSm_{(\cO_K,\varpi)}^{\rm cart, op}\to \cD(\mathfrak{S})
\end{align*}
with values in $(p,z)$-complete $E_\infty$-algebras together with a $\phi$-linear Frobenius endomorphism having the following features:
\begin{enumerate}
\item (Breuil--Kisin)  The Frobenius endomorphism $\phi$ induces equivalences
\begin{align*}
R\Gamma_{\cPrism}(-/\mathfrak{S})\otimes_{\mathfrak{S},\phi}\mathfrak{S}[1/E]&\simeq R\Gamma_{\cPrism}(-/\mathfrak{S})[1/E]\\
R\Gamma_{\cPrism}(-/(\mathfrak{S},\varpi))\otimes_{\mathfrak{S},\phi}\mathfrak{S}[1/E]&\simeq R\Gamma_{\cPrism}(-/\mathfrak{S})[1/E].
\end{align*} 
\item (de Rham comparison) After scalar extension along $\theta=\tilde \theta \circ \phi$, one recovers log de Rham cohomology:
\begin{align*}
R\Gamma_{\cPrism}(-/\mathfrak{S})\otimes_{\mathfrak{S},\theta}^L\cO_K &\simeq R\Gamma_{\rm dR}(-/\cO_K)\\
R\Gamma_{\cPrism}(-/(\mathfrak{S},\varpi))\otimes_{\mathfrak{S},\theta}^L\cO_K &\simeq R\Gamma_{\rm dR}(-/(\cO_K,\varpi))
\end{align*}
\item (crystalline comparison) After scalar extension along the map $\mathfrak{S}\to W(k)$ which is the Frobenius on $W(k)$ and sends $z$ to $0$, one recovers log crystalline cohomology of the special fiber:
\begin{align*}
R\Gamma_{\cPrism}(-/\mathfrak{S})\otimes_{\mathfrak{S}}^L W(k)&\simeq R\Gamma_{\rm crys}((-)_k/W(k))\\
R\Gamma_{\cPrism}(-/(\mathfrak{S},\varpi))\otimes_{\mathfrak{S}}^L W(k)&\simeq R\Gamma_{\rm crys}((-)_{(k,\N)}/(W(k,\N)))
\end{align*}
\end{enumerate}
\end{thm}
The strategy of the construction of these cohomology theories follows \cite[\S 11]{BMS2}: they are defined in terms of relative log $\THH$ over the log cyclotomic bases $(\Sph[\N])$ and $(\Sph[\N],\N)$, together with a Frobenius descent: this last part, in turns, follows from a logarithmic version of the Segal conjecture (see Proposition \ref{prop:log-segal} and \ref{prop:log-segal2}).

Similarly to the absolute case, the cohomology theory $R\Gamma_{\cPrism}(-/\mathfrak{S})$ is again representable in the log motivic category of formal schemes over $\cO_K$, building an oriented ring spectrum (Theorem \ref{thm:BK-motivic}):\[
\E^{\rm BK}\in \CAlg(\logFDA(\cO_K,\mathfrak{S})).
\] 
This has the following interesting application (see Corollary \ref{cor:gysin-Ainf}):
let $\kZ\to \kX$ be a morphism of $p$-adic proper smooth formal schemes over $\cO_K$ such that it is locally the $p$-completion of a pure codimension $d$ closed immersion of smooth schemes over $\cO_K$.
Let $C$ be the completion of an algebraic closure of the $p$-adic completion of $K(p^{1/p^{\infty}})$. Then there is a functorial Gysin map\[
R\Gamma_{A_{\inf}}(\kZ_{\cO_C})\{j-d\}[-2d]\to R\Gamma_{A_{\inf}}(\kX_{\cO_C})\{j\},
\]
where $R\Gamma_{A_{\inf}}$ is the $A_{\inf}$-cohomology of \cite{BMS1}, and whose homotopy cofiber is\[
R\Gamma_{\cPrism}((\Bl_\kZ(\kX),E)/\mathfrak{S})\{j\}\otimes_{\mathfrak{S}}^LA_{\inf}.
\]
Similarly, we have the computation of the cohomology of Grassmannians and the smooth blow-up formula. We remark that the Gysin or residue sequences considered here are compatible with the corresponding sequences for de Rham, Hodge, Hodge-Tate, and crystalline cohomology, via the comparison isomorphisms.

\subsection{Future perspectives} In a follow up paper \cite{LogBeilinson} we have considered the problem of understanding under which conditions the Nygaard-complete variant of log prismatic cohomology considered here indeed comes to life as the Nygaard-completion of the site-theoretic definition pursued by Koshikawa--Yao. Such a comparison result should, among other things, open the door to pursue log variants of the global motivic filtration of Bhatt--Lurie \cite[\S 6.4]{BhattLurie}. We consider this, in turn, to be a stepping stone to understanding the even filtration of Hahn--Raksit--Wilson in the setting of log ring spectra, and its compatibility with the residue sequences considered here. 

In light of Theorem \ref{thm:main-sat-descent-prism}, it is tempting to \emph{define} log $K$-theory in terms of saturated descent. We intend to pursue this definition in future work and relate it to trace invariants based on Rognes' log topological Hochschild homology. A similar definition of log $K$-theory is introduced by Nizio{\l} \cite[\S 4]{Niziol}.

\subsection*{Notation}  We will use freely the notation from the table in \cite[\S 1.1]{BPO-SH}: in particular every (log) smooth morphism will always be considered separated and of finite type, unless specified.
Recall also that the class of such log smooth morphisms is denoted by $\mathrm{lSm}$.
 
The base of every category of (log) motives will always be assumed to be quasi-compact and quasi-separated, but not necessarily noetherian. 

For $\Lambda$ a commutative ring and $\cC$ a $\Lambda$-linear stable $\infty$-category, we let $\Map_{\cC}$ be the mapping spectrum enriched in $\cD(\Lambda)$, and $\uMap_{\cC}$ the internal mapping spectrum of $\cC$. If there is no confusion, we will suppress $\cC$.

We will use the following convention regarding monoid rings: For any commutative monoid $M$ and a fixed prime number $p$,  we let ${\Z}_p \langle M \rangle$ denote the $p$-completion of the monoid ring ${\Z}_p[M]$. For a different $p$-complete base ring $R$,
we shall write $R\langle M \rangle$ for the $p$-complete tensor product $R \widehat{\otimes}_{{\Z}_p} {\Z}_p \langle M \rangle$.

If $\cC$ is an $\infty$-category endowed with a topology $\tau$ and $\cD$ is an $\infty$-category with limits and colimits, we denote by $\Shv_\tau(\cC, \cD)\subset \PSh(\cC,\cD)$ the full sub $\infty$-category of $\cD$-valued $\tau$-sheaves, and by $\Shv^\wedge_\tau(\cC,\cD) \subset \Shv_\tau(\cC, \cD)$ its full subcategory of $\tau$-hypersheaves. We will write $L_\tau$ to denote the $\tau$-sheafification functor as well as the $\tau$-hypersheafification functor (depending on the context).

We will always consider Gabber's log cotangent complex (\cite[\S 8]{Olsson}) and the corresponding version of log derived de Rham cohomology: we refer to  \cite[\S 6]{Bhatt2012padicDD} for the construction and \cite[\S 3]{BLPO-HKR} and \cite[\S 2]{BLPO-BMS2} for a quick recall. We will freely refer to \cite[\S 2 and 3]{BLPO-BMS2} for the results and constructions on logarithmic (topological) Hochschild homology, as introduced by Rognes \cite{rognes}.

To ease the notation, for $R$ a perfectoid ring, we will write $\Prism_{/R}$ and $\cPrism_{/R}$ instead of $\Prism_{/A_{\inf}(R)}$ and $\cPrism_{/A_{\inf}(R)}$.

\subsection*{Acknowledgements}
The authors would like to thank Ben Antieau, Teruhisa Koshikawa, John Rognes, Shuji Saito, and Paul Arne {\O}stv{\ae}r for many valuable discussions and comments. The authors thank the anonymous referees for their meticulous analysis of the paper, providing helpful comments which filled some gaps in the arguments and led to an improved presentation.

\section{Log motives and abstract representability results}\label{sec:representability}

Throughout this section, we fix a $p$-complete ring $R$ with bounded $p^\infty$-torsion (that is, the $p$-primary torsion is killed by a fixed power of $p$) and trivial log structure. This assumption is necessary to apply the motivic formalism developed in \cite{BPO}, \cite{BPO-SH}, but will be dropped in later sections where more general comparison results are considered.

\subsection{Motivic sheaves}
We begin by recalling some facts about $p$-adic formal (log) schemes in the non (necessarily) locally noetherian setting. Log structures on formal schemes will always be defined \'etale locally. For locally noetherian  fs log schemes, this does not interfere with the setting of \cite{BPO} in light of Nizio\l's theorem \cite[Theorem 5.6]{NiziolToric}.

\begin{defn}\label{def:formal-scheme} (see \cite[\S A.1]{Koshikawa})
    A \emph{bounded $p$-adic formal log scheme} over $R$ is a pair consisting of a bounded $p$-adic formal scheme $\ul{\kX}$ (i.e.\ with bounded $p^\infty$-torsion) and a map of \'etale sheaves of monoids $\alpha\colon M_{\kX}\to \cO_{\ul{\kX}}$ that induces an isomorphism $\alpha^{-1}\cO_{\kX}^\times \simeq \cO_{\kX}^{\times}$. For $(A,M)$ a $p$-complete pre-log $R$-algebra such that $A$ has bounded $p^\infty$-torsion, we let $\Spf(A,M)$ be the $p$-adic formal log scheme $\Spf(A)$ with log structure given by the logification of the map $M\to \cO_{\Spf(A)}$, where $M$ is the constant sheaf on the \'etale site of $\Spf(A)$. 
A bounded $p$-adic formal log scheme is \emph{quasi-coherent} \cite[\S 1.1]{BeilinsonPeriod} if \'etale-locally it admits a chart, i.e. \'etale-locally on $\ul{\kX}$ it is isomorphic to $\Spf(A,M)$ for $(A,M)$ a $p$-complete pre-log $R$-algebra such that $A$ has bounded $p^\infty$ torsion.  We say that $\kX$ is \emph{coherent} if it is quasi-coherent and if the charts $(A, M)$ can be chosen with $\overline{M} = M/M^*$ a finitely generated monoid.  
\end{defn}

Recall that for a commutative ring $B$, an object $M\in \cD(B)$ is called $p$-completely flat if $M\otimes^{L}_B B/pB\in \cD(B/p B)$ is concentrated in degree $0$ and a flat $B/p B$-module. Similarly, a $B$-module $N$ is $p$-completely flat if $N[0]\in \cD(B)$ is $p$-completely flat. 
\begin{defn}
 We let $\FlQSm_R$ (formally log quasismooth) denote the category of quasi-coherent bounded $p$-adic formal log schemes $\kX$ over $R$ that, strict \'etale locally, are isomorphic to $\Spf(A,M)$ with $(A,M)$ a pre-log $R$-algebra such that $A$ is $p$-complete with bounded $p^\infty$-torsion such that $R\langle M \rangle\to A$ is $p$-completely flat and  the cotangent complex $\L_{(A,M)/R}$ is $p$-completely flat.
 We equip $\FlQSm_R$ with the strict \'etale topology, i.e. $\kY\to \kX$ is a strict \'etale cover if  it is strict and $(|\ul{\kY}|,\cO_{\ul{\kY}/p^n})\to (|\ul{\kX}|,\cO_{\ul{\kX}/p^n})$ is an \'etale cover for all $n$.
\end{defn}

\begin{rmk}
If $R$ has bounded $p^{\infty}$ torsion and $A$ is $p$-completely flat over $R$, then $A$ has bounded $p^{\infty}$-torsion by \cite[Corollary 4.8]{BMS2}. Note that the log cotangent complex is logification-invariant by \cite[Lemma 3.11]{BLPO-HKR} and satisfies strict \'etale descent, so if $(A,M)$ is a quasismooth pre-log $R$-algebra in the sense of \cite[Definition 4.5]{BLPO-BMS2} such that every point has a chart $M'\to A$ that satisfies $R\langle M'\rangle \to A$ is $p$-completely flat, then \'etale locally on $A$ we have $\L_{(A,(M')^a)}\simeq \L_{(A,M)}$. This is $p$-completely flat, so $\Spf(A,M)$ is  formally log quasismooth over $R$.
\end{rmk}

\begin{rmk}
For $X$ a \emph{fine} log scheme with bounded $p^\infty$-torsion, we consider its formal $p$-completion as the bounded $p$-adic formal log scheme $X_p^\wedge$ with $\ul{X_p^\wedge}=\ul{X}\times \Spf(\Z_p)$ with pullback log structure. If $X$ is log smooth over $R$, $X_p^\wedge$ is formally log smooth over $R$:  indeed by \cite[Theorem IV.3.3.1]{ogu}, \'etale-locally on $\ul{X}$ there is a chart $\Spec(A,M)$ such that $R[M] \to A$ is smooth. Since $R[M]$ is a free $R$-module, $R[M]$ has bounded $p^\infty$-torsion, in particular by \cite[Lemma 4.4 and 4.7]{BMS2}, the (classic) $p$-completion $R\langle M\rangle =(R[M])_p^\wedge\to A_p^\wedge$ is $p$-completely flat;
 moreover, the cofiber sequence \cite[(3.4)]{BLPO-HKR} \[
A_p^{\wedge}\otimes_A^L\L_{(A,M)/R}\to \L_{(A_p^\wedge,M)/R}\to \L_{(A_p^\wedge,M)/(A,M)}
\] 
implies that after derived $p$-completion $\L_{(A_p^\wedge,M)/R}\simeq A_p^{\wedge}\otimes_A^L\L_{(A,M)/R}$, which is $p$-completely flat as $\L_{(A,M)/R}\simeq \Omega^1_{(A,M)/R}$ by \cite[Propositions 4.5 and 4.6]{BLPO-HKR}, which is a projective (hence flat) $R$-module. The assignment $X\mapsto X_p^\wedge$ is clearly functorial and for any $\cC\in \PrLo$ it induces an adjoint pair: 
\[
\begin{tikzcd}
\Sh_{\set}(\lSm_R,\cC)\ar[r,shift left=1.5]& \Sh_{\set}(\FlQSm_R,\cC)\ar[l]
\end{tikzcd}\]
{For $\SmlSm_R\subseteq \lSm_R$ the category of log smooth schemes over $R$ such that $\ul{X}$ is smooth (see \cite[A.5.10]{BPO}): the $p$-completion induces again}\[
\begin{tikzcd}
\Sh_{\set}(\SmlSm_R,\cC)\ar[r,shift left=1.5]& \Sh_{\set}(\FlQSm_R,\cC)\ar[l]
\end{tikzcd}\]
\end{rmk}
\begin{defn}
    Let $\Lambda$ be a commutative ring. Let $(\P^n_{\Spf(R)},\P^{n-1}_{\Spf(R)})\in \FlQSm_R$ be the formal completion of the log scheme $(\P^n,\P^{n-1})$. Then we let $\logFDAeff(R,\Lambda)$ be the localization of $\Sh_{\set}(\FlQSm_R,\cD(\Lambda))$ with respect to the class maps given by the projections $(\P^n_{\Spf(R)},\P^{n-1}_{\Spf(R)})\times \kX\to \kX$
    for $\kX\in \FlQSm_R$ and $n\geq 1$. By construction, there is a canonical monoidal adjunction:
    \begin{equation}\label{eq:DAtoFDA}
        \begin{tikzcd}
    \logDAeff(R,\Lambda)\ar[r,shift left=1.5,"\Comp_!"]&\logFDAeff(R,\Lambda)\ar[l,"\Comp^*"]
    \end{tikzcd}
    \end{equation}
    where $\logDAeff(R,\Lambda)$ is the $\infty$-category of algebraic (strict) \'etale log motivic sheaves constructed in the same way starting from $\Sh_{\set}(\SmlSm_R,\cC)$ (see \cite[Proposition 5.4.2]{BPO} and \cite[Theorem 3.5.5 and 3.5.6]{BPO-SH} for all the equivalent models).
\end{defn}

\subsection{Generalities on spectra} We briefly review the process of formally inverting the tensor product with a fixed object of a symmetric monoidal stable $\infty$-category. This material is not new, and probably well-known to the experts, but we collect the relevant results for the reader's convenience. See \cite{robalo} for a detailed discussion.

\begin{defn} Let $\CAlg(\Cat_\infty) \to \Cat_\infty$ be the forgetful functor from symmetric monoidal $\infty$-categories. By \cite[Example 3.1.3.14]{HA} it admits a left adjoint, denoted $\Sym^\infty(-)$. We
let $\mathrm{Sym}^\infty (*)$ be the free symmetric monoidal $\infty$-category generated by one object $*$. This was denoted $free^\otimes(\Delta[0])$ in \cite[\S 2.1]{robalo}.  Note that the underlying $\infty$-category of $\Sym^\infty(*)$ agrees with $B\Sigma_\mathbb{N} = \coprod_n B\Sigma_n$ (see also \cite[Construction 1.3.1]{AnnalaIwasaUnivers}).
\end{defn}

\begin{defn}Let $\cC^\otimes \in \PrLo$ be a  stable and presentably symmetric monoidal $\infty$-category, and let $\cC^\Sigma =  \Fun^{\rm L}(\Sym^\infty(*), \cC)^\otimes$ be the category of symmetric sequences in $\cC$ as in  \cite[Definition 6.3.0.2]{HA}  (see \cite[Definition 6.3]{CiDeg-LocalStableHomAlg} for the $1$-categorical version), equipped with the Day convolution product \cite[\S 2.2.6]{HA}. 

We have $ \Fun^{\rm L}(\Sym^\infty(*), \cC)_{\langle1\rangle} = \prod_n \Fun^{\rm L}(B\Sigma_n, \cC)$. 
Informally, its objects are given by collections $(A_n)_{n\in \mathbb{N}}$, where each $A_n$ is an object of $\cC$ equipped with an action of $\Sigma_n$. 
\end{defn}

\begin{constr}We define an endofunctor $(-)\{-1\}$ of $\cC^\Sigma$ as follows. Given $X \in \cC^\Sigma$, we set  $X\{-1\}_0 = 0$ and  $X\{-1\}_n = \Sigma_n\times_{\Sigma_{n-1}}X_{n-1}$, where $\Sigma_n\times_{\Sigma_{n-1}}X_{n-1}$ is the equalizer 
\[
\bigoplus_{\gamma \in \Sigma_n} X_{n-1} \xrightarrow{(\tau_\sigma  -  \tau'_\sigma)_{\sigma \in \Sigma_{n-1}} } \bigoplus_{\gamma \in \Sigma_n} X_{n-1}
\]
and $\tau_\sigma$ (resp. $\tau_\sigma'$) denotes the endomorphism of $\bigoplus_{\gamma \in \Sigma_n} X_{n-1}$ given by the action of $\Sigma_{n-1} \subset \Sigma_n$ permuting the components (resp. given by the diagonal action). See \cite[\S 6.2]{CiDeg-LocalStableHomAlg}.

Let $\mathbbm{1}\{-1\}\in \cC^{\Sigma}$ be given as  $(\mathbbm{1}\{-1\})_n = \mathbbm{1}$ if $n=1$ and $0$ otherwise, where $\mathbbm{1}$ is the unit of $\cC^{\otimes}$. By definition of the convolution product, we observe that $M\{-1\} = M \otimes \mathbbm{1}\{-1\}.$ 
\end{constr}

\begin{rmk}
For every $n\in \N$, the canonical projection induces an evaluation functor 
\[ \cC^\Sigma \to \cC, \quad (A_i)_{i\in \mathbb{N}}  \mapsto A_n \]
This functor has a left adjoint $\cC \to \cC^\Sigma$, given by sending $X$ to $X\{n\} = (-)\{1\}^{\circ n} X$, where by abuse of notation we write $X$ for the image of $X$ via
\[ \cC = \Fun^L(*, \cC)  \to \prod_n \Fun^L(B\Sigma_n ,\cC) \]
given by left Kan extension along the maps $* \to B\Sigma_n$. 
\end{rmk}

Recall from \cite[Definition 2.16]{robalo} that an object $T$ in a symmetric monoidal $\infty$-category $\cC^\otimes$ is called \emph{symmetric} if the cyclic permutation $\sigma_3$ on $T\otimes T \otimes T$ is equivalent to the identity on $\cC$.

\begin{defn} Let $T$ be a symmetric object of $\cC^\otimes$, and let $S_T = \Sym^\infty(T\{-1\})\in \CAlg(\cC^\Sigma)$ be the free commutative algebra on $T\{-1\}$.  
Following \cite[Definition 6.2]{HoveySpectra} or \cite[\S 6.6]{CiDeg-LocalStableHomAlg}, we define the $\infty$-category $\PSpt_T^\Sigma(\cC)$ of symmetric $T$-pre-spectra in $\cC$ as the symmetric monoidal $\infty$-category $\Mod_{S_T}(\cC^\Sigma)$. 
\end{defn}

For $M\in \Mod_{S_T}(\cC^\Sigma)$, we have the structure map $S_T \otimes M\to M$. Precomposing with the canonical morphism $T\{-1\} \to S_T$, we obtain a map
\[ \sigma_M\colon  T\{-1\} \otimes M \simeq (T \otimes M)\{-1\} \to M\]
where the tensor product is the convolution product in $\cC^\Sigma$. By adjunction, we obtain a map 
\begin{equation}\label{eq:adjoint_assembly_map} \gamma_M \colon M\to \underline{\Map}_{\cC^\Sigma} ( T\{-1\}, M) =: \Omega_T M.\end{equation}
In components, it is given by a collection of morphisms $M_n \to \ul\Map_\cC ( T, M_{n+1})$ for $n \geq 0$. 

\begin{defn}[{\cite[Definitions 7.6 and 7.7]{HoveySpectra}, \cite[\S 6.23]{CiDeg-LocalStableHomAlg}}]
 The symmetric monoidal $\infty$-category of symmetric $T$-spectra  $\Spt_T^\Sigma(\cC)$ in $\cC$ is the Bousfield localization of the $\infty$-category $\PSpt_T^\Sigma(\cC)$ of symmetric $T$-pre-spectra with respect to the collection of morphisms $\gamma_M$ of \eqref{eq:adjoint_assembly_map} for $M\in \cC^\Sigma$. 
\end{defn}

\begin{prop}\label{prop:def-spt}
    The $\infty$-category $\Spt_T^\Sigma(\cC)$ is a presentable stable symmetric monoidal $\infty$-category, and it is naturally equivalent to the formal inversion $\cC[T^{-1}]$ of the object $T$ constructed in \cite[Proposition 2.9]{robalo}.
\end{prop}
    \begin{proof}
       Since $T$ is a symmetric object of $\cC$, this follows from Robalo's comparison of Hovey's category of symmetric spectra, which is the content of \cite[Theorem 2.26]{robalo}. 
    \end{proof}

\subsection{Log motivic spectra}

 Let $S$ be a quasi-compact and quasi-separated scheme. Recall that the Tate object $\Lambda(1)\in \logDAeff(S,\Lambda)$ is given by the splitting $\Lambda(\P^1)\simeq \Lambda\oplus \Lambda(1)[2]$. By \cite[Proposition 3.2.7]{BPO-SH}, the object $\Lambda(1)$ is symmetric in $\logDAeff(S,\Lambda)$. We define the symmetric monoidal stable $\infty$-category of log motivic spectra by\[
\logDA(S,\Lambda):= \Spt_{\Lambda(1)}^{\Sigma}\logDAeff(S,\Lambda),
\]
cf.\ Proposition \ref{prop:def-spt}. For $X\in \PSh(\SmlSm_S,\cD(\Lambda))$, we let $X(n):=X\otimes \Lambda(1)^{\otimes n}$. An object of $\logDA(S,\Lambda)$ is then a sequence $(X_i,\sigma_i)_{i\in \N}$ where $X_n$ are $(\set,(\P^\bullet,\P^{\bullet-1}))$-local
\footnote{Recall that if $S$ is finite dimensional and noetherian this is the same as $(\det,\bcube)$-local by \cite[
Theorems 3.5.5 and 3.5.6]{BPO-SH}
}
objects of $\PSh(\SmlSm_S,\cD(\Lambda))$ equipped with an action of the symmetric group $\Sigma_n$ and $\sigma_n\colon X_n(1)\to X_{n+1}$ are maps of complexes of presheaves such that: 
\begin{itemize}
    \item[i)]
The composite\[
X_n(p)\to X_{n+1}(p-1)\to \ldots \to X_{n+p}
\]
is $\Sigma_n\times \Sigma_p$-equivariant, where $\Sigma_p$ acts on the left as the permutation isomorphism of the tensor structure of $\logDAeff(S,\Lambda)$ and on the right via the embedding $\Sigma_n\times \Sigma_p\to \Sigma_{n+p}$.
\item[ii)] The map $
X_n\to \ul{\Map}(\Lambda(1),X_{n+1})
$
adjoint to $\sigma_n$ is an equivalence in $\logDAeff(S,\Lambda)$.
\end{itemize}
If $R$ is a $p$-complete ring with bounded $p^\infty$-torsion, we define $\Lambda(1)\in \logFDAeff(R,\Lambda)$ in the same way as before. It is the image of $\Lambda(1)\in \logDAeff(R,\Lambda)$ under the strong monoidal functor $\mathrm{Comp_!}$ of \eqref{eq:DAtoFDA}, so $\Lambda(1)$ is also symmetric in $\logFDAeff(R,\Lambda)$. We define the the symmetric monoidal stable $\infty$-category\[
\logFDA(R,\Lambda):= \Spt_{\Lambda_{\rm F}(1)}^{\Sigma}\logFDAeff(R,\Lambda),
\]
whose objects are sequences $(X_n,\sigma_n)$, analogously to the case above.
The adjunction \eqref{eq:DAtoFDA} induces canonically an adjunction
\begin{equation}\label{eq:DAtoFDA-stab}
        \begin{tikzcd}
    \logDA(R,\Lambda)\ar[r,shift left=1.5,"\Comp_!"]&\logFDA(R,\Lambda)\ar[l,"\Comp^*"]
    \end{tikzcd}
    \end{equation}
    
\begin{defn}
    Let $\cC^\otimes \in \PrLo$ be a presentable symmetric monoidal $\infty$-category. The category of graded objects of $\cC$ is the functor category $\Fun( \mathbb{Z}^{\rm ds}_{\geq 0}, \cC)$. It is a symmetric monoidal $\infty$-category, under the Day convolution product. A graded commutative monoid in $\cC$ is, by definition, an object of $\CAlg(\Fun( \mathbb{Z}^{\rm ds}_{\geq 0}, \cC))$, where the symmetric monoidal structure on $\mathbb{Z}^\mathrm{ds}_{\geq 0}$ is given by the sum.
\end{defn}

\begin{rmk} A graded commutative monoid $E_*$ of $\cC$ is a sequence $\{E_i\}_{i\in \N}$ of objects of $\cC$ equipped with a unit map $\eta\colon \one \to E_0$ and a coherent commutative and associative multiplication. In particular, we obtain in the homotopy category $h\cC$ the following set of data. For any pair of integers $(i,j)$ a multiplication $\mu_{i,j}\colon E_i\otimes E_j\to E_{i+j}$ such that the diagrams commute: 
\[
    \begin{tikzcd}
        E_i\ar[r,"1\otimes \eta"]\ar[dr,equal]&E_i\otimes E_0\ar[d,"\mu_{1,0}"] &E_i\otimes E_j\otimes E_k\ar[r,"1\otimes \mu_{j,k}"]\ar[d,"\mu_{i,j}\otimes 1"]&E_i\otimes E_{j+k}\ar[d,"\mu_{i,j+k}"] &E_i\otimes E_j\ar[d,"\gamma_{i,k}"]\ar[dr,"\mu_{i,k}"]\\
        &E_i&E_{i+j}\otimes E_k\ar[r,"\mu_{i+j,k}"]&E_{i+j+k}&E_j\otimes E_i\ar[r,"\mu_{j,i}"]&E_{i+j}
    \end{tikzcd}
    \]
    where $\gamma_{i,k}$ is the symmetry isomorphism. 
\end{rmk}

For $X\in \lSm_S$, let $R\Gamma(X, F) = \Map_{\PSh(\lSm_S,\cD(\Lambda))}(\Lambda(X) , F)$. We write $\widetilde{R\Gamma} (X\times \P^1, F)$ for $\Map_{\PSh(\lSm_S,\cD(\Lambda))}(\Lambda(X\times\P^1)/\Lambda(X) , F)$. 
The following proposition is then analogous to \cite[Proposition 1.4.10]{DegMaz}:

\begin{prop}\label{prop:build-spectra}
    Let $E_*$ be a graded commutative monoid in $\logDAeff(S, \Lambda)$ (resp. $\logFDAeff(S, A)$) together with a section $c\colon \Lambda(\P^1)\to E_1[2]$ in $\PSh(\SmlSm_S,\cD(\Lambda))$ (resp. in $\PSh(\FlQSm_R,\cD(\Lambda))$) such that for all $X\in \SmlSm_S$ (resp. $\FlQSm_R$) and all $i$, the following composition is an equivalence:\[
        \begin{tikzcd}
        R\Gamma(X,E_i)\ar[r]\ar[drr,bend right=8,"\simeq"] &R\Gamma(X\times \P^1,E_i\otimes \Lambda(\P^1)) \ar[r,"c"] &R\Gamma(X\times \P^1,E_i\otimes E_{1}[2])\ar[d,"\mu_{i,1}"]\\
        &&\widetilde{R\Gamma}(X\times\P^1,E_{i+1}[2])
        \end{tikzcd}
        \]
    Then there is $\mathbf{E}\in \CAlg(\logDA(S,\Lambda))$ (resp. $\in \CAlg(\logFDA(S,\Lambda))$) such that for all $X\in \lSm_S$ (resp. $\FlQSm_R$),    
\begin{align*}
\Map_{\operatorname{\mathbf{log}\mathcal{DA}}(S,\Lambda)}(\Sigma^{\infty}(X),\Sigma^{m,n}\mathbf{E}) &\simeq R\Gamma(X,E_n[m])
\\
\text{(resp.\ }
\Map_{\operatorname{\mathbf{log}\mathcal{FDA}}(S,\Lambda)}(\Sigma^{\infty}(X),\Sigma^{m,n}\mathbf{E}) &\simeq R\Gamma(X,E_n[m])\text{).}
\end{align*}

\begin{proof}
We prove it only for $\logDA$: the proof for $\logFDA$ proceeds \emph{verbatim}. Consider the symmetric sequence $\mathbf{E}\in \logDAeff(S, A)^\Sigma$ given by setting $\mathbf{E}_i = E_i$ for $i\in \Z_{\geq 0}$, where each $E_i$ is considered as $\Sigma_i$-module with trivial action. More concretely, $\mathbf{E}$ is the image of $E_*$ along the functor
\begin{equation}\label{eq:build-sequence}
\Fun^L(\Z_{\geq 0}^{\rm ds}, \logDAeff(S,\Lambda) )\to  \prod_n \Fun^L(B\Sigma_n , \logDAeff(S,\Lambda))
= \logDAeff(S,\Lambda)^\Sigma
\end{equation}
induced by the projection $B \Sigma_{\N} = \oplus_n B\Sigma_n\to \Z_{\geq 0}^{\rm ds}$. Note that the functor \eqref{eq:build-sequence} is symmetric monoidal, hence it induces a functor
\begin{equation}\label{eq:graded_to_symmetric} \CAlg(\Fun^L(\Z_{\geq 0}^{\rm ds}, \logDAeff(S,\Lambda)))\to \CAlg(\logDAeff(S,\Lambda)^\Sigma).
\end{equation}
Since $E_*$ is by assumption a graded commutative monoid, we obtain that  $\mathbf{E}$ is a commutative algebra object of $\logDAeff(S, A)^\Sigma$. 
For all $i$, the section $c$ induces maps
\begin{equation}\label{eq:assembly}
E_i(1) \to E_i\otimes\Lambda(\P^1_S)[-2]\to E_{i}\otimes E_1\to E_{i+1},    
\end{equation}
in $\logDAeff(S,\Lambda)$, which build up to a map $A(1)\{-1\} \to \mathbf{E}$ in $\logDAeff(S, A)^\Sigma$, since $A(1)$ is a symmetric object. By the universal property, we obtain a morphism of algebras $S_{A(1)}\to \mathbf{E}$, so that $\mathbf{E} \in \PSpt^\Sigma_{A(1)}(\logDAeff(S,\Lambda))$.

We are left to show that $\mathbf{E}$ is  indeed a symmetric spectrum: this amounts to checking that the induced map\[
E_i\to \ul{\Map}(A(1),E_{i+1})
\]
is an equivalence in $\logDAeff(S,\Lambda)$ for all $i$: by construction, for all $X\in \lSm_S$ the map\[
\Map(A(X),E_i)\to \Map(A(X)(1),E_{i+1})\simeq \Map(A(X\times \P^1)/A(X),E_{i+1}[2])
\]
agrees with the map $R\Gamma(X,E_i)\to \widetilde{R\Gamma}(X\times \P^1,E_{i+1}[2])$ in the assumption, hence it is an equivalence.
Since the localization $\PSpt^\Sigma_{A(1)}(\logDAeff(S,\Lambda))\to \logDA(S,\Lambda)$ is symmetric monoidal, we conclude that $\mathbf{E}\in \CAlg(\logDA(S,\Lambda))$.
\end{proof}
\end{prop}

\begin{prop}\label{prop:build-spectra-maps}
    Let $E_*$ and $E'_*$ be graded commutative monoids in $\logDAeff(S,\Lambda)$ (resp.\ in $\logFDAeff(R,\Lambda)$) with $c\colon \Lambda(\P^1)\to E_1[2]$ and $c'\colon \Lambda(\P^1)\to E_1'[2]$ that satisfy the hypotheses of Proposition \ref{prop:build-spectra}, and let $\E,\E'\in \logDA(S,\Lambda)$ (resp.\ in $\logFDA(R,\Lambda)$) be the associated spectra. Let $\psi_*\colon E_*\to E'_*$ be a map of graded commutative monoids such that $\psi_1[2]\circ c = c'$. Then there exists a map $\psi\colon \E\to \E'$ in $\CAlg(\logDA(S,\Lambda))$ (resp.\ in $\CAlg(\logFDA(R,\Lambda))$)  
    such that for all $X\in \SmlSm_S$ (resp.\ $\FlQSm_R$), the map\[
    R\Gamma(X,E_n[m])\simeq \Map(\Sigma^\infty(X),\Sigma^{m,n}\E)\xrightarrow{\psi_X}\Map(\Sigma^\infty(X),\Sigma^{m,n}\E')\simeq R\Gamma(X,E'_n[m])
    \]
    agrees with $\psi_n$.
\end{prop}

\begin{proof}
The morphism $\psi_*\colon E_*\to E_*'$ in $\CAlg(\Fun^L(\Z_{\geq 0}^{\rm ds}, \logDAeff(S,\Lambda)))$ induces a morphism 
$\psi\colon \E\to \E'$ in $\CAlg(\logDAeff(S,\Lambda)^\Sigma)$ via the symmetric monoidal functor \eqref{eq:graded_to_symmetric} induced by  $B \Sigma_{\N} = \oplus_n B\Sigma_n\to \Z_{\geq 0}^{\rm ds}$. 
    If it is a map in $\PSpt^\Sigma_{\Lambda(1)}(\logDAeff(S,\Lambda))$, then $\psi$ is indeed a map in $\CAlg(\logDA(S,\Lambda))$ satisfying the desired property. By the compatibility with $c$, we have that the diagram\[
    \begin{tikzcd}
        E_i(1)\ar[d,"\psi_i(1)"] \ar[r] &E_{i+1}\ar[d,"\psi_{i+1}"]\\
        E'_i(1) \ar[r] &E'_{i+1},
    \end{tikzcd}
    \]
    commutes, where the horizontal maps come from \eqref{eq:assembly}. By the universal property, the map $\psi$ is a map of $S_{\Lambda(1)}$-algebras. The proof for $\logFDAeff$ is analogous. 
\end{proof}

Recall from \cite[Definition 7.1.3]{BPO-SH} that a homotopy commutative monoid $\mathbf{E}$ in $\logDA(S,\Lambda)$ (resp.\ $\logFDA(R,\Lambda)$) admits a \emph{Chern orientation}, or simply that it is \emph{oriented} if there is a class $c_\infty\in \pi_0\Map(\Sigma^{\infty}(\P^\infty),\mathbf{E}(1)[2])$ whose restriction to $\P^1/{\rm pt}$ is the map\[
\Sigma^{\infty}(\P^1)\otimes \eta\colon \Sigma^{\infty}(\P^1)\to \Sigma^{\infty}(\P^1)\otimes \mathbf{E}\simeq E(1)[2].
\]

\begin{rmk}
Oriented spectra are particularly nice for the following reason (see \cite[Theorem 9.4]{BLPO-HKR}): Let  $X\in \SmlSm_S$: {by \cite[Lemma A.5.10]{BPO} $|\partial X|$ is a divisor with simple normal crossing}. Let $\ul{Z}\subseteq \ul{X}$ a smooth closed subscheme of pure codimension $d$ with normal crossing with $|\partial X|$, so that $Z:=(\ul{Z},\partial X_{|Z})\in \SmlSm_S$. Let $\partial \Bl_Z(X)$ be the log structure on $\Bl_\ul{Z}(\ul{X})$ induced by the total transform of $|\partial X|$, so that also $\Bl_Z(X):=(\Bl_\ul{Z}(\ul{X}),\partial \Bl_Z(X))\in \SmlSm_S$. Then by \cite[Theorem 3.2.14]{BPO-SH} for all oriented $\E\in \logDA(S,\Lambda)$ there is a fiber sequence
\begin{equation}\label{eq:htp-purity-oriented-spectra}
\Map(\Sigma^{\infty}\Lambda(Z),\Sigma^{r-d,s-d}\mathbf{E})\to \Map(\Sigma^{\infty}\Lambda(X),\Sigma^{r,s}\mathbf{E})\to \Map(\Sigma^{\infty}\Lambda(\Bl_Z(X)),\Sigma^{r,s}\mathbf{E}).
\end{equation}
Moreover, if $\E\in \logFDA(S,\Lambda)$ is oriented, let $\mathrm{Comp}^*$ be the right adjoint of \eqref{eq:DAtoFDA-stab}: for all $Y\in \SmlSm_R$ and all $r,s$ we have an equivalence \[
    \Map_{\logDA(R,\Lambda)}(\Sigma^{\infty}\Lambda(Y),\Sigma^{r,s}\mathrm{Comp}^*(-)) \simeq  \Map_{\logFDA(R,\Lambda)}(\Sigma^{\infty}\Lambda(Y_p^\wedge),\Sigma^{r,s}(-)),
    \]
    by adjunction. In particular, $\mathrm{Comp}^*$ preserves oriented objects. This implies that for all $X$ and $Z$ as above, we have a fiber sequence
\begin{equation}\label{eq:htp-purity-oriented-spectra-formal}
\Map(\Sigma^{\infty}\Lambda(Z_p^\wedge),\Sigma^{r-d,s-d}\mathbf{E})\to \Map(\Sigma^{\infty}\Lambda(X_p^\wedge),\Sigma^{r,s}\mathbf{E})\to \Map(\Sigma^{\infty}\Lambda(\Bl_Z(X)_p^\wedge),\Sigma^{r,s}\mathbf{E}).
\end{equation}
\end{rmk}

\begin{rmk}
Let $X$ be a smooth scheme over $S$ and let $Z\subseteq X$ be a closed subscheme of  pure codimension $d$ such that $Z\to X\to S$ is smooth. Let $X' = \Bl_Z(X)$ and $Z'=Z\times_X X'$. Then by \cite[Theorem 7.3.3]{BPO}, for all $\E\in \logDA(S,\Lambda)$ there is a cartesian square, functorial in $(X,Z)$,
    \begin{equation}\label{eq:sm-blow-up-spectra}
        \begin{tikzcd}
        \Map(\Sigma^{\infty}\Lambda(X),\Sigma^{r,s}\mathbf{E}) \ar[r]\ar[d]& \Map(\Sigma^{\infty}\Lambda(Z),\Sigma^{r,s}\mathbf{E})\ar[d] \\
        \Map(\Sigma^{\infty}\Lambda(X'),\Sigma^{r,s}\mathbf{E}) \ar[r] &\Map(\Sigma^{\infty}\Lambda(Z'),\Sigma^{r,s} \mathbf{E}).
        \end{tikzcd}
    \end{equation}
    Let $X''$ be the blow-up of $X\times \P^1$ in $Z\times 0$, $E$ be the log structure on $X''$ induced by the exceptional divisor and let $Z''=(Z\times 0)\times_{(X\times \P^1)} X''$.
By \cite[Theorem 7.4.2]{BPO}, there is a cartesian square, functorial in $(X,Z)$\[
\begin{tikzcd}
        \Map(\Sigma^{\infty}\Lambda(X\times (\P^1,\infty)),\Sigma^{r,s}\mathbf{E}) \ar[r]\ar[d]& \Map(\Sigma^{\infty}\Lambda(Z\times 0),\Sigma^{r,s}\mathbf{E})\ar[d] \\
        \Map(\Sigma^{\infty}\Lambda(X'',E),\Sigma^{r,s}\mathbf{E}) \ar[r] &\Map(\Sigma^{\infty}\Lambda(Z''),\Sigma^{r,s}\mathbf{E}).
        \end{tikzcd}
\] 
By an argument completely analogous to \cite[Corollary 15.13]{MVW}, the (homotopic) equivalence \[
1_X\times 0\sim 1_X\times 1\colon \Sigma^{\infty}\Lambda(X)\simeq \Sigma^{\infty}\Lambda(X\times(\P^1,\infty))
\] gives a splitting of the map\[
\Map(\Sigma^{\infty}\Lambda(Z\times 0),\Sigma^{r,s}\mathbf{E})\oplus \Map(\Sigma^{\infty}\Lambda(X''),\Sigma^{r,s}\mathbf{E}) \to  \Map(\Sigma^{\infty}\Lambda(Z''),\Sigma^{r-2i,s-i}\mathbf{E}).
\] 
For all oriented $\E\in \logDA(S,\Lambda)$, the map \[
\Map(\Sigma^{\infty}\Lambda(Z''),\Sigma^{r-2i,s-i}\mathbf{E})\to \Map(\Sigma^{\infty}\Lambda(Z'),\Sigma^{r-2i,s-i}\mathbf{E})
\]
is a split epimorphism, since it comes from the embedding $\P(N_{Z\subseteq X}) \to \P(N_{Z\subseteq X}\oplus \cO)$: combining everything we get that the fiber sequence induced by \eqref{eq:sm-blow-up-spectra} splits, giving an equivalence:
 \begin{equation}\label{eq:sm-blow-up-oriented-spectra}
\Map(\Sigma^{\infty}\Lambda(X),\Sigma^{r,s}\mathbf{E})\oplus (\oplus_{0<i<d}\Map(\Sigma^{\infty}\Lambda(Z),\Sigma^{r-2i,s-i}\mathbf{E}))\simeq \Map(\Sigma^{\infty}\Lambda(X'),\Sigma^{r,s}\mathbf{E}).
\end{equation}
As before, via the adjunction \eqref{eq:DAtoFDA-stab} we deduce that for all $\E\in \logFDA(S,\Lambda)$ there is a cartesian square, functorial in $(X,Z)$:
    \begin{equation}\label{eq:sm-blow-up-spectra-formal}
        \begin{tikzcd}
    \Map(\Sigma^{\infty}\Lambda(X_p^\wedge),\Sigma^{r,s}\mathbf{E}) \ar[r]\ar[d]& \Map(\Sigma^{\infty}\Lambda(Z_p^\wedge),\Sigma^{r,s}\mathbf{E})\ar[d] \\
\Map(\Sigma^{\infty}\Lambda((X')_p^\wedge),\Sigma^{r,s}\mathbf{E}) \ar[r] &\oplus_{i=1}^{d-1}\Map(\Sigma^{\infty}\Lambda((Z')_p^\wedge),\Sigma^{r,s}\mathbf{E})
        \end{tikzcd}
    \end{equation}
    and for $\E$ oriented we have an equivalence
    \begin{equation}\label{eq:sm-blow-up-oriented-spectra-formal}
\Map(\Sigma^{\infty}\Lambda(X_p^\wedge),\Sigma^{r,s}\mathbf{E})\oplus (\oplus_{0<i<d}\Map(\Sigma^{\infty}\Lambda(Z_p^\wedge),\Sigma^{r-2i,s-i}\mathbf{E}))\simeq \Map(\Sigma^{\infty}\Lambda((X')_p^\wedge),\Sigma^{r,s}\mathbf{E}).
\end{equation}
\end{rmk}

\subsection{First examples: de Rham and crystalline  motivic spectra}\label{sec:deRham_Hodge_Crys_spectra}

A strict \'etale sheaf on the category of pre-log rings does not automatically produce a strict \'etale sheaf on the category of log schemes unlike the case of schemes since the functor $\Spec$ from the category of pre-log rings to the category of log schemes is not fully faithful.
Gabber \cite[\S 8]{Olsson} used pre-log structures in a topos to define the cotangent complexes for schemes without experiencing the globalization problem,
see also \cite[Remark 2.12]{Koshikawa-Yao}.
However, it is not clear how to adapt such an argument to the various other sheaves we will consider on the category of pre-log rings.
By inspecting \cite[Definition 4.9]{Koshikawa-Yao} and \cite[Proof of Proposition 8.3.8]{BPO-SH},
we have the globalization procedure as follows.

\begin{defn}
\label{df:etale globalization}
Let $\cC$ be an $\infty$-category with colimits and limits.
We have the adjoint functors
\[
\Gamma
:
\mathrm{lSch}
\rightleftarrows
\mathrm{PreLog}^\op
:
\Spec
\]
where $\Gamma(X):=(\Gamma(X,\cO_X),\Gamma(X,\cM_X))$ for $X\in \mathrm{lSch}$.
These induce the adjoint functors
\[
\Gamma^*
:
\PSh(\mathrm{PreLog}^\op,\cC)
\rightleftarrows
\PSh(\mathrm{lSch},\cC)
:
\Gamma_*
\]
such that
$\Gamma^* \cF(X):=\cF(\Gamma(X))$ and $\Gamma_*\cG(A,M):=\cG(\Spec(A,M))$  for every $(A,M)\in \mathrm{PreLog}$, $X\in \mathrm{lSch}$, $\cF\in \PSh(\mathrm{PreLog}^\op,\cC)$, and $\cG\in \PSh(\mathrm{lSch},\cC)$.
Consider also the adjoint functors
\[
L_\set
:
\PSh(\mathrm{lSch},\cC)
\rightleftarrows
\Sh_\set(\mathrm{lSch},\cC)
:
\iota,
\]
where $L_\set$ is the sheafification functor, and $\iota$ is the inclusion.
A presheaf $\cF\in \PSh(\mathrm{PreLog}^\op,\cC)$ is \emph{globalizable} if the unit morphism
\[
\alpha_\cF \colon \cF\to \Gamma_*\iota L_\set \Gamma^*\cF
\]
is an equivalence.
\end{defn}

We provide a general criterion of globalizablity as follows.

\begin{prop}
\label{prop:globalization}
With the above notation,
assume that $\cF$ satisfies the following conditions:
\begin{enumerate}
\item[(i)] $\cF$ preserves filtered colimits as a functor $\mathrm{PreLog}\to \cC$.
\item[(ii)] For every pre-log ring $(A,M)$,
the presheaf
\[
(A\to B)\mapsto \cF(B,M)
\]
on the opposite category of $A$-algebras is an \'etale sheaf.
\item[(iii)] For every pre-log ring $(A,M)$ such that $A$ is strictly local, the induced morphism $\cF(A,M)\to \cF(A,M^a)$ is an isomorphism,
i.e., we have the ``logification invariance'' when the underlying ring is strictly local.
\end{enumerate}
Then $\cF$ is globalizable.
\end{prop}
\begin{proof}
By e.g., \cite[\href{https://kerodon.net/tag/01DK}{Tag 01DK}]{kerodon}, we need to show that the induced morphism $\cF(A,M)\to L_\set\Gamma_*\cF(\Spec(A,M))$ is an equivalence for every pre-log ring $(A,M)$.
Fixing $M$, the natural transformation $\alpha_\cF$ determines a morphism from the presheaf
\begin{equation}
\label{eq:left-hand in lem}
(A\to B)\mapsto \cF(B,M)
\end{equation}
to the presheaf
\begin{equation}
\label{eq:right-hand in lem}
(A\to B) \mapsto L_\set \Gamma_*\cF(\Spec(B,M))
\end{equation}
on the opposite category of $A$-algebras.
The first one is an \'etale sheaf by (ii),
and the second one is an \'etale sheaf by construction.
Hence it suffices to show that every stalk of this morphism of sheaves is an equivalence.

Let $\ol{x}$ be a geometric point of $\Spec(A)$.
By (i), the stalk of \eqref{eq:left-hand in lem} at $\ol{x}$ is $\cF(A_{\ol{x}},M)$.
On the other hand,
the stalk of the sheafification is the same as the original one.
Indeed, if $i\colon \ol{x}\to \Spec(A)$ is the structure morphism, then the stalk of a presheaf $\cG$ on $\Spec(A)_{\et}$ is $i^*\cG(\ol{x})$.
Hence to show $i^*\cG(\ol{x})\simeq i^*L_{\et}\cG(\ol{x})$, by taking right adjoints, it suffices to show $\iota i_* \cH\simeq i_*\cH$ for every constant sheaf $\cH$ on $\ol{x}_{\et}$, i.e., $i_*\cH$ is a sheaf.
This follows from the fact that every \v{C}ech nerve of a covering in $\Spec(A)_{\et}$ pulls back to the \v{C}ech nerve of the corresponding covering in $\ol{x}_{\et}$.
Note also that the category of fs log schemes $Y$ strict over $\Spec(A,M)$ such that $\ul{Y}$ is affine is equivalent to the opposite category of $A$-algebras, and if $\ul{Y}=\Spec(B)$, then $Y\cong \Spec(B,M)$.
Together with (i) again, we see that the stalk of \eqref{eq:right-hand in lem} at $\ol{x}$ is $\cF(A_{\ol{x}},\cM_{\Spec(A,M),\ol{x}})$.
Hence it suffices to show that the induced morphism
\[
\cF(A_{\ol{x}},M)
\to
\cF(A_{\ol{x}},\cM_{\Spec(A,M),\ol{x}})
\]
is an equivalence.
This follows from (iii) since the stalk $\cM_{\Spec(A,M),\ol{x}}$ is the logification $M^a$ of $M$.
\end{proof}

Let $R$ be a $p$-complete ring.
Let $\widehat{\cD\cF}(R)_p^\wedge$ be the full subcategory of $\widehat{\cD\cF}(R)$ spanned by $p$-complete objects.
Consider the Hodge-completed derived log de Rham complex \[
\widehat{L\Omega}_{-/R}\colon \mathrm{PreLog}_R \to \widehat{\cD\cF}(R)_p^\wedge
\]
equipped with the Hodge filtration with graded pieces $(\bigwedge^i \L_{-/R})_p^\wedge$.
The completion functors $\cD\cF(R)\to \widehat{\cD\cF}(R)$ and $\widehat{\cD\cF}(R)\to \widehat{\cD\cF}(R)_p^\wedge$ preserve colimits since it is a left adjoint, so $\widehat{L\Omega}_{-/R}$ preserves sifted colimits. As in \cite[Example 5.11]{BMS2}, $\widehat{L\Omega}_{-/R}$ satisfies strict \'etale descent. Furthermore, $L\Omega_{-/R}$ satisfies logification invariance, see e.g.\ \cite[Lemma 3.11]{BLPO-HKR}. Hence $\widehat{L\Omega}_{-/R}$ satisfies the conditions (i)--(iii) in Proposition \ref{prop:globalization},
so we can globalize $\widehat{L\Omega}_{-/R}$ to formal $R$-schemes.
Note that we do \emph{not} apply Proposition \ref{prop:globalization} to the functor $\widehat{L\Omega}_{-/R}\colon \mathrm{PreLog}_R \to \cD(R)$ obtained by forgetting the filtration structure.
Let $S$ be a $p$-completed $R$-algebra and consider the strict \'etale sheaf \[
\widehat{L\Omega}_{-/R}\colon \FlQSm_S\to \cD(R)\quad \kX\mapsto R\Gamma(\kX,\widehat{L\Omega}_{-/R})
\]
by forgetting the filtration structure.
Similarly, we consider the strict \'etale sheaves 
\begin{equation}\label{eq:cotan_spt}
L\Omega^i_{-/R} \colon \FlQSm_S\to \cD(R)\quad \kX\mapsto R\Gamma(\kX,(\bigwedge^i \L_{-/R})_p^\wedge);
\end{equation}
for $i \geq 0$. The exterior product of differential forms can be used to assemble \eqref{eq:cotan_spt} into a graded commutative monoid $\{L\Omega^i_{-/R}\}_{i\in \N}$ in $\Sh_{\set}(\FlQSm_R,\cD(R))$. Similarly, the cdga structure of the log de Rham complex and the multiplicativity of the Hodge filtration induce a structure of a (constant) graded commutative monoid on $\{\widehat{L\Omega}_{-/R}\}_i$ (see e.g. \cite[4.1]{Bhattcompletion}).

Let now $R$ be a perfectoid or $R=\Z_p$: we can consider the $p$-completed derived log de Rham complex equipped with the conjugate filtration:\[
L\Omega_{-/R}\colon \mathrm{PreLog}_R \to \cD\cF(R).
\]
With an argument completely analogous to \cite[Example 5.12]{BMS2}, it is a log quasisyntomic sheaf.
\begin{prop}\label{prop:spectra-hodge-dR}
\begin{enumerate}
    \item There are oriented $\E^{\rm Hdg},\E^{\rm \widehat{dR}}\in \CAlg(\logFDA(S,R))$ such that for all $\kX\in \FlQSm_S$:
\begin{align*}
    \Map(\Sigma^{\infty}(\kX),\Sigma^{p,q}\E^{\rm Hdg})&\simeq R\Gamma(\kX,L\Omega^q_{-/R})[p],\\
    \Map(\Sigma^{\infty}(\kX),\Sigma^{p,q}\E^{\rm \widehat{dR}})&\simeq R\Gamma(\kX,\widehat{L\Omega}_{-/R})[p].
\end{align*}
\item Let $R$ be perfectoid or $R=\Z_p$, and $S\in \QSyn_R$. Then there is an oriented homotopy commutative monoid $\E^{\rm dR}\in \CAlg(\logFDA(S,R))$ such that\[
\Map(\Sigma^{\infty}(\kX),\Sigma^{p,q}\E^{\rm {dR}})\simeq R\Gamma(\kX,{L\Omega}_{-/R})[p]
\]
\end{enumerate}
    \begin{proof} We begin by showing that $(L\Omega^i_{-/R})_{i\in \N}$ is in fact a graded commutative monoid in $\logFDAeff(S,R)$.
By strict \'etale descent it is enough to show that for any pre-log $R$-algebra $(A,M)$, the projection induces an equivalence
\begin{equation}\label{eq:box-inv-cotan}
R\Gamma_{\Zar}(A, \L_{(A,M)/R})\simeq R\Gamma_{\Zar}(\P^n_{A},  \L_{(\P^n_{(A,M)},\P^{n-1}_{(A,M)})/R}).    
\end{equation}
By left Kan extension from the case of finite free pre-log algebras, this follows from \cite[Proposition 8.3]{BLPO-HKR}.
    Next,  for $\kX\in \FlQSm_S$ and $\cE\to \kX$ a vector bundle of rank $r+1$, consider the map induced by the first Chern class:
    \begin{equation}\label{eq:derived-chern-de-rham}
\bigoplus_{i=0}^rR\Gamma(\kX,L\Omega^{j-i}_{\kX/R})[-2i]\to R\Gamma(\P(\cE),L\Omega^j_{\P(\cE)/R}).
    \end{equation}
By choosing a trivializing cover and by an argument completely analogous to \cite[B.8]{BhattLurie} we have an equivalence \[
R\Gamma(\P^n_{\kX},L\Omega^j_{\P^n_{\kX}/R})\simeq R\Gamma(\P^n_R,\Omega^j_{\P^n_R/R})\otimes_R^L R\Gamma(\kX,L\Omega^j_{\kX/R})
\]
compatible with the Chern classes, so the classical computation of the Hodge cohomology of projective spaces \cite[\href{https://stacks.math.columbia.edu/tag/0FMI}{Tag 0FMI}]{stacks-project} implies that the map \eqref{eq:derived-chern-de-rham} is an equivalence. We can then apply Proposition \ref{prop:build-spectra} to obtain an oriented spectrum $\E^{\rm Hdg}$. 

We now pass to derived de Rham cohomology. By passing to the graded pieces of the filtration, we see from the above computation that $\widehat{L\Omega}_{-/R}$ is $(\P^n,\P^{n-1})$-invariant, so $\widehat{L\Omega}_{-/R}$ is a (constant) commutative monoid in $\logFDAeff(S,R)$. The first Chern class again induces a map
\begin{equation}\label{eq:pbf-de-Rham}
\bigoplus_{i=0}^rR\Gamma(\kX,\widehat{L\Omega}_{-/R})[-2i]\to R\Gamma(\P(\cE),\widehat{L\Omega}_{-/R})
\end{equation}
which is an equivalence by passing to the graded pieces and noting that \eqref{eq:derived-chern-de-rham} is an equivalence. The spectrum $\E^{\rm \widehat{dR}}$ is then assembled using again Proposition \ref{prop:build-spectra}. This proves $(1)$.

Let now $R$ be a perfectoid or $R=\Z_p$. By induction on $n$ and \eqref{eq:box-inv-cotan}, we have similarly to \cite[Example 5.12]{BMS2} that for all $(A,M)\in \lQSyn_R$, and all $m$, \[
R\Gamma_{\Zar}(A,\Fil_m^{\rm conj}L\Omega_{(A,M)/R}\otimes^L_{\Z}\Z/p\Z) \simeq R\Gamma_{\Zar}(\P^n_A,\Fil_m^{\rm conj}L\Omega_{(\P^n_{(A,M)},\P^{n-1}_{(A,M)})}\otimes^L_{\Z}\Z/p\Z),
\]
and since $(A,M)\in \lQSyn_R$, by Lemma \ref{lem:cotimes-qsyn}, we have that $\Fil_m^{\rm conj}L\Omega_{(A,M)/R}\otimes^L_{\Z}\Z/p\Z$ takes value in $\cD^{\geq -1}$ for all $n$, as each $\L_{(A,M)/R}$ has $p$-completed Tor amplitude $[-1,0]$: this implies that cohomology commutes with filtered colimits so we conclude that
\begin{align*}
    R\Gamma_{\Zar}(A,L\Omega_{(A,M)/R}\otimes^L_{\Z}\Z/p\Z)&\simeq \colim_m R\Gamma_{\Zar}(A,\Fil_m^{\rm conj}L\Omega_{(A,M)/R}\otimes^L_{\Z}\Z/p\Z) \\
    &\simeq \colim_m R\Gamma_{\Zar}(\P^n_A,\Fil_m^{\rm conj}L\Omega_{(\P^n_{(A,M)},\P^{n-1}_{(A,M)})}\otimes^L_{\Z}\Z/p\Z)\\
    &\simeq R\Gamma_{\Zar}(\P^n_A,L\Omega_{(\P^n_{(A,M)},\P^{n-1}_{(A,M)})}\otimes^L_{\Z}\Z/p\Z)
\end{align*} 
Again $L\Omega_{/R}$ is a (constant) commutative monoid and the first Chern class induces an equivalence analogous to \eqref{eq:pbf-de-Rham} again by considering $\Fil_m^{\rm conj}$ and passing to the filtered colimit. This proves $(2)$.
    \end{proof}
\end{prop}

Let $k$ be a perfect field of characteristic $p$. For $X\in \lSm_k$, and $m, n \geq 0$,  let $W_m\Omega^n_{X/k}$ be the log de Rham--Witt sheaf of \cite{matsuue} and let $W_m\Omega^\bullet_{X/k}$ the $m$-truncated log de Rham--Witt complex. 
\begin{prop}
There are oriented \[\E^{m,\rm dRW},\E^{m, \rm crys}, \E^{\rm crys}\in \CAlg(\logDA(k,W(k)))\] such that for all $X\in \SmlSm_k$: 
\begin{align*}
    \Map(\Sigma^{\infty}(X),\Sigma^{r,s}\E^{m,\rm dRW}) &\simeq R\Gamma(X,W_m\Omega^s_{X/k})[r]\\
    \Map(\Sigma^{\infty}(X),\Sigma^{p,q}\E^{m,\rm crys})&\simeq R\Gamma_{\rm crys}(X/W_m(k))[p]\\
    \Map(\Sigma^{\infty}(X),\Sigma^{p,q}\E^{\rm crys})&\simeq R\Gamma_{\rm crys}(X/W(k))[p]
\end{align*}
\end{prop}
\begin{proof}
Recall from \cite[Theorem 1.3]{mericicrys} that for all $m$ there is $W_m\Omega^s\in \logDAeff(k,W(k))$ representing the cohomology of $W_m\Omega^s{\_/k}$. By considering the graded commutative monoid $\{W_m\Omega^i\}_{i\in \N}$, the crystalline Chern class of \cite{Gros1985} induces a section $c_{m}\colon \Lambda(\P^1)\to W_m\Omega^1[2]$ that induces, for all $X\in \SmlSm_k$, an equivalence (see \cite[(4.5.1)]{mericicrys})\[
    R\Gamma(X,W_m\Omega^i) \simeq \widetilde{R\Gamma}(X\times \P^1,W_m\Omega^{i+1})[2]
\]
By applying Proposition \ref{prop:build-spectra}, we get oriented ring spectra $\E^{m,\rm dRW}\in \CAlg(\logDA(k,W(k)))$.
By considering the $m$-truncated de Rham--Witt complex, again by \cite[Theorem 1.3]{mericicrys}, we get again a constant graded commutative monoid $\{W_m \Omega\}$ in $\logDAeff(k,W(k))$ producing oriented ring spectra $\E^{m,\rm crys}\in \CAlg(\logDA(k,W(k)))$ with, for all $m_1\geq m_2$, maps of oriented motivic spectra $\E^{m_1,\rm crys}\to \E^{m_2,\rm crys}$ 
in $\logDA(k,W(k))$ by Proposition \ref{prop:build-spectra-maps}. Then $\E^{\rm crys}:=\lim_m \E^{m,\rm crys}$.
\end{proof}

If $F\colon W(k)\to W(k)$ is the Frobenius on $W(k)$, we let $F^*\E^{\rm crys}$ be the oriented ring spectrum with $W(k)$-module structure twisted by $F$.

\begin{rmk}
    As proven in \cite{mericicrys}, $\E^{m,\rm dRW}$ and $\E^{\rm crys}$ are ring spectra in $\logDM(k,W(k))$. This will suggest, via the crystalline comparison, that the prismatic and syntomic cohomology are presheaves with logarithmic transfers (in an appropriate sense): this will be investigated in future work.
\end{rmk}

\section{Prismatic and syntomic realizations}\label{sec:prism_synt}

We recall from \cite[Definition 4.5]{BLPO-BMS2}
or from \cite[Section 3.1]{Koshikawa-Yao} (which is simply a translation to the log setting of \cite[Definition 4.10(3)]{BMS2}) the following definition:

\begin{defn}
\label{df:log quasisyntomic}
 A pre-log ring $(A,M)$ is called \emph{log quasisyntomic} if  $A$ is $p$-complete with bounded $p^\infty$ torsion, $p$-completely flat over $\Z_p$ and satisfies the property that $\L_{(A, M_A)/\Z_p}$ has $p$-complete Tor amplitude in degree $[-1,0]$. 
 We denote by $\lQSyn$ the category of pre-log rings that are log quasisyntomic, and for any ring $R$, we let $\lQSyn_R$ denote the category of pre-log $R$-algebras which are log quasisyntomic.
 Let $\kX$ be a quasi-coherent bounded $p$-adic formal log scheme. Then $\kX$ is \emph{log quasisyntomic} if
strict \'etale locally it is isomorphic to $\Spf(A,M)$ with $(A,M)\in \lQSyn$.
We will denote $\FlQSyn$ the category of bounded $p$-adic log quasisyntomic formal log schemes.

The \emph{strict quasisyntomic topology on $\FlQSyn$} is generated by the families of strict morphisms $\{\mathfrak{U}_i\to \mathfrak{X}\}_{i\in I}$ such that $\{\ul{\mathfrak{U}_i}\to \ul{\mathfrak{X}}\}_{i\in I}$ is a quasisyntomic covering.
 \end{defn}
 
 We also recall the following notion:
 
\begin{defn}
An integral pre-log ring $(S,Q)$ is called log quasiregular
semiperfectoid if it is log quasisyntomic, there exists a map $R\to S$ with $R$ perfectoid, and $S/pS$ and $Q$ are semiperfect (in the sense of Definition \ref{def:perf_semiperf_mon}). 
\end{defn}

\begin{lemma}
\label{lem:formal vs ring}
Let $X$ be a $p$-adic formal log scheme,
and let $(A,M)$ be a $p$-complete pre-log ring.
Then there is a natural isomorphism
\begin{equation}
\label{eq:formal vs ring}
\Hom(X,\Spf(A,M))
\cong
\Hom((A,M),(\Gamma(X,\cO_X),\Gamma(X,\cM_X))).
\end{equation}
\end{lemma}
\begin{proof}
Since $\Spf(A,M)\cong \Spf(A)\times_{\Spf(\Z_p)}\Spf(\Z_p\langle M \rangle,M)$,
the left-hand side of \eqref{eq:formal vs ring} is isomorphic to
\[
\Hom(X,\Spf(A))\times_{\Hom(X,\Spf(\Z_p\langle M \rangle))}\Hom(X,\Spf(\Z_p\langle M \rangle,M)).
\]
The right-hand side of \eqref{eq:formal vs ring} is isomorphic to
\[
\Hom(A,\Gamma(X,\cO_X))\times_{\Hom(M,\Gamma(X,\cO_X))}\Hom(M,\Gamma(X,\cM_X)).
\]
Since $\Hom(X,\Spf(A))\cong \Hom(A,\Gamma(X,\cO_X))$ and $\Hom(X,\Spf(\Z_p\langle M \rangle))\cong \Hom(M,\Gamma(X,\cO_X))$,
it suffices to show
\[
\Hom(X,\Spf(\Z_p\langle M \rangle,M))
\cong
\Hom(M,\Gamma(X,\cM_X)).
\]
This can be shown arguing as in \cite[Proposition III.1.2.9]{ogu}.
\end{proof}

\begin{defn}
\label{df:QSyn globalization}
Let $\cC$ be an $\infty$-category with colimits and limits.
By Lemma \ref{lem:formal vs ring},
we have the adjoint functors
\[
\Gamma
:
\mathrm{FlQSyn}
\rightleftarrows
\mathrm{lQSyn}^\op
:
\Spf
\]
where $\Gamma(X):=(\Gamma(X,\cO_X),\Gamma(X,\cM_X))$ for $X\in \mathrm{FlQSyn}$.
We have the induced adjoint functors
\[
\Gamma^*
:
\PSh(\mathrm{lQSyn}^\op,\cC)
\rightleftarrows
\PSh(\mathrm{FlQSyn},\cC)
:
\Gamma_*,
\]
Let $\sqsyn$ be the shorthand for the strict quasisyntomic topology.
Then we have the adjoint functors
\[
L_\sqsyn
:
\PSh(\mathrm{FlQSyn},\cC)
\rightleftarrows
\Sh_\sqsyn(\mathrm{FlQSyn},\cC)
:
\iota,
\]
A presheaf $\cF\in \PSh(\mathrm{lQSyn}^\op,\cC)$ is \emph{globalizable} if the unit morphism
\[
\alpha_\cF \colon \cF\to \Gamma_*\iota L_\sqsyn \Gamma^*\cF
\]
is an equivalence.
\end{defn}

\begin{prop}
\label{prop:QSyn globalization}
With the above notation,
assume that $\cF$ satisfies the following conditions:
\begin{enumerate}
\item[(i)] The restriction of $\cF$ to $\lQSyn_{\Z_p^\mathrm{cyc}}$ with $\Z_p^\mathrm{cyc}:=\Z_p[\mu_{p^\infty}]_p^\wedge$ preserves filtered colimits as a functor $\lQSyn_{\Z_p^\mathrm{cyc}}\to \cC$.
\item[(ii)] For every $(A,M)\in \lQSyn$,
the presheaf
\[
(A\to B)\mapsto \cF(B,M)
\]
on the opposite category of quasisyntomic $A$-algebras is a quasisyntomic sheaf.
\item[(iii)] For every $(A,M)\in \lQSyn_{\Z_p^\mathrm{cyc}}$ such that $A$ is strictly local, the induced morphism $\cF(A,M)\to \cF(A,M^a)$ is an isomorphism.
\end{enumerate}
Then $\cF$ is globalizable.
\end{prop}
\begin{proof}
We need to show that the induced morphism $\cF(A,M)\to L_\sqsyn \Gamma^* \cF(\Spf(A,M))$ is an equivalence for every $(A,M)\in \lQSyn$.
Both sides satisfy quasisyntomic descent on $A$.
Since $\Z_p^\mathrm{cyc}\in \QSyn$ and $\Z_p\to \Z_p^\mathrm{cyc}$ is $p$-completely faithfully flat, $\Z_p\to \Z_p^\mathrm{cyc}$ is a quasisyntomic cover.
Use the \v{C}ech nerve of the induced quasisyntomic cover $A\to A\widehat{\otimes}_{\Z_p}\Z_p^\mathrm{cyc}$ to reduce to the case where $A$ is a $\Z_p^\mathrm{cyc}$-algebra.

We only need to show that
the morphism from the quasisyntomic sheaf
\[
(A\to B)\mapsto \cF(B,M)
\]
to the quasisyntomic sheaf
\[
(A\to B) \mapsto L_\sqsyn \Gamma_*\cF(\Spf(B,M))
\]
on the opposite category of quasisyntomic $A$-algebras is an equivalence.
The quasisyntomic pretopology consists of the families $\{A\to B_i\}_{i\in I}$ in $\QSyn$ with finite $I$ such that $A\to \oplus_{i\in I} B_i$ is a quasisyntomic cover.
Due to the finiteness of the index set $I$,
the quasisyntomic topos is coherent.
Hence by Deligne's theorem \cite[Proposition VI.9.0]{SGA4}, the quasisyntomic topology has enough points.
Arguing as in Proposition \ref{prop:globalization},
we reduce to the case where $A$ is local in the quasisyntomic topology. In this case, $A$ is strictly local since the \'etale topology is coarser than the quasisyntomic topology, so (iii) finishes the proof.
\end{proof}

\begin{rmk}
\label{rmk:QSyn globalization}
We often apply Proposition \ref{prop:QSyn globalization} to the case where $\cC=\widehat{\cD\cF}(R)_p^\wedge$ for some ring $R$.
Together with \cite[Lemma 5.2(1),(3)]{BMS2}, we see that $\cF\in \PSh(\lQSyn,\widehat{\cD\cF}(R)_p^\wedge)$ is globalizable if the graded pieces $\mathrm{gr}^i\cF$ satisfy the conditions (i)--(iii) in Proposition \ref{prop:QSyn globalization} for all integers $i$.
\end{rmk}

Recall from \cite[\S 7.3]{BLPO-BMS2} that, analogously to \cite[\S 7]{BMS2}, the Nygaard-complete absolute log prismatic cohomology 
$\cPrism_{-}:= R\Gamma_{\rm lqsyn}(-,\pi_{0}\TP(-)_p^\wedge)$  equipped with the Nygaard filtration define functors
\begin{align*}
&\cPrism_{-}\in \Fun(\lQSyn,\cD(\Z_p)), &\Fil^{\geq \bullet}_\rN\cPrism_{-}
\in \Fun(\lQSyn, \widehat{\cD\cF}(\Z_p)_p^\wedge)).
\end{align*}
These complexes commute with (homotopy) limits and satisfy quasisyntomic descent (this is a consequence of \cite[Theorem 2.3]{BLPO-BMS2} and the construction).
Define the Breuil--Kisin twists 
\begin{align*}
    \cPrism_{-}\{1\}&:= R\Gamma_{\rm lqsyn}(-,\pi_{2}\TP(-)_p^\wedge)[-2]\\
    \cPrism_{-}\{-1\}&:= R\Gamma_{\rm lqsyn}(-,\pi_{-2}\TP(-)_p^\wedge)[2]
\end{align*}
with the Nygaard filtrations given by unfolding the double-speed Postnikov filtration.  
In analogy with \cite[Theorem 1.12 (3)]{BMS2}, we get the twisted prismatic cohomology $\cPrism_{-}\{i\}$ by taking tensor powers in $\Fun(\lQSyn,\widehat{\cD\cF}(\Z_p)_p^\wedge)$, with the induced filtration.

As observed in \cite[Lemma 7.14]{BMS2}, the object $\cPrism_{(A,M)}\{1\}$ is not invertible as an $\cPrism_{(A,M)}$-module, but it is invertible when considered a module over the \emph{filtered ring} $\cPrism_{(A,M)}$ (the proof for log rings works verbatim). In particular, in the filtered category $\widehat{\cD\cF}(\Z_p)_p^\wedge)$, we have an equivalence
\begin{equation*}
    \Fil^{\geq \bullet}_\rN\cPrism_{-}\{i\}\otimes\Fil^{\geq \bullet}_\rN\cPrism_{-}\{j\} \to \Fil^{\geq \bullet}_\rN\cPrism_{-}\{i+j\}
\end{equation*}
of filtered objects for all $i,j\in \Z$, where the tensor product is taken in $\Fun(\lQSyn\widehat{\cD\cF}(\Z_p)_p^\wedge)$. By definition of the convolution product, we have   \[
\Fil^m(\Fil^{\geq \bullet}_\rN\cPrism_{-}\{i\}\otimes\Fil^{\geq \bullet}_\rN\cPrism_{-}\{j\})\simeq \bigoplus_{m_1 + m_2=m} \Fil^{\geq m_1}_\rN\cPrism_{-}\{i\}\otimes_{\Z_p}^L\Fil_\rN^{\geq m_2}\cPrism_{-}\{j\}.
\] 
This gives associative and commutative multiplication maps in $\Fun(\lQSyn,\cD(\Z_p))$
\begin{equation}\label{eq:mult}
\Fil^{\geq m}_\rN\cPrism_{-}\{i\}\otimes^L_{\Z_p}\Fil^{\geq n}_\rN\cPrism_{-}\{j\} \to \Fil^{\geq m+n}_\rN\cPrism_{-}\{i+j\}    
\end{equation} for all $i,j,m,n$.

\begin{prop}
\label{prop:prism globalizable}
The above functor
$\Fil^{\geq \bullet}_\rN\cPrism_{-}$ is globalizable.
\end{prop}
\begin{proof}
Observe that $\Fil^{\geq \bullet}_\rN\cPrism_{-}$ is a log quasisyntomic sheaf and hence strict quasisyntomic sheaf.
By \cite[Proposition 7.4]{BLPO-BMS2} and Remark \ref{rmk:QSyn globalization},
it suffices to show that the sheaves
\[
(A,M)\mapsto (\wedge_A^i \L_{(A,M)/\Z_p^\mathrm{cyc}})_p^\wedge
\]
on $\lQSyn_{\Z_p^\mathrm{cyc}}$ with values in the full subcategory $\cD(\Z_p)_p^\wedge$ of $\cD(\Z_p)$ spanned by the $p$-complete objects satisfy the conditions (i) and (iii) in Proposition \ref{prop:QSyn globalization} for all integers $i$.
The completion functor $\cD(\Z_p)\to \cD(\Z_p)_p^\wedge$ preserves colimits since it is left adjoint to the inclusion functor,
so it suffices to show that the sheaves
\[
(A,M)\mapsto \wedge_A^i \L_{(A,M)/\Z_p^\mathrm{cyc}}
\]
on $\lQSyn_{\Z_p^\mathrm{cyc}}$ with values in $\cD(\Z_p)$ satisfy the conditions (i) and (iii) in Proposition \ref{prop:QSyn globalization} for all integers $i$.
The condition (i) holds by \cite[Lemma 3.11]{BLPO-HKR}, and the condition (iii) holds since $\wedge_A^i \L_{(A,M)/\Z_p^\mathrm{cyc}}$ is left Kan extended from $\mathrm{Poly}_{(\Z_p^\mathrm{cyc},\{1\})}$ and hence preserves sifted colimits, see \cite[\S 2.10]{BLPO-HKR} for the notation $\mathrm{Poly}$.
\end{proof}
 
For every $i\in \Z$, we denote by $\Fil^{\geq \bullet}_\rN R\Gamma_{\cPrism}(-)\{i\}$ the object defined in $\Sh_{\set}(\FlQSyn,\widehat{\cD\cF}(\Z_p)_p^\wedge)$ by globalization,
and we denote by $R\Gamma_{\cPrism}(-)\{i\}$ the object defined in $\Sh_{\set}(\FlQSyn,\cD(\Z_p))$ after forgetting the filtration structure.

\begin{rmk}\label{rmk:prism-monoid}
    Since the filtrations obtained by unfolding from the double-speed Postnikov filtration of (logarithmic) $\TC^{-}$  and $\TP$ are functorial and multiplicative, we have that
\[
R\Gamma_{\cPrism}(-)\{i\}\simeq R\Gamma_{\rm lqsyn}(-,\pi_{2i}\TP(-)_p^\wedge)[-2i]
\]
and \[
\Fil^{\geq i}_\rN R\Gamma_{\cPrism}(-)\{i\}\simeq R\Gamma_{\rm lqsyn}(-,\pi_{2i}\TCmin(-)_p^\wedge)[-2i].
\]
The multiplication maps in \eqref{eq:mult} agree with the one induced by the $\E_\infty$-structure of $\TP$ and $\TCmin$, which also satisfy the identity axiom: this implies that the objects $\{R\Gamma_{\cPrism}(-)\{i\}\}_{i\in \Z}$ and $\{\Fil^{\geq i}_\rN R\Gamma_{\cPrism}(-)\{i\}\}_{i\in \Z}$ are graded commutative monoids in their respective categories.
\end{rmk}

\begin{rmk} 
The filtration $\Fil_\rN^{\geq \bullet}R\Gamma_{\cPrism}(-)$ is complete and exhaustive (in fact, constant in negative degrees). 
Indeed, since the category $\Sh_{\set}(\FlQSyn,\cD(\Z))$ is generated by $\Z(U)$ with $U=\Spf(A,M)$, it is enough to check that
\begin{align*}
    \Map(\Z(U),\lim_n\Fil_\rN^{\geq n}R\Gamma_{\cPrism}(-/R)) &= \lim_n \Map(\Z(U),\Fil_\rN^{\geq n}R\Gamma_{\cPrism}(-/R)) \\ &= \lim_n \Fil_\rN^{\geq n}\cPrism_{(A,M)/R} = 0
    \end{align*}
and, since the colimit is filtered and $\Z(U)$ is compact, for all $m\geq 0$ the map \[
     \Map(\Z(U),\Fil_\rN^{\geq -m}R\Gamma_{\cPrism}(-))\to \Map(\Z(U),\colim_n\Fil_\rN^{\geq -n}R\Gamma_{\cPrism}(-))
    \]
    agrees with the map\[
    \cPrism_{(A,M)}\to \colim_n \Fil_\rN^{\geq -n}\cPrism_{(A,M)},
    \]
    which is an equivalence (this can be checked $\lQSyn$-locally, reducing to the case of $\lQRSPerfd$, where it is obvious since the Postnikov filtration is exhaustive, so $\Fil_\rN^{\geq \bullet}R\Gamma_{\cPrism}(-)$ is eventually constant. 
\end{rmk}

\begin{rmk}\label{rmk:secondary_filtr} 
If $R$ is perfectoid ring, let $\Lambda:= \cPrism_R\cong A_{\rm inf}(R)$. By Proposition \ref{prop:prism globalizable}, we get an object of $\Sh_{\set}(\FlQSyn_R,\cD(\Lambda))$ (resp. in $\Sh_{\set}(\FlQSyn_R,\widehat{\cD\cF}(\Lambda))$) that we denote $R\Gamma_{\cPrism}(-/R)$ (resp. by $\Fil^{\geq \bullet}_\rN R\Gamma_{\cPrism}(-/R)$) In this case, we have that the graded pieces $\gr_N^i(R\Gamma_{\cPrism}))$ carry themselves a finite \emph{conjugate} filtration with graded pieces given by
    $\Map(\Z(U),\gr^j\gr_N^i(R\Gamma_{\cPrism}(-/R))) = (L\Omega^j_{(A,M)/R})_p^{\wedge}[-j]$ for $0 \le j \le i$. 
    \end{rmk}
\begin{rmk}
    The perfectoid base is useful for the following reason: the Breuil-Kisin  twist of $\cPrism_R$ is given as in \cite[\S 6.2]{BMS2} by $\cPrism_R\{1\}:=\pi_2\TP(R;\Z_p)\in \cD(\cPrism_R)$,  defined in terms of the (non-logarithmic) prismatic cohomology. Since $R$ is perfectoid, it is a free module of rank $1$ over $\cPrism_R$ generated by the element $\xi \in A_{\inf}(R)$. In this case, the $i$th twisted prismatic cohomology satisfies \[
    \cPrism_{-/R}\{i\}\simeq \cPrism_{-/R}\cotimes^L_{\cPrism_R}\cPrism_R\{1\}^{\cotimes i},\]
    Similarly, the twisted Nygaard filtration satisfies
    \[\Fil_\rN^{\geq i}\cPrism_{-/R}\{i\}\simeq \Fil_\rN^{\geq i}\cPrism_{-/R}\cotimes^L_{\cPrism_R}\cPrism_R\{1\}^{\cotimes i}.\] 
\end{rmk} 
\begin{rmk}
    We can interpret the filtered object $\Fil_\rN^{\geq \bullet}\cPrism_{-/R}\{i\}$ as the convolution product of the filtered objects $\Fil_\rN^{\geq \bullet}\cPrism_{-/R}$ and $\cPrism_R\{1\}^{\cotimes i}$, where the latter is equipped with the trivial filtration. With this convention, the object $\cPrism_{R/R}\{i\}$ comes equipped with its Nygaard filtration as in \cite[\S 6.2]{BMS2}, given by $\Fil^{\geq n}_\rN \cPrism_{R/R}\{i\} = \xi^n A_{\inf} (R) u^i$, where $u$ is a generator of $\pi_{2} \TCmin(R;\Z_p)$. This agrees with the filtration given by the convolution product on \[
    (\cPrism_{-/R}\{1\})^{\otimes_{\widehat{\cD\cF}} i}.
    \]
    By quasisyntomic descent, this recovers the filtration on the absolute prismatic cohomology.  
\end{rmk}

We recall the non-Nygaard complete version:
\begin{constr}\label{constr:non-nygaard-complete}
    Let $R$ be a perfectoid ring and let $\cPrism^{\rm nc}_{-/R}$ be the functor obtained by left Kan extension from $p$-completion of polynomial pre-log $R$ algebras   (see \cite[Remark 2.11]{BLPO-HKR})\footnote{Equivalently, from derived log smooth pre-log $R$-algebras in the sense of \cite[Definition 4.2]{BLPO-HKR}.}
 Then as in \cite[Construction 7.12]{BMS2}, $\cPrism^{\rm nc}_{-/R}/\xi\simeq L\Omega_{-/R}$, where $L\Omega_{-/R}$ is the $p$-completed derived de Rham cohomology. As observed in \cite[Construction 7.12]{BMS2}, this depends on the choice of a perfectoid base. 
By the comparison with $L\Omega$, we deduce that $\cPrism^{\rm nc}$ is a quasisyntomic sheaf on $\lQSyn_R$ with values in $\cD(A_{\inf}(R))$ and that it takes discrete values on $\lQRSPerfd_R$: thus, it can be globalized as before and we consider the quasisyntomic sheaves $R\Gamma_{\cPrism^{\rm nc}}(-/R)$ and the Breuil--Kisin twists given by tensoring with $\cPrism_R\{1\}$ (note that $\cPrism_R\simeq \cPrism_R^{\rm nc}$ by construction). As observed in the proof of \cite[Theorem 13.1]{BSPrism}, it comes equipped with the Nygaard filtration $\Fil_\rN^{\geq m}R\Gamma_{\cPrism^{\rm nc}}(S/R)\{i\}$ again by left Kan extension, and $R\Gamma_{\cPrism}(S/R)\{i\}$ is the completion of $R\Gamma_{\cPrism^{\rm nc}}(S/R)\{i\}$ with respect to this filtration: this follows from the comparison of the graded pieces.

\end{constr}
\begin{rmk}
In \cite[Theorem 1.3]{LogBeilinson}, we apply saturated descent to compare the Nygaard-complete log prismatic cohomology considered here with the Nygaard-completion of the site-theoretic log prismatic cohomology pursued by Koshikawa \cite{Koshikawa} and Koshikawa--Yao \cite{Koshikawa-Yao}.
\end{rmk}
\begin{prop}\label{prop:prism-motivic} Assume that $R$ is a  perfectoid ring. Let $S\in \lQSyn_R$. For all $i,m\in \Z$ and $n\geq 1$, we have equivalences:
    \begin{align*}
    R\Gamma_{\cPrism}(S/R)\{i\}&\simeq R\Gamma_{\cPrism}(((\P^n_R,\P^{n-1}_R)\otimes_R S)/R)\{i\}\\
    \Fil_\rN^{\geq m}R\Gamma_{\cPrism}(S/R)\{i\}&\simeq \Fil_\rN^{\geq m}R\Gamma_{\cPrism}(((\P^n_R,\P^{n-1}_R)\otimes_R S)/R)\{i\}\\
    R\Gamma_{\cPrism^{\rm nc}}(S/R)\{i\}&\simeq R\Gamma_{\cPrism^{\rm nc}}(((\P^n_R,\P^{n-1}_R)\otimes_R S)/R)\{i\}\\
    \Fil_\rN^{\geq m}R\Gamma_{\cPrism^{\rm nc}}(S/R)\{i\}&\simeq \Fil_\rN^{\geq m}R\Gamma_{\cPrism^{\rm nc}}(((\P^n_R,\P^{n-1}_R)\otimes_R S)/R)\{i\}.
    \end{align*}
    In particular, for all $S\in \lQSyn_R$, the strict \'etale sheaves
    \begin{align*}
    R\Gamma_{\cPrism}(-/R)\{i\},\quad \Fil_\rN^{\geq m}R\Gamma_{\cPrism}(-/R)\{i\}, \\
    R\Gamma_{\cPrism^{\rm nc}}(-/R)\{i\}, \quad \Fil_\rN^{\geq m}R\Gamma_{\cPrism^{\rm nc}}(-/R)\{i\}
    \end{align*}
    are $(\P^n,\P^{n-1})$-invariant, so they define objects of $\logFDAeff(S,\Lambda)$. 

\end{prop}
\begin{proof}
Via the trivialization $\cPrism_R\{1\}\simeq \cPrism_R$ (which can be chosen since $R$ is perfectoid) we reduce to the case where $i=0$. Let us prove the completed version first.
Since the filtration is complete and exhaustive, it is enough to prove it for $\gr_N^i(R\Gamma_{\cPrism}(-/R))$. Since the conjugate filtration they carry by Remark \ref{rmk:secondary_filtr} is finite, it is enough to check it on the graded pieces, i.e. it is enough to check that $L\Omega^j_{-/R}$ is $(\P^n,\P^{n-1})$-local: this is done in Proposition \ref{prop:spectra-hodge-dR}. 

For the non-complete version, we deduce the result for $R\Gamma_{\cPrism^{\rm nc}}(-/R)$ by checking it modulo $\xi$ from the analogous result for $p$-completed de Rham cohomology \ref{prop:spectra-hodge-dR}. Finally, we deduce the statement for the filtration from the sequence\[
\Fil_\rN^{\geq m}R\Gamma_{\cPrism^{\rm nc}}(-/R)\to \Fil_\rN^{\geq m}R\Gamma_{\cPrism^{\rm nc}}(-/R) \to \gr^{\geq m}R\Gamma_{\cPrism}(-/R)
\]
plus the fact that $\Fil_\rN^{\geq m}R\Gamma_{\cPrism^{\rm nc}}(-/R)\simeq R\Gamma_{\cPrism^{\rm nc}}(-/R)$ for $m\leq 0$. 
\end{proof}

\begin{cor} \label{cor:prism-motivic-abs} 
 
    Let $\kX\in \FlQSyn$. For all $i,m\in \Z$ and $n\geq 1$, we have equivalences:
    \begin{align*}
    R\Gamma_{\cPrism}(\kX)\{i\}&\simeq R\Gamma_{\cPrism}((\P^n,\P^{n-1})\times \kX)\{i\}\\
    \Fil_\rN^{\geq m}R\Gamma_{\cPrism}(\kX)\{i\}&\simeq \Fil_\rN^{\geq m}R\Gamma_{\cPrism}((\P^n,\P^{n-1})\times \kX)\{i\}.
    \end{align*}
    In particular, for all $S\in \lQSyn$, the strict \'etale sheaves   \[
    R\Gamma_{\cPrism}(-)\{i\},\Fil_\rN^{\geq m}R\Gamma_{\cPrism}(-)\{i\}\colon \FlQSm_S^{\rm op}\to \cD(\Z_p)
    \]
    are $(\P^n,\P^{n-1})$-invariant, so they define objects of $\logFDAeff(S,\Z_p)$. 

\begin{proof}
    By Proposition \ref{prop:prism globalizable}, we can assume $\kX=\Spf(S)$ with $S\in \lQSyn$. Let $S\to S'$ be a quasisyntomic cover in $\lQSyn$ with $S'\in \lQRSPerfd$ and let $(S')^\bullet$ be its \v Cech nerve. Then there is a perfectoid ring $R$ with a map $R\to S'$ so $\cPrism_{S'}\{i\}\simeq \cPrism_{S'/R}\{i\}$ and $\Fil_\rN^{\geq m}\cPrism_{S}\{i\}\simeq \Fil_\rN^{\geq m}\cPrism_{S/R}\{i\}$. In particular, the projections $\P^n_S\to S$ and $\P^n_{(S')^\bullet}\to (S')^\bullet$ induce an equivalence\[
    \begin{tikzcd}
        \arrow[dd, start anchor=north, end anchor=south, no head, xshift=-5em, decorate, decoration={brace,mirror}]R\Gamma_{\cPrism}(S)\{i\}\ar[dd]\arrow[dd, start anchor=north, end anchor=south, no head, xshift=5em, decorate, decoration={brace}]&&&\arrow[dd, start anchor=north, end anchor=south, no head, xshift=-6em, decorate, decoration={brace,mirror}, "\varlim_{m\in \Delta}" left=3pt]R\Gamma_{\cPrism}((S')^m/R)\{i\}\ar[dd]\arrow[dd, start anchor=north, end anchor=south, no head, xshift=6em, decorate, decoration={brace}]\\
       &\ \ar[r,"\simeq"] &\  \\
        R\Gamma_{\cPrism}(\P^n_S,\P^{n-1}_S)\{i\}&&&R\Gamma_{\cPrism}((\P^n_{(S')^m},\P^{n-1}_{(S')^m})/R)\{i\}
     \end{tikzcd}
    \]
     in $\Fun(\Delta^1,\cD(\Z_p))$. Since the vertical maps on the right-hand side are equivalences by Proposition \ref{prop:prism-motivic}, we deduce that the map on the left-hand side is an equivalence too. The result for $\Fil_\rN^{\geq m}R\Gamma_{\cPrism}(-)\{i\}$ follows in the same way. 
\end{proof}
\end{cor}
\subsection{The first Chern class} In order to apply Proposition \ref{prop:build-spectra}, we want to exploit the Chern classes defined in \cite{BhattLurie}. 

We consider the non-Nygaard-complete absolute prismatic cohomology of (non-log) quasisyntomic bounded $p$-adic formal schemes $R\Gamma_{\Prism}(-)\{i\}$ of \cite{BhattLurie} with its Nygaard filtration $\Fil^{\geq \bullet}_\rN R\Gamma_{\Prism}(-)\{i\}$. For $\ul{\kX}$ a (non-log) bounded $p$-adic formal scheme, the filtered object $(R\Gamma_{\cPrism}(\ul{\kX}), \Fil^{\geq \bullet}_\rN R\Gamma_{\cPrism}(\ul{\kX}))$ of \cite{BMS2} coincides with the completion of the filtered object $(R\Gamma_{\Prism}(\ul{\kX}),\Fil^{\geq \bullet}_\rN R\Gamma_{\Prism}(\ul{\kX}))$ functorially in  $\ul{\kX}$ (see \cite[Theorem 5.6.2]{BhattLurie} and \cite[Theorem 13.1]{BSPrism} for the case over a perfectoid ring and then use quasisyntomic descent). 
Recall that by \cite[Notation 7.5.3]{BhattLurie} for all (non-log) bounded $p$-adic formal schemes $\ul{\kX}$ there is a Chern class map\[
c_1^\Prism \colon R\Gamma_{\et}(\ul{\kX},\G_m)[-1]\to \Fil^1_\rN R\Gamma_\Prism(\ul{\kX})\{1\},
\]
Let $\ul{\cE}\to \ul{\kX}$ be a vector bundle and let $\P(\ul{\cE})$ be the associated projective bundle. Since $H^0_{\et}(\P(\ul{\cE})/\ul{\kX},\G_m)=0$ and $H^1_{\et}(\P(\ul{\cE})/\ul{\kX},\G_m)=\Pic(\P(\ul{\cE})/\ul{\kX})\cong \Z$, we have a Chern class given by:
\begin{align*}
\Z\cong\Pic(\P(\ul{\cE})/\ul{\kX}) &\simeq \tau_{\leq 0}R\Gamma_{\et}(\P(\ul{\cE})/\ul{\kX}),\G_m)[1] \to R\Gamma_{\et}(\P(\ul{\cE}),\G_m)[1] \\
&\xrightarrow{c_1^\Prism[2]} \Fil^1_\rN R\Gamma_\Prism(\P(\ul{\cE}))\{1\})[2],
\end{align*}
and by composing with $\Fil^1_\rN R\Gamma_\Prism(-)\{1\}[2]\to R\Gamma_\Prism(-)\{1\}[2]$ we get \[c_1^{\Prism}\colon \Z\cong \Pic(\P(\ul{\cE})/\ul{\kX})\to R\Gamma_\Prism(\P(\ul{\cE}))\{1\}[2].\]

More generally, for all $i\geq 0$ we get \[
\begin{tikzcd}
&\Fil^1_\rN R\Gamma_\Prism(\P(\ul{\cE}))\{1\}[2]^{\otimes i}\ar[r] &\Fil^i_\rN R\Gamma_\Prism(\P(\ul{\cE}))\{i\}[2i]\\
    \Pic(\P(\ul{\cE})/\ul{\kX})\ar[r,"\Delta"]&\Pic(\P(\ul{\cE})/\ul{\kX})^{\otimes i}\ar[u]\ar[d]\\ 
    &R\Gamma_\Prism(\P(\ul{\cE}))\{1\}[2]^{\otimes i}\ar[r]&R\Gamma_\Prism(\P(\ul{\cE}))\{i\})[2i]
\end{tikzcd}\]
If $\ul{\kX}$ is a bounded $p$-adic formal scheme (with trivial log structure), we obtain by \cite[Theorem 6.2.8]{BhattLurie}
\begin{align*}
R\Gamma_\Prism(\ul{\kX})\{i\}&\to \lim_n  R\Gamma_\Prism(\ul{\kX})\{i\}/
\Fil_\rN^{\geq n}R\Gamma_\Prism(\ul{\kX})\{i\} \simeq R\Gamma_\cPrism(\ul{\kX})\{i\}\\
\Fil_\rN^{\geq \bullet}R\Gamma_\Prism(\ul{\kX})\{i\}&\to \lim_n  \Fil_\rN^{\geq \bullet}R\Gamma_\cPrism(\ul{\kX})\{i\}/
\Fil_\rN^{\geq \bullet + n}R\Gamma_\cPrism(\ul{\kX})\{i\}\simeq \Fil_\rN^{\geq \bullet}R\Gamma_\cPrism(\ul{\kX})\{i\}
\end{align*}
induced by completions. This gives Chern classes 
\begin{align*}
(c_1^{\Fil\cPrism})^i\colon \Z\cong \Pic(\P(\ul{\cE})/\ul{\kX})&\to \Fil_\rN^{\geq i}R\Gamma_\cPrism(\P(\ul{\cE}))\{i\}[2i]\\ (c_1^{\cPrism})^i\colon \Z\cong \Pic(\P(\ul{\cE})/\ul{\kX}) &\to R\Gamma_\cPrism(\P(\ul{\cE}))\{i\}[2i].
\end{align*}

\begin{constr}\label{constr:log-chern} Let $\kX\in \FlQSyn$ and $\ul{\cE}\to \ul{\kX}$ a vector bundle, we let $\cE$ and $\P(\cE)$ be the log schemes with the pullback log structure from $\kX$ (see e.g. \cite[Definition 7.1.2]{BPO}). Recall by \cite[Theorem 13.1]{BSPrism} that  for all $i,m\in \Z$, we have maps
\begin{equation}\label{eq:chern-map}
\begin{tikzcd}
  &R\Gamma_{\cPrism}({\kX})\{m-i\}[-2i]\ar[d,"\simeq"]\\ 
  &  R\Gamma_{\cPrism}({\kX})\{m-i\}[-2i]\otimes_{\Z}^L\Pic(\P(\ul{\cE})/\ul{\kX})\ar[d,"id\otimes (c_1^{\cPrism})^i"]\\
  &  R\Gamma_{\cPrism}({\kX})\{m-i\}[-2i]\otimes_{\Z}^L R\Gamma_{\cPrism}(\P(\ul{\cE}))\{i\}[2i]\ar[d,"(*)"]\\
  &R\Gamma_{\cPrism}(\P(\cE)))\{m-i\}[-2i]\otimes_{\Z}^L R\Gamma_{\cPrism}(\P(\cE))\{i\}[2i]\ar[d,"\mu"]\\
  &R\Gamma_{\cPrism}(\P(\cE))\{m\}
\end{tikzcd}    
\end{equation}
 which is functorial in $\kX$ and $\cE$. Here $(*)$ is induced by the pull-back along the maps $\P(\cE)\to \kX$ and $\P(\cE)\to \P(\ul{\cE})$. Similarly, we obtain a map
 \begin{equation}\label{eq:chern-map-fil-i}
    \Fil^{\geq m-i}_\rN R\Gamma_{\cPrism}({\kX})\{n-i\}[-2i]\to \Fil^{\geq m}_\rN R\Gamma_{\cPrism}(\P(\cE))\{n\}.     
 \end{equation}

The maps of \eqref{eq:chern-map-fil-i} assemble to a map
\begin{equation}\label{eq:chern-map-fil}
    \Fil_\rN^{\geq \bullet-i}R\Gamma_\cPrism(\kX)\{n-i\}[-2i]\to \Fil_\rN^{\geq \bullet}R\Gamma_\cPrism(\P(\cE))\{n\}.
\end{equation} in $\widehat{\mathcal{DF}}(\Z_p)$ for every $n, i \in \Z$. 
\end{constr}

In a completely analogous way, for $R$ perfectoid using that $\cPrism^{\rm nc}_{\ul{\kX}/R}\simeq \Prism_{\ul{\kX}/R}$ by \cite[Theorem 13.1]{BSPrism} we obtain a map\[
 \Fil_\rN^{\geq \bullet-i}R\Gamma_{\cPrism^{\rm nc}}(\kX)\{n-i\}[-2i]\to \Fil_\rN^{\geq \bullet}R\Gamma_{\cPrism^{\rm nc}}(\P(\cE))\{n\}.
\]
in $\mathcal{DF}(A_{\inf}(R))$ for every $n, i \in \Z$.
The following result is now analogous to \cite[Lemma 9.1.4]{BhattLurie}.

\begin{lemma}\label{lem:pbf}
    Let $R$ be a perfectoid ring and let $\kX\in \FlQSyn_R$. Let $\cE\to \kX$ be a vector bundle of rank $r+1$ equipped with the induced log structure as above. The maps defined in \eqref{eq:chern-map} and \eqref{eq:chern-map-fil-i} give equivalences for all $m, n$.
    \begin{align*}
        &\bigoplus_{i=0}^r\Fil_\rN^{\geq m-i}R\Gamma_\cPrism(\kX/R)\{n-i\}[-2i]\xrightarrow{\sim} \Fil_\rN^{\geq m}R\Gamma_\cPrism(\P(\cE)/R)\{n\}\\
        &\bigoplus_{i=0}^rR\Gamma_\cPrism(\kX/R)\{n-i\}[-2i]\xrightarrow{\sim} R\Gamma_\cPrism(\P(\cE)/R)\{n\}\\
&\bigoplus_{i=0}^r\Fil_\rN^{\geq m-i}R\Gamma_{\cPrism^{\rm nc}}(\kX/R)\{n-i\}[-2i]\xrightarrow{\sim} \Fil_\rN^{\geq m}R\Gamma_{\cPrism^{\rm nc}}(\P(\cE)/R)\{n\}\\
        &\bigoplus_{i=0}^rR\Gamma_{\cPrism^{\rm nc}}(\kX/R)\{n-i\}[-2i]\xrightarrow{\sim} R\Gamma_{\cPrism^{\rm nc}}(\P(\cE)/R)\{n\}
    \end{align*}

    \begin{proof} We follow the pattern of \cite[Lemma 9.1.4]{BhattLurie}: we first prove the completed version.
    Notice that the second equivalence is a special case of the first with $m\leq 0$. 
As in \eqref{eq:chern-map-fil}, for fixed $n$ we can build a map \[
\bigoplus_{i=0}^r\Fil_\rN^{\geq \bullet-i}R\Gamma_\cPrism(\kX/R)\{n-i\}[-2i]\rightarrow \Fil_\rN^{\geq \bullet}R\Gamma_\cPrism(\P(\cE)/R)\{n\}
\] 
in $\widehat{\mathcal{DF}}(\cPrism_R)$. Since the filtration is complete it is enough to check that the induced map \[
\bigoplus_{i=0}^r\gr_N^{m-i}R\Gamma_\cPrism(\kX/R)\{n-i\}[-2i]\xrightarrow{} \gr_N^{m}R\Gamma_\cPrism(\P(\cE)/R)\{n\}
\]
on graded pieces is an equivalence by \cite[Lemma 5.2 (1)]{BMS2}. Let $I$ be the kernel of the map $\theta\colon \cPrism_R\to R$, generated by $\xi$. By \cite[Remark 6.6]{BMS2}, there is an equivalence of $R$-modules\[
\cPrism_R\{i\}\cotimes_{\cPrism_R}^L R \simeq (I/I^2)^{\cotimes_R i}.
\]
Since the structure of $\cPrism_R$-module on $\gr^m_NR\Gamma_\cPrism(\kX/R)$ induces an $R$-module structure by  \cite[Proposition 7.8]{BMS2}, we have that \[
\gr^m_N(R\Gamma_\cPrism(\kX/R)\{i\})\simeq \gr^m_N(R\Gamma_\cPrism(\kX/R))\cotimes_{\cPrism_R}^L\cPrism\{i\}\simeq \gr^m_N(R\Gamma_\cPrism(\kX/R))\cotimes_{R}^L (I/I^2)^{\otimes i}.
\]
This implies that the finite filtration on $\gr^i_N(R\Gamma_\cPrism(\kX/R))$ induces a map  
\[
\bigoplus_{i=0}^rR\Gamma(\kX,(\bigwedge^{j-i}\L_{\kX/R})_p^\wedge\cotimes_R^L (I/I^2)^{\otimes n-i})[-2i]\to R\Gamma(\P(\cE),\bigwedge^j\L_{\P(\cE)/R})_p^\wedge\cotimes_R^L (I/I^2)^{\otimes n}.\]
The map on $I/I^2$ is just the multiplication. Since they are free $R$-modules of rank $1$, they can be coherently trivialized, which leaves us with the map\[
\bigoplus_{i=0}^rR\Gamma(\kX,\bigwedge^{j-i}\L_{\kX/R})[-2i]\to R\Gamma(\P(\cE),\bigwedge^j\L_{\P(\cE)/R}).
\] 
By construction (see \cite[Theorem 7.6.2]{BhattLurie}), this map agrees with the powers of the first Chern class $c_1^{\mathrm{dR}}(\cO(1))$, which is an equivalence (see \eqref{eq:derived-chern-de-rham}).

For the non-completed version, in the same way as before we see that the twists trivialize modulo $\xi$ and by construction as before, the Chern classes also agree with the de Rham chern class modulo $\xi$ (see again \cite[Theorem 7.6.2]{BhattLurie}): this gives the equivalence
\begin{align*}
\bigoplus_{i=0}^rR\Gamma_{\cPrism^{\rm nc}}(\kX/R)\{n-i\}/\xi[-2i]&\simeq \bigoplus_{i=0}^rR\Gamma(\ul{\kX},L\Omega_{\kX/R})\{n-i\}\xrightarrow{\sim} R\Gamma(\P(\ul{\cE}),L\Omega_{\P(\cE)/R})\\
&\simeq R\Gamma_{\cPrism^{\rm nc}}(\P(\cE)/R)\{n\}/\xi
\end{align*}
which gives the fourth equation by $\xi$-completeness. Finally, again use the sequence\[
\Fil_\rN^{\geq m}R\Gamma_{\cPrism^{\rm nc}}(-/R)\{i\}\to \Fil_\rN^{\geq m}R\Gamma_{\cPrism^{\rm nc}}(-/R)\{i\} \to \gr^{\geq m}R\Gamma_{\cPrism}(-/R)\cotimes_R^L(I/I^2)^{\otimes i}
\]
plus the fact that $\Fil_\rN^{\geq m}R\Gamma_{\cPrism^{\rm nc}}(-/R)\simeq R\Gamma_{\cPrism^{\rm nc}}(-/R)$ for $m\leq 0$ to deduce the third equivalence.

\end{proof}
\end{lemma}

As before, we deduce the result in the absolute case:

\begin{cor}\label{cor:pbf-abs}
 
    Let $\kX\in \FlQSyn$ and let $\cE\to \kX$ be a vector bundle of rank $r+1$ equipped with the induced log structure as above. Then, for all $m$, the maps defined in \eqref{eq:chern-map} give equivalences
    \begin{align*}
        &\bigoplus_{i=0}^r\Fil_\rN^{\geq m-i}R\Gamma_\cPrism(\kX)\{n-i\}[-2i]\xrightarrow{\sim} \Fil_\rN^{\geq m}R\Gamma_\cPrism(\P(\cE))\{n\}\\
        &\bigoplus_{i=0}^rR\Gamma_\cPrism(\kX)\{n-i\}[-2i]\xrightarrow{\sim} R\Gamma_\cPrism(\P(\cE))\{n\}
    \end{align*}
\begin{proof}
    As in the proof of Corollary \ref{cor:prism-motivic-abs}, we assume $\kX=\Spf(S)$ with $S\in \lQSyn$ and we let $S\to S'$ be a quasisyntomic cover in $\lQSyn$ with $S'\in \lQRSPerfd$, $(S')^\bullet$ the \v Cech nerve, and we fix $R\to S'$ with $R$ perfectoid. We have an equivalence \[
    \begin{tikzcd}
        \arrow[dd, start anchor=north, end anchor=south, no head, xshift=-7em, decorate, decoration={brace,mirror}]\bigoplus_{i=0}^r R\Gamma_\cPrism(S)\{n-i\}[-2i]\ar[dd]\arrow[dd, start anchor=north, end anchor=south, no head, xshift=7em, decorate, decoration={brace}]&&&\arrow[dd, start anchor=north, end anchor=south, no head, xshift=-8em, decorate, decoration={brace,mirror}, "\varlim_{m\in \Delta}" left=3pt]\bigoplus_{i=0}^rR\Gamma_\cPrism((S')^m/R))\{n-i\}[-2i]\ar[dd]\arrow[dd, start anchor=north, end anchor=south, no head, xshift=8em, decorate, decoration={brace}]\\
       &\ \ar[r,"\simeq"] &\  \\
        R\Gamma_{\cPrism}(\P(\cE_S))\{n\}&&&R\Gamma_{\cPrism}((\P(\cE_{(S')^m}))/R)\{n\}
     \end{tikzcd},
    \]
    in $\Fun(\Delta^1,\cD(\Z_p))$. The result follows from Lemma \ref{lem:pbf}.
\end{proof}
    \end{cor}

We are now ready to prove our first main result. 

\begin{thm}\label{thm:spectra-prism}
    Let $S\in \QSyn$. There are oriented ring spectra $\mathbf{E}^{\cPrism}$ and $\mathbf{E}^{\Fil\cPrism}$ in $\CAlg(\logFDA(S,\Z_p))$ such that for all $\kX\in \FlQSm_S$ we have\[
\Map_{\logFDA(S,\Z_p)}(\Sigma^{\infty}(\kX),\Sigma^{r,s}\mathbf{E}^{\cPrism}) \simeq R\Gamma_{\cPrism}(\kX)\{s\}[r]
\]
and
\[
\Map_{\logFDA(S,\Z_p)}(\Sigma^{\infty}(\kX),\Sigma^{r,s}\mathbf{E}^{\Fil\cPrism}) \simeq \Fil_\rN^{\geq s}R\Gamma_{\cPrism}(\kX)\{s\}[r].
\]
If $S\in \QSyn_R$ for $R$ perfectoid, $\mathbf{E}^{\cPrism}$ and $\mathbf{E}^{\Fil\cPrism}$ live in $\logFDA(S,A_{\inf}(R))$, and there are oriented ring spectra $\mathbf{E}^{\cPrism^{\rm nc}}$ and $\mathbf{E}^{\Fil\cPrism^{\rm nc}}$ in $\CAlg(\logFDA(S,A_{\inf}(R)))$, together with equivalences 
\begin{equation}\label{eq:prism-deRham}
\mathbf{E}^{\cPrism}\otimes_{\cPrism_R,\theta}R \simeq \E^{\rm \widehat{dR}}\qquad \mathbf{E}^{\cPrism^{\rm nc}}\otimes_{\cPrism_R,\theta}R \simeq \E^{\rm dR}        
\end{equation}
of oriented ring spectra in $\CAlg(\logFDA(S,R))$, where $\theta \colon A_{\rm inf}(R) \to R$ is Fontaine's period map.

\begin{proof} 
Let $\{E_i\}$ be either the collection $\{R\Gamma_{\cPrism}\{i\}\}_{i\in \N}$ or the collection $\{\Fil_\rN^{\geq i}R\Gamma_{\cPrism}\{i\}\}_{i\in \N}$.
By Corollary \ref{cor:prism-motivic-abs} and Remark \ref{rmk:prism-monoid}, $E_*$ is a graded commutative monoid in $\logFDAeff(S,\Lambda)$. Moreover, Proposition \ref{prop:prism-motivic} gives equivalences
\begin{align*}
&R\Gamma_\cPrism(\P^1_{\Spf(S)})\{1\}[2]\simeq\Map_{\logFDAeff(S)}(\Lambda (\P^1),E_1^{\cPrism}[2])\\
&\Fil_\rN^{\geq m}R\Gamma_\cPrism(\P^1_{\Spf(S)})\{1\}[2]\simeq\Map_{\logFDAeff(S)}(\Lambda(\P^1),E_1^{\Fil\cPrism}[2]),
\end{align*}
from which we get a section $\Lambda(\P^1)\to E_1[2]$. By Corollary \ref{cor:pbf-abs}, the composition as constructed in Proposition \ref{prop:build-spectra} is an equivalence. So we can apply Proposition \ref{prop:build-spectra} and get the ring spectra $\mathbf{E}^{\cPrism}$ and $\mathbf{E}^{\Fil\cPrism}$. They are oriented by Lemma \ref{lem:pbf}. 

If $S\in \QSyn_R$ with $R$ perfectoid, the same construction using Proposition \ref{prop:prism-motivic} and  Lemma \ref{lem:pbf} gives $\mathbf{E}^{\cPrism^{\rm nc}}$ and $\mathbf{E}^{\Fil\cPrism^{\rm nc}}$. Moreover, we have equivalences of graded commutative monoids:
\begin{align*}
(R\Gamma_{\cPrism}(-/R)\{i\}\otimes_{\cPrism_R} R)_i &\simeq (R\Gamma_{\cPrism}(-/R)\otimes_{\cPrism_R} (I/I^2)^{\otimes i})_i\\ 
(R\Gamma_{\cPrism^{\rm nc}}(-/R)\{i\}\otimes_{\cPrism_R} R)_i &\simeq (R\Gamma_{\cPrism^{\rm nc}}(-/R)\otimes_{\cPrism_R} (I/I^2)^{\otimes i})_i
\end{align*}
similarly to Lemma \ref{lem:pbf}. Here $I$ is the kernel of Fontaine's period map $\cPrism_R\to R$, generated by $\xi$. In particular, as before, the twists can be canonically trivialized and  the natural equivalence\[
\cPrism_{-/R}/\xi\simeq \widehat{L\Omega_{-/R}}
\]
of \cite[\S 7.2]{BLPO-BMS2} and the equivalence\[
\cPrism^{\rm nc}_{-/R}/\xi\simeq L\Omega_{-/R}
\]
obtained by left Kan extension, are indeed equivalences of graded commutative monoids in $\logFDAeff(S,R)$ by descent from $\lQRSPerfd$ (analogously to \cite[Proposition 7.9]{BMS2}), which by Proposition \ref{prop:build-spectra-maps} lift to the desired equivalences.
\end{proof}
\end{thm}

\subsection{Syntomic cohomology}
Recall that the canonical and Frobenius maps \[
\mathrm{can}\{n\},\phi\{n\}\colon \Fil_\rN^{\geq n}R\Gamma_{\cPrism}(-)\{n\}\to R\Gamma_{\cPrism}(-)\{n\}
\]
arise via unfolding of the maps of graded commutative algebras\[
\pi_{2*}(\mathrm{can}),\pi_{2*}(\phi)\colon \pi_{2*}\TCmin(S,\Z_p)\to \pi_{2*}\TP(S,\Z_p).
\]
These are functorial in $S\in \lQRSPerfd$ (since they are functorial in $\mathrm{Alg}_{E_1}$, see e.g. \cite[\S III]{NikolausScholze}), so in particular they induce maps of graded commutative monoids\[
\{\mathrm{can}\{i\}\},\{\phi\{i\}\}\colon \{\Fil^{\geq i}_\rN R\Gamma_{\cPrism}(-)\{i\}\}_{i}\to \{R\Gamma_{\cPrism}(-)\{i\}\}_{i}
\] by unfolding.
Let $S\in \QSyn$ and $\kX\in \FlQSm_S$, and let 
\begin{align*}
    R\Gamma_{\Fsyn}(\kX,\Z_p(s)):=\fib(\Fil_\rN^{\geq s}R\Gamma_\cPrism(\kX)\{s\}[r]\xrightarrow{\phi\{s\}-\mathrm{can}\{s\}}R\Gamma_\cPrism(\kX)\{s\}[r])
\end{align*} for $s \in \Z$.
If $\kX$ has trivial log structure, then it agrees with the syntomic cohomology of formal schemes of \cite[Construction 7.4.1]{BhattLurie} (therein denoted $R\Gamma_{\syn}$: we keep the notation for schemes and formal schemes separate). Observe that the same proof as in \cite[Theorem 5.1]{AntieauMathewMorrowNikoalus} implies that $R\Gamma_{\Fsyn}$ is left Kan extended from $p$-complete polynomial algebras, so in case $S\in \QSyn_R$ for $R$ perfectoid, it agrees with the same fiber taken in the non-Nygaard completed setting. 
\begin{thm}
Let $S\in \QSyn$. Then there is an oriented ring spectrum $\mathbf{E}^{\Fsyn}$ in $\CAlg(\logFDA(S,\Z_p))$ such that for all $\kX\in \FlQSm_S$ we have\[
\Map_{\logFDA(S,\Lambda)}(\Sigma^{\infty}(\kX),\Sigma^{r,s}\mathbf{E}^{\Fsyn}) \simeq R\Gamma_{\Fsyn}(\kX)\{s\}[r].
\]
\end{thm}
\begin{proof}
Let $c\colon \Lambda(\P^1)\to E_1^{\cPrism}[2]$ and $c'\colon \Lambda(\P^1)\to E_1^{\Fil\cPrism}[2]$ be the maps constructed in the proof of Theorem \ref{thm:spectra-prism}. By construction, the diagram\[
\begin{tikzcd}
    \Lambda(\P^1)[2]\ar[r,"c'"]\ar[dr,bend right=20,"c"]&E_1^{\Fil\cPrism}[2]\ar[d,"{\mathrm{can}\{1\}}"]\\
    &E_1^{\cPrism}[2]
\end{tikzcd}\] 
commutes. By Proposition \ref{prop:build-spectra-maps}, there is a map of ring spectra $\mathrm{can}\colon \E^{\Fil\cPrism}\to \E^{\cPrism}$ in $\logFDA(S,\Lambda)$. Moreover, the two compositions
\[
\begin{tikzcd}
R\Gamma_{\mathrm{\acute{e}t}}(\P^1_{\Spf(S)},\G_m)[1]\ar[r,"{c_1^{\cPrism}}"] &\Fil^1_\rN R\Gamma_\cPrism(\P^1_{\Spf(S)})\{1\}[2]\ar[r,shift left = .5,"{\rm can}\{1\}"]\ar[r,shift right = .5,"\phi\{1\}"']  &\Fil^1_\rN R\Gamma_\cPrism(\P^1_{\Spf(S)})\{1\}[2]
\end{tikzcd}
\]
agree, since $c_1^{\cPrism}$ factors through the equalizer of $\mathrm{can}\{1\}$ and $\phi\{1\}$ (see \cite[Notation 7.5.3]{BhattLurie}). This implies that the diagram \[
\begin{tikzcd}
    \Lambda(\P^1)[2]\ar[r,"c'"]\ar[dr,bend right=20,"c"]&E_1^{\Fil\cPrism}[2]\ar[d,"\phi"]\\
    &E_1^{\cPrism}[2],
\end{tikzcd}
\] commutes.
By Proposition \ref{prop:build-spectra-maps}, there is a map $\phi\colon \E^{\Fil\cPrism}\to \E^{\cPrism}$ of oriented ring spectra in $\CAlg(\logFDA(S,\Z_p))$.

Let $\mathbf{E}^{\Fsyn}$ denote the equalizer of $\varphi$ and ${\rm can}$ in $\logFDA(S,\Z_p)$. For $\kX\in \FlQSm_S$, we have
\begin{align*}
    \Map(\Sigma^{\infty}(\kX),\Sigma^{r,s}\mathbf{E}^{\Fsyn})\simeq &\fib( \Map(\Sigma^{\infty}(\kX),\Sigma^{r,s}\mathbf{E}^{\Fil\cPrism})\xrightarrow{\phi-{\rm can}} \Map(\Sigma^{\infty}(\kX),\Sigma^{r,s}\mathbf{E}^{\cPrism}))\\
      &\simeq \fib(\Fil_\rN^{\geq s}R\Gamma_\cPrism(\kX)\{s\}[r]\xrightarrow{\phi\{s\}-\mathrm{can}\{s\}}R\Gamma_\cPrism(\kX)\{s\}[r])\\
    &=R\Gamma_{\Fsyn}(\kX,\Z_p(s))[r].
\end{align*}
Since $\phi$ and $\mathrm{can}$ are maps in $\CAlg(\logFDA(S,\Z_p))$, and the forgetful functor \[\CAlg(\logFDA(S,\Z_p))\to \logFDA(S,\Z_p)\] preserves all limits and all limits are representable in $\CAlg(\logFDA(S,\Z_p))$ by \cite[Corollary 3.2.2.5]{HA}, we conclude.
\end{proof}

\subsection{Consequences of motivic representability}

 We now list some immediate result following from the fact that $\mathbf{E}^{\cPrism}, \mathbf{E}^{\Fil\cPrism}$, $\mathbf{E}^{\Fsyn}$, $\mathbf{E}^{\cPrism^{\rm nc}}, \mathbf{E}^{\Fil\cPrism^{\rm nc}}$  are oriented ring spectra in $\logDA$.

\begin{thm}\label{thm:gysin-prism}
Let $S\in \QSyn$. Let $\kZ\to \kX$ be a morphism of $p$-adic smooth formal schemes over $S$ such that it is locally the $p$-completion of a pure codimension $d$ closed immersion of smooth schemes over $S$.
Let $\Bl_\kZ(\kX)$ denote the blow-up of $\kX$ in $\kZ$ and $E$ be the exceptional divisor, so that $(\Bl_\kZ(\kX),E)$ is log smooth over $S$. Then for all $j$ there are Gysin maps in $\cD(\Z_p)$, functorial in $(\kX,\kZ)$, 
\begin{align*}
\mathrm{gys}_{\kZ/\kX}^{\cPrism}\colon &R\Gamma_{\cPrism}(\kZ)\{j-d\}[-2d]\to R\Gamma_{\cPrism}(\kX)\{j\}\\ 
\mathrm{gys}_{\kZ/\kX}^{\Fil  \cPrism}\colon  &\Fil_\rN^{\geq j-d}R\Gamma_{\cPrism}(\kZ)\{j-d\}[-2d]\to \Fil_\rN^{\geq j} R\Gamma_{\cPrism}(\kX)\{j\}\\
\mathrm{gys}_{\kZ/\kX}^{\Fsyn}\colon &R\Gamma_{\Fsyn}(\kZ,\Z_p(j-d))[-2d]\to R\Gamma_{\Fsyn}(\kX,\Z_p(j))
\end{align*}
whose homotopy cofibers are respectively given as 
\begin{align*}
    &R\Gamma_{\cPrism}(\Bl_\kZ(\kX),E)\{j\}\\
    &\Fil_\rN^{\geq j} R\Gamma_{\cPrism}(\Bl_\kZ(\kX),E)\{j\}\\
    & R\Gamma_{\Fsyn}((\Bl_\kZ(\kX),E),\Z_p(j)).
\end{align*}
If $S\in \mathrm{QSyn}_R$ with $R$ perfectoid, we have similar Gysin maps for the non-completed prismatic cohomology of \cite{BSPrism}, relative to the perfect prism $(A_{\inf}(R),\ker(\theta))$:
\begin{align*}
\mathrm{gys}_{\kZ/\kX}^{\Prism}\colon &R\Gamma_{\Prism}(\kZ/A_{\inf}(R))\{j-d\}[-2d]\to R\Gamma_{\Prism}(\kX/A_{\inf}(R))\{j\}\\ 
\mathrm{gys}_{\kZ/\kX}^{\Fil\Prism}\colon  &\Fil_\rN^{\geq j-d}R\Gamma_{\Prism}(\kZ/A_{\inf}(R))\{j-d\}[-2d]\to \Fil_\rN^{\geq j} R\Gamma_{\Prism}(\kX/A_{\inf}(R))\{j\}.
\end{align*}
The homotopy cofibers are respectively given as 
\begin{align*}
    &R\Gamma_{\cPrism^{\rm nc}}((\Bl_\kZ(\kX),E)/R)\{j\}\\
    &\Fil_\rN^{\geq j} R\Gamma_{\cPrism^{\rm nc}}((\Bl_\kZ(\kX),E)/R)\{j\}
\end{align*}
\end{thm}
\begin{proof}
    Immediate from \eqref{eq:htp-purity-oriented-spectra-formal}.
\end{proof}
\begin{rmk}
    When $S\in \QSyn_R$ for $R$ a perfectoid ring, the Gysin sequences above take values in $\cD(A_{\inf}(R))$. In that case, after base-change along $\theta$, the equivalence \eqref{eq:prism-deRham} exchanges the prismatic Gysin sequence with the analogous sequence for de Rham cohomology, as both come from homotopy purity. More generally, we expect the Gysin sequence to take value in $\cD(R^{\rm syn})$, following \cite{Bhatt_FGauges}. 
\end{rmk}

\begin{rmk}\label{rmk:generalGysin}  More general log structures on $\kX$ and $\kZ$ are allowed, in light of \cite[Theorem 7.5.4]{BPO} and \cite[Proposition 7.3.9]{BPO-SH},
 see \cite[Theorem 9.4]{BLPO-HKR}. More precisely, the Gysin maps can be constructed in the following generality:
For simplicity of writing,
we work with the non-formal case here.
Let $D_1, D_2, \ldots, D_r$ be smooth divisors on $X\in \Sm_S$, forming a strict normal crossing divisor $D$ over $S$. Let $Z$ be a smooth closed subscheme of $X$ of pure codimension $d$, having strict normal crossings with $D_1+\ldots + D_r$ over $S$ in the sense of \cite[Definition 7.2.1]{BPO} and not contained in any component of $D_1\cup \ldots \cup D_r$. Let $Y = (X, D)$ (resp.\ $W=(Z, Z\cap D)$ be the log scheme in $\SmlSm_S$ given by the compactifying log structure $X-D \hookrightarrow X$ (resp.\ $Z- (Z\cap D) \hookrightarrow Z$). Let $E$ be the exceptional divisor in $\Bl_Z(X)$, and consider the log scheme $(\Bl_Z(Y),E)$, given by the compactifying log structure 
    \[
    \Bl_Z(X) - (E \cup W_1 \cup \ldots \cup W_r) \hookrightarrow \Bl_Z(X)
    \]
    where $W_i$ is the strict transform of $D_i$. For all $j$, there are fiber sequences
    \[
    R\Gamma_{\cPrism}(W^\wedge_p)\{j-d\}[-2d] \to R\Gamma_{\cPrism}(Y^\wedge_p)\{j\} \to R\Gamma_{\cPrism}((\Bl_Z(Y), E)^\wedge_p)\{j\}
    \]
    in $\cD(\Z_p)$, and similarly for the filtered version, for syntomic cohomology, and for the non-complete version over a perfectoid. 
\end{rmk}
\begin{rmk}[Comparison with Tang's Gysin map]   In \cite[Theorem 1.4]{TangPaper}, Tang constructs a Gysin map in (derived) syntomic cohomology \[ \mathrm{cyc}_{Y/X}\colon
R\Gamma_{\syn}(Y, \Z_p(j-d))[-2d] \to R\Gamma_{\syn}(X, \Z_p(j))
\]
for $Y\to X$ a regular immersion of pure codimension $r$, via weighted deformation to the normal cone. We expect this map to agree with our Gysin map $\mathrm{gys}_{Y/X}^{\Fsyn}$ when both are defined.
We remark that the argument of \cite{TangPaper} does not allow, to the best of our understanding, to identify the cofiber of ${\rm cyc}_{Y/X}$ explicitly. 
\end{rmk}
We can also use motivic methods to compute the prismatic and syntomic  cohomology of Grassmannians. Write $\mathrm{Gr}(r,n)$ for the Grassmannian classifying $r$-dimensional subspaces of an $n$-dimensional vector space, defined over $\mathbb{Z}$. We consider $\mathrm{Gr}(r,n)$ as a log scheme over $\Spec(\Z)$, with trivial log structure. 
As in \cite[\S 7.4]{BPO-SH}, we let :\begin{equation}
\label{Znd}
Z_{n,d} = \Z[x_1,\ldots,x_n,y_1,\ldots,y_d]/(z_1, \ldots , z_n)
\end{equation}
with $z_1,\ldots, z_n$ satisfying
\[
1+tz_1 +\ldots + t^nz_n =(1+tx_1 +\ldots +t^rx_r)(1+ty_1+\ldots +t^{n-r}y_{n-r}).
\]
As a ring, it agrees with the singular cohomology of the complex Grassmannian $\mathrm{Gr}(r,n)(\mathbb{\C})$ (see \cite[Section 6.2]{Naumann_Spitzweck_Oestvaer} for another description).

\begin{thm}
\label{thm:coho_Grassmannian}Let $S\in \QSyn$. Let $\kX$ be a $p$-adic log smooth formal scheme over $S$.
Then there are isomorphisms of bigraded rings, functorial in $X$:
    \begin{align*}
 \phi_{r,n}^{\cPrism} \colon  H^*_{\cPrism}(\kX)\{\bullet\} \otimes_{\mathbb{Z}} Z_{r,n} &\xrightarrow{\sim} H^*_{\cPrism}(\mathrm{Gr}(r, n) \times \kX)\{\bullet\}\\
 \phi_{r,n}^{\Fil\cPrism}\colon \Fil^{\geq \bullet}  H^*_{\cPrism}(\kX)\{\bullet\} \otimes_\mathbb{Z} Z_{r,n} &\xrightarrow{\sim} \Fil^{\geq \bullet}H^*_{\cPrism}(\mathrm{Gr}(r, n)\times \kX)\{\bullet\} \\
\phi_{r,n}^{\Fsyn}\colon H^*_{\Fsyn}(\kX,\Z_p(\bullet))\otimes_\Z Z_{r,n} &\xrightarrow{\sim} H^*_{\Fsyn}(\mathrm{Gr}(r, n)\times \kX,\Z_p(\bullet)).
\end{align*}
If $S\in \QSyn_R$ with $R$ perfectoid, a similar statement holds for the non-completed prismatic cohomology relative to $R$. 
\end{thm}
\begin{proof}
The analogous version with $R\Gamma$ is strict \'etale local on $\kX$,
so we may assume that $\kX$ is a $p$-completion of $X\in \lSm_S$.
Then this is an immediate consequence of \cite[Theorem 7.4.2]{BPO-SH}, together with the fact that the motivic ring spectra $\mathbf{E}^{\cPrism}, \mathbf{E}^{\Fil\cPrism}$, $\mathbf{E}^{\Fsyn}$ are oriented.
\end{proof}

Finally, by applying the smooth blow-up formula \eqref{eq:sm-blow-up-oriented-spectra-formal} to the aforementioned ring spectra, we have:

\begin{thm}\label{thm:blow-up_main} (Blow-up formula) Let $S\in \QSyn$.
Let $\kZ\to \kX$ be a morphism of $p$-adic smooth formal schemes over $S$ such that it is locally the $p$-completion of a pure codimension $d$ closed immersion of smooth schemes over $S$.
Let $\kX' = \Bl_\kZ(\kX)$. Then there are equivalences, functorial in $(\kX,\kZ)$,
    \begin{align*}
 R\Gamma_{\cPrism}(\kX)\{j\} \oplus \bigoplus_{0<i< d} R\Gamma_{\cPrism}(\kZ)\{j-i\} [-2i] &\xrightarrow{\sim} R\Gamma_{\cPrism}(\kX')\{j\}\\
  \Fil^{\geq j} R\Gamma_{\cPrism}(\kX)\{j\} \oplus  \bigoplus_{0<i< d} \Fil^{\geq j-i} R\Gamma_{\cPrism}(\kZ)\{j-i\}[-2i]&\xrightarrow{\sim} \Fil^{\geq j} R\Gamma_{\cPrism} \{j\}(\kX') \\
R\Gamma_{\Fsyn}(\kX, \Z_p(j)) \oplus \bigoplus_{0<i< d} R\Gamma_{\Fsyn}(\kZ, \Z_p(j-i)) [-2i] &\xrightarrow{\sim} R\Gamma_{\Fsyn}(\kX', \Z_p(j))
    \end{align*}
If $S\in \QSyn_R$ with $R$ perfectoid, a similar statement holds for the non-completed prismatic cohomology relative to $R$. 
\end{thm}

\section{Saturated descent and crystalline comparison}
The goal of this section is to prove that, under certain conditions, Gabber's cotangent complex is controlled by the classical, non-logarithmic cotangent complex. In light of the conjugate filtration, the same is true for log prismatic cohomology of log quasisyntomic rings over a perfectoid ring (see Theorem \ref{thm:main-sat-descent-prism}). 
Similar descent results are enjoyed by the log de Rham-Witt complex of \cite{Yao-logDRW}, and are a key tool in the proof of the crystalline comparison Theorems \ref{thm:crys_comp_triv} and \ref{thm:crys_comp_logpoint}. 

\subsection{Saturated descent for the cotangent complex} Fix a prime number $p$.  Recall that for any commutative monoid $M$, we have the $p$-power map (we will use the additive notation for monoids in this section) 
\[F_M\colon M\to M, \quad x\mapsto p x\]

\begin{defn}\label{def:perf_semiperf_mon}We say that the monoid $M$ is \emph{perfect} (resp.\ \emph{semiperfect}) if $F_M$ is an isomorphism (resp.\ if $F_M$ is surjective). \end{defn}

The two notions admit a relative version given as follows. We write $M_{\rm perf}$ (resp.~$M^\flat$) for the direct (resp.~inverse) limit perfection (see \cite[\S 4.3]{BLPO-BMS2}).

\begin{definition}[Relative Frobenius]\label{def:FrobMonoid}
Let $P\to M$ be a map of monoids.
Consider the commutative diagram
\[
\begin{tikzcd}
P\ar[r,"F_{P}"]\ar[d]&
P\ar[d]
\\
M\ar[r]\ar[rr,bend right,"F_M"']&
M^{(1)}\ar[r,"F_{M/P}"]&
M
\end{tikzcd}
\]
where the left square is defined to be cocartesian. Here $F_P$ and $F_M$ denote the $p$-Frobenius endomorphism of $P$ and $M$ respectively (see \cite[\S I.4.4]{ogu}). 
We say that the map $P\to M$ is \emph{relatively perfect} (resp.\ \emph{relatively semiperfect}) if the relative Frobenius $F_{M/P}$ is an isomorphism (resp.\ surjective).
\end{definition}
 
Recall that a monoid $M$ is \emph{saturated} if it is integral (i.e.\ $M\subseteq M^{\rm gp}$) and if $x\in M^{\rm gp}$ and $nx\in M$ for some $n>0$ imply $x \in M$. A monoid $M$ is \emph{sharp} if its neutral element is its unique invertible element. 
\begin{lemma}
\label{descent_1}
Let $M$ be a semiperfect saturated monoid.
Then the inclusion  $M^*\to M$ is relatively perfect.
In particular,
if $M$ is a semiperfect sharp saturated monoid,
then $M$ is perfect.
\end{lemma}
\begin{proof}
Let $(x,y)$ and $(x',y')$ be two elements of $M\oplus_{M^*,F_{M^*}}M^*=M^{(1)}$.
If $F_{M/M^*}(x,y)=F_{M/M^*}(x',y')$,
then we have $px+y=px'+y'$.
This implies $p(x-x')\in M^*$,
so $x-x'\in M^*$, since $M$ is saturated.
Hence $(x,y)=(x',y')$ in $M^{(1)}$.
This shows that $F_{M/M^*}$ is injective.
Since $F_M$ is surjective by assumption,
$F_{M/M^*}$ is surjective.
\end{proof}

We now state a key result that will be very useful in the rest of the section. The proof is the same as \cite[Lemma 3.3]{Illusielog}, but for the sake of completeness, we spell it out in the generality that we need. Recall \cite[Definition I.4.3.1]{ogu} that a morphism $u\colon P\to Q$ of integral monoids is \emph{Kummer} if it is injective and if, for every $q\in Q$ there exists $n\in \Z^+$ and $p\in P$ such that $nq = u(p)$ (this last condition is also known as being $\Q$-surjective). 
 
\begin{lemma}\label{lem:illusie-trick}
    Let $u\colon P\to Q$ be a Kummer map of saturated monoids.
    Then the map \[
    Q\oplus_P^{\rm sat} Q \to Q\oplus Q^{\rm gp}/P^{\rm gp}\qquad (a,b)\mapsto (ab, \overline{b})
    \] is an isomorphism of monoids.
The left-hand side is the pushout in the category of saturated monoids and $\overline{(\ )}$ is the natural map $Q\hookrightarrow Q^{\rm gp}\to Q^{\rm gp}/P^{\rm gp}$. More generally, we have an isomorphism
\begin{equation}\label{eq:illusie-trick}
Q^{\oplus^{\rm sat}_P d}\cong Q\oplus (Q^{\rm gp}/P^{\rm gp})^{\oplus d-1}
\end{equation} for all $d \ge 1$. 
\end{lemma}
\begin{proof}
The second part follows from the first by induction on $d$ since\[
Q^{\oplus^{\rm sat}_P d} \cong Q\oplus^{\rm sat}_P (Q^{\oplus^{\rm sat}_P d-1})\cong   Q \oplus^{\rm sat}_P (Q\oplus (Q^{\rm gp}/P^{\rm gp})^{\oplus d-2}) \cong (Q\oplus_P^{\rm sat} Q)\oplus (Q^{\rm gp}/P^{\rm gp})^{\oplus d-1}).
\]
We follow \cite[Lemma 3.3]{Illusielog}. Recall that $Q\oplus_P^{\rm sat} Q$ is the submonoid of $(Q\oplus_{P}Q)^{\rm gp}\cong Q^{\rm gp}\oplus_{P^{\rm gp}}Q^{\rm gp}$ given by pairs $(a,b)$ such that there is $m\in \N^+$ with   $ma, m b \in Q$. On the other hand, $Q^{\rm gp}\oplus_{ P^{\rm gp}}Q^{\rm gp}$ is also a pushout in the category of abelian groups, so there is an isomorphism $Q^{\rm gp}\oplus_{P^{\rm gp}}Q^{\rm gp}\simeq Q^{\rm gp}\oplus Q^{\rm gp}/P^{\rm gp}$ given by the map $(a,b)\mapsto (ab,\overline{b})$. In particular, we have a commutative diagram\[
\begin{tikzcd}
Q\oplus^{\rm sat}_P Q\ar[d,hook]\ar[rrr,"({a,b)\mapsto (ab,\overline{b})}"]&&&Q\oplus Q^{\rm gp}/P^{\rm gp}\ar[d]\\
(Q\oplus_{P}Q)^{\rm gp}\ar[r,"\simeq"]&Q^{\rm gp}\oplus_{P^{\rm gp}}Q^{\rm gp}\ar[rr,"{(a,b)\mapsto (ab,\overline{b})}"',"\simeq"]&&Q^{\rm gp}\oplus Q^{\rm gp}/P^{\rm gp}
\end{tikzcd}
\]
which implies that \eqref{eq:illusie-trick} is injective. Let $(x,y)\in Q\oplus Q^{\rm gp}/P^{\rm gp}$, then since $P\to Q$ is $\Q$-surjective there is $n>0$ such that $ny=0$, so $n(x,y) = (nx,0)$, which is in the image of $Q\oplus^{\rm sat}_PQ$, and since it is saturated we have that $(x,y)$ is also in the image, showing the surjectivity of \eqref{eq:illusie-trick}.
\end{proof}

\begin{rmk}\label{rmk:perf-kummer}
If $M$ is $p$-torsionfree, the map $M\to M_{\rm perf}$ is always Kummer: Indeed, it is injective, and for all $x\in M_{\rm perf}$ there is $n$ such that ${p^{n}}x\in M$, hence the canonical morphism is $\Q$-surjective. 
Moreover, if $M$ is integral (resp.~saturated), then $M_{\rm perf}$ is also integral (resp.~saturated) by \cite[Proposition I.1.3.6]{ogu}.
\end{rmk} 

\begin{rmk}\label{rmk:Niziol_contraction}
 Let $P\to Q$ be a map of saturated monoids, and let $\Z[Q^\bullet]$ the cosimplicial ring given by $d\mapsto \Z[Q^{\oplus^{\rm sat}_P d}]$, with the usual {\v C}ech differentials. In light of Lemma \ref{lem:illusie-trick}, following \cite[Lemma 3.28]{Niziol} we can rewrite it as 
\[
\begin{tikzcd}
\Z[P]\xrightarrow{\epsilon} \Z[Q] \ar[r,shift left = 1.5,"b_0"]\ar[r,"b_1"'] & \Z[Q \oplus Q^{\rm gp}/ P^{\rm gp}] \ar[r,shift left = 2]\ar[r,shift left = .5]\ar[r,shift right=1] &\Z[Q \oplus (Q^{\rm gp}/ P^{\rm gp})^{\oplus 2}] \longrightarrow \ldots
\end{tikzcd}
\]
where the differentials $b_k^n \colon \Z[Q \oplus (Q^{\rm gp}/ P^{\rm gp})^{\oplus n}] \to \Z[Q \oplus (Q^{\rm gp}/ P^{\rm gp})^{\oplus n+1}] $    are determined (in multiplicative notation) by
\[ 
b_k^n(x_0, x_1, \ldots, x_n) = \begin{cases}
    (x_0,   \overline{x_0}x_1^{-1}x_2^{-1}\cdots x_n^{-1}, x_1, \ldots, x_n) &\text{ if $k=0$}\\
    (x_0,x_1, \ldots, x_{k-1}, 1, x_k, \ldots, x_n)\quad &\text{ if $k\neq0$}
\end{cases}
\]
and, as above, $\overline{x_0}$ denotes the image  of $x_0\in Q$ in $Q^{\rm gp}/P^{\rm gp}$.
By (the proof of) \cite[Lemma 3.28]{Niziol} we have that $\Z[P]\to \Z[Q \oplus (Q^{\rm gp}/P^{\rm gp})^{\oplus \bullet}]$ is a $\mathbb{Z}[P]$-linear chain homotopy equivalence, and therefore the same is true for $\Z[P] \to \Z[Q^\bullet]$.
\end{rmk}
The following is a generalization of \cite[Corollary 4.13]{BLPO-BMS2}, providing sufficient conditions under which the $p$-completed log cotangent complex agrees with its non-log (classical) counterpart. 

\begin{prop}
\label{prop:semipf-invariance-cotangent}
 Let $R\to (R,P)\to (A,M)$ be pre-log rings with $P$ and $M$ saturated.
\begin{enumerate}
    \item If $M$ is semiperfect, the map\[
\bigwedge^i\L_{A/R}\to \bigwedge^i\L_{(A,M)/R}
\]
is an equivalence after derived $p$-completion.
\item If $P$ is semiperfect, the map\[
\bigwedge^i\L_{(A,M)/R}
\rightarrow
\bigwedge^i\L_{(A,M)/(R,P)}
\] 
is an equivalence after derived $p$-completion.
\end{enumerate}
In particular, if both $P$ and $M$ are semiperfect the composite map \[(\bigwedge^i\L_{A/R})_p^\wedge \xrightarrow{}
(\bigwedge^i\L_{(A,M)/(R,P)})_p^\wedge
\] is an equivalence.
\end{prop}
\begin{proof}
Let $M$ be semiperfect. By derived Nakayama, it is enough to prove that the maps are equivalences after applying $-\otimes_\Z^L\Z/p\Z$.
 By Lemma  \ref{descent_1},  we can use \cite[Lemma 4.12]{BLPO-BMS2} (which is essentially \cite[Corollary 7.11]{Bhatt2012padicDD}) and 
obtain equivalences
\[
\L_{\F_p[M]/\F_p[M^*]}
\simeq
\L_{(\F_p[M],M)/\F_p[M^*]}
\simeq
0.
\]
The two transitivity sequences
\begin{gather*}
\L_{\F_p[M^*]/\F_p}\otimes_{\F_p}^L \F_p[M]
\to
\L_{\F_p[M]/\F_p}
\to
\L_{\F_p[M]/\F_p[M^*]}
\\
\L_{\F_p[M^*]/\F_p}\otimes_{\F_p}^L \F_p[M]
\to
\L_{(\F_p[M],M)/\F_p}
\to
\L_{(\F_p[M],M)/\F_p[M^*]}
\end{gather*}
induce an equivalence
$
\L_{\F_p[M]/\F_p}
\simeq
\L_{(\F_p[M],M)/\F_p}.
$
From this,
we obtain an equivalence
\[
\L_{\Z[M]/\Z}\otimes_\Z^L \Z/p\Z
\simeq
\L_{(\Z[M],M)/\Z}\otimes_\Z^L \Z/p\Z
\]
since $\Z[M]\otimes_\Z^L \Z/p\Z\simeq \F_p[M]$.
We have a cocartesian square
\[
\begin{tikzcd}
\L_{\Z[M]/\Z}\otimes_{\Z[M]}^L A\ar[d]\ar[r]&
\L_{A/R}\ar[d]
\\
\L_{(\Z[M],M)/\Z}\otimes_{\Z[M]}^L A\ar[r]&
\L_{(A,M)/R}
\end{tikzcd}
\]
by \cite[(3.3)]{BLPO-HKR},
so we deduce
$
\L_{A/R}\otimes_\Z^L \Z/p\Z
\simeq
\L_{(A,M)/R}\otimes_\Z^L \Z/p\Z
$, in particular $\L_{(A,M)/A}\otimes_\Z^L\Z/p\Z \simeq 0$. Then the transitivity sequence\[
\L_{A/R}\to \L_{(A,M)/R}\to \L_{(A,M)/A}
\]
induces a finite filtration on $\bigwedge^i\L_{(A,M)/R}$ with graded pieces given by \[
\gr^j(\bigwedge^i\L_{(A,M)/R})\simeq \bigwedge^j\L_{A/R}\otimes^L_A \bigwedge^{i-j}\L_{(A,M)/A}.
\]
By applying $-\otimes^L_\Z \Z/p\Z$ we deduce that $\gr^j$ vanishes for $j<i$ and for $j=i$ it gives the first desired equivalence \[
\bigwedge^i\L_{A/R}\otimes_\Z^L\Z/p\Z\simeq \bigwedge^i\L_{(A,M)/R}\otimes_\Z^L\Z/p\Z.
\]
Let now $P$ be semiperfect, consider  the transitivity sequence
\[
\L_{(R,P)/R}\otimes_R^L A
\to
\L_{(A,M)/R}
\to
\L_{(A,M)/(R,P)}.
\]
By the previous case, we have $\L_{(R,P)/R}\otimes^L \Z/p\Z\simeq 0$,
hence again by looking at the filtration on the derived exterior powers, we conclude that \[
\bigwedge^i\L_{(A,M)/R}\otimes_\Z^L \Z/p\Z
\simeq
\bigwedge^i \L_{(A,M)/(R,P)}\otimes_\Z^L \Z/p\Z
\]
as required.
\end{proof}

\begin{cor}
\label{semiperfect.5}
Let $\phi\colon (R,P)\to (A,M)$ be a map of saturated pre-log rings with $P$ and $M$ semiperfect.
Then $\phi$ is log quasismooth (resp.\ quasisyntomic) if and only if $\underline{\phi}\colon R\to A$ is quasismooth (resp.\ quasisyntomic).
\end{cor}
\begin{proof}
Immediate from Proposition \ref{prop:semipf-invariance-cotangent} and the definition of (log) quasismooth (resp.\ (log) quasisyntomic), see  \cite[Definition 4.5]{BLPO-BMS2} and \cite[Definition 4.10(2, 3)]{BMS2}. 
\end{proof} 

Let $\Delta$ be the cosimplicial category. Recall that if $\cC^\otimes$ is a monoidal $\infty$-category, then the category of cosimplicial objects $\Fun(\Delta, \cC^\otimes)$ is again a monoidal $\infty$-category with tensor product defined levelwise $(A\otimes B)^{[i]} = A^{[i]}\otimes_{\cC}B^{[i]}$.  

\begin{lemma}\label{lem:vanish-cotangent-Kummer}
    Let $(A,M)$ be a pre-log ring with $M$ saturated, and let $M\to N$ be a Kummer map with $N$ saturated. Let $N^{\oplus_M^{\rm sat} (\bullet + 1)}$ be the  cosimplicial monoid given by the \v Cech nerve in the category of saturated monoids of the map $M\to N$. Then
    \[
\mathbb{L}_{(A\otimes_{\Z[M]}^L\Z[N^{\oplus_M^{\rm sat} (\bullet + 1)}],N^{\oplus_M^{\rm sat} (\bullet + 1)})/(A,M)}\]
    is nullhomotopic, and its cosimplicial limit vanishes in $\cD(A)$.
\begin{proof}  
By \cite[Proposition 3.8]{BLPO-HKR}, we have that \[
\mathbb{L}_{(\Z[N^d],N^d)/(\Z[M],M)}\simeq \Z[N^d]\otimes^L_\Z (N^d)^{\rm gp}/M^{\rm gp}
\]
for all $d \ge 1$. The map $M^{\rm gp} \to (N^d)^{\rm gp}$ is the diagonal. From this, we obtain an equivalence\begin{equation}\label{eq:sat_cot1}
\mathbb{L}_{(\Z[N^{\oplus_M^{\rm sat} (\bullet + 1)}],N^{\oplus_M^{\rm sat} (\bullet + 1)})/(\Z[M],M)}\simeq \Z[N^{\oplus_M^{\rm sat} (\bullet + 1)}]\otimes^L_\Z (N^{\oplus_M^{\rm sat} (\bullet + 1)})^{\rm gp}/M^{\rm gp},
\end{equation} 
in $\Fun(\Delta,\cD(\Z[M]))$, where the tensor product is taken levelwise.
We now apply the equivalence $N^{\oplus_M^{\rm sat} d}\simeq N\oplus (N^{\rm gp}/M^{\rm gp})^{\oplus d-1}$ of Lemma \ref{lem:illusie-trick} to obtain an equivalence \[
(N^{\oplus_M^{\rm sat} (\bullet + 1)})^{\rm gp}/M^{\rm gp} \simeq (N^{\rm gp}/M^{\rm gp})^{\oplus \bullet}.
\] in $\Fun(\Delta,\cD(\Z))$. 
The right-hand side is the cosimplicial abelian group 
\[
\begin{tikzcd}
N^{\rm gp}/M^{\rm gp}\ar[r,shift left = 1.5,"id\oplus 0"]\ar[r,"0\oplus id"'] &N^{\rm gp}/M^{\rm gp}\oplus N^{\rm gp}/M^{\rm gp}\ar[r,shift left = 2]\ar[r,shift left = .5]\ar[r,shift right=1] &N^{\rm gp}/M^{\rm gp}\oplus N^{\rm gp}/M^{\rm gp}\oplus N^{\rm gp}/M^{\rm gp}\ldots,
\end{tikzcd}
\] 
which is chain homotopic to $0$ via the map $N^{\rm gp}/M^{\rm gp}\to 0$. This implies that the cosimplicial $\Z[M]$-module \eqref{eq:sat_cot1} is chain homotopic to zero. 
Base-change for the Gabber cotangent complex provides an equivalence \[
\mathbb{L}_{(A\otimes_{\Z[M]}^L\Z[N^{\oplus_M^{\rm sat} (\bullet + 1)}],N^{\oplus_M^{\rm sat} (\bullet + 1)})/(A,M)}\simeq A\otimes^L_{\Z[M]}\Z[N^{\oplus_M^{\rm sat} (\bullet + 1)}]\otimes^L_\Z (N^{\oplus_M^{\rm sat} (\bullet + 1)})^{\rm gp}/M^{\rm gp}
\] in $\Fun(\Delta,\cD(A))$.
Here $A\in \Fun(\Delta,\cD(\Z[M]))$ is the constant functor and the tensor product is taken in $\Fun(\Delta,\cD(\Z[M]))$. By the Eilenberg--Zilber theorem, taking the associated double complex we have a quasi--isomorphism of double complexes:\[
0\simeq A\otimes^L_{\Z[M]}\Z[N^{\oplus_M^{\rm sat} (\bullet + 1)}]\otimes^L_\Z (N^{\oplus_M^{\rm sat} (\bullet + 1)})^{\rm gp}/M^{\rm gp}
\]
Recall that the homotopy limit of cosimplicial chain complexes agrees with the product totalization of the associated double complex by \cite[Proposition 4.9]{BousfieldKan}, and since the double complex associated to $A\otimes^L_{\Z[M]}\Z[N^{\oplus_M^{\rm sat} (\bullet + 1)}]\otimes^L_\Z (N^{\oplus_M^{\rm sat} (\bullet + 1)})^{\rm gp}/M^{\rm gp}$ is an acyclic right-half plane double complex, by \cite[Acyclic Assembly Lemma 2.7.3]{WeibelHom} we conclude that
\begin{align*}
&\varlim_\Delta (\mathbb{L}_{(A\otimes_{\Z[M]}^L\Z[N^{\oplus_M^{\rm sat} (\bullet + 1)}],N^{\oplus_M^{\rm sat} (\bullet + 1)})/(A,M)})\\&\simeq \Tot(A\otimes^L_{\Z[M]}\Z[N^{\oplus_M^{\rm sat} (\bullet + 1)}]\otimes^L_\Z (N^{\oplus_M^{\rm sat} (\bullet + 1)})^{\rm gp}/M^{\rm gp})\\&\simeq 0,
\end{align*}
as required.
\end{proof}

\end{lemma} 
\begin{lemma}\label{lem:sat-descent-cotangent}
    Let $(R,P)\to (A,M)$ be a map of integral pre-log rings with $M$ saturated, and let $M\to N$ be a Kummer map of saturated monoids. Then the canonical map\[
    \bigwedge^i\mathbb{L}_{(A,M)/(R,P)}\to \varlim_{\Delta}(\bigwedge^i\mathbb{L}_{(A\otimes_{\Z[M]}^L\Z[N^{\oplus_M^{\rm sat} (\bullet + 1)}],N^{\oplus_M^{\rm sat} (\bullet + 1)})/(R,P)})
    \]
    is an equivalence.
\end{lemma}
\begin{proof}
As in \cite[Theorem 2.9]{BLPO-HKR}, the transitivity sequence for the composition $(R,P)\to (A,M)\to (A\otimes_{\Z[M]}^L\Z[N^{\oplus_M^{\rm sat} (\bullet + 1)}],N^{\oplus_M^{\rm sat} (\bullet + 1)})$ gives a cofiber sequence \[\begin{tikzcd}[row sep = small]
\Z[N^{\oplus_M^{\rm sat} (\bullet + 1)}]\otimes_{\Z[M]}^L\mathbb{L}_{(A,M)/(R,P)} \ar{d} \\ \mathbb{L}_{(A\otimes_{\Z[M]}^L\Z[N^{\oplus_M^{\rm sat} (\bullet + 1)}],N^{\oplus_M^{\rm sat} (\bullet + 1)})/(R,P)} \ar{d} \\ \mathbb{L}_{(A\otimes_{\Z[M]}^L\Z[N^{\oplus_M^{\rm sat} (\bullet + 1)}],N^{\oplus_M^{\rm sat} (\bullet + 1)})/(A,M)}
\end{tikzcd}\]
in $\Fun(\Delta,\cD(A))$. This induces a finite filtration on the cosimplicial complex of $A$-modules $\bigwedge^i\mathbb{L}_{(A\otimes_{\Z[M]}^L\Z[N^{\oplus_M^{\rm sat} (\bullet + 1)}],N^{\oplus_M^{\rm sat} (\bullet + 1)})/(R,P)}$ with $j$th graded piece \[
\mathrm{gr}^j
=
\bigwedge^j ( \Z[N^{\oplus_M^{\rm sat} (\bullet + 1)}] \otimes^L_{\Z[M]}\mathbb{L}_{(A,M)/(R,P)} )\otimes^L_{A} \bigwedge^{i-j}\mathbb{L}_{(A\otimes_{\Z[M]}^L\Z[N^{\oplus_M^{\rm sat} (\bullet + 1)}],N^{\oplus_M^{\rm sat} (\bullet + 1)})/(A,M)}.
\]
For $j<i$, $\bigwedge^{i-j}\mathbb{L}_{(A\otimes_{\Z[M]}^L\Z[N^{\oplus_M^{\rm sat} (\bullet + 1)}],N^{\oplus_M^{\rm sat} (\bullet + 1)})/(A,M)}$ is nullhomotopic by Lemma \ref{lem:vanish-cotangent-Kummer},
so we conclude that $\varlim_\Delta \mathrm{gr}^j$ vanishes as in the proof of Lemma \ref{lem:vanish-cotangent-Kummer}.
While for $j=i$ we have\[
\bigwedge^i\mathbb{L}_{(A\otimes_{\Z[M]}^L\Z[N^{\oplus_M^{\rm sat} (\bullet + 1)}],N^{\oplus_M^{\rm sat} (\bullet + 1)})/(R,P)}\simeq \bigwedge^i \Z[N^{\oplus_M^{\rm sat} (\bullet + 1)}] 
 \otimes^L_{\Z[M]}\mathbb{L}_{(A,M)/(R,P)}.
\]
By Remark \ref{rmk:Niziol_contraction}, the map $\Z[M]\to \Z[N^\bullet]$ is a homotopy equivalence, hence\[
\varlim_{\Delta} \bigwedge^i\mathbb{L}_{(A\otimes_{\Z[M]}^L\Z[N^{\oplus_M^{\rm sat} (\bullet + 1)}],N^{\oplus_M^{\rm sat} (\bullet + 1)})/(R,P)} \simeq
\varlim_{\Delta} \bigwedge^i\mathbb{L}_{(A,M)/(R,P)}\otimes_{\Z[M]}^L\Z[N^{\oplus_M^{\rm sat} (\bullet + 1)}],
\]
which is equivalent to $\bigwedge^i\mathbb{L}_{(A,M)/(R,P)}$, as required.
\end{proof}

The following definition is inspired by \cite[Definition 4.1]{BMS2}.
\begin{defn}
Let $A$ be a ring and let $C\in \cD(A)$. We say that $C$ is \emph{$p$-completely discrete} if $C \otimes^L_A A/p \in \cD(A/p)$ is concentrated in degree $0$. 
\end{defn}
We remark that if $C$ is $p$-completely flat in the sense of \cite[Definition 4.1(2)]{BMS2}, then it is $p$-completely discrete in the above sense.

From Lemma \ref{lem:sat-descent-cotangent} we get the following descent result for the cotangent complex. 

\begin{thm}\label{thm:sat-descent-nonderived}
 
Let $(R,P)\to (A,M)$ be an integral map of integral pre-log rings with bounded $p$-power torsion and let
$M \to N$ be a Kummer map of saturated monoids.
Assume that $A\otimes^L_{
\Z[M]} \Z[N]$ 
is $p$-completely discrete.
Then for all $(A,M)\to (S,Q)$ with $A\to S$ $p$-completely 
flat, $M\to Q$ integral and injective, and $S$ with bounded $p$-power torsion, the natural map
\begin{equation}\label{eq:sat-descent}
    (\bigwedge^i\L_{(S,Q)/(R,P)})_p^{\wedge}\to \varlim_{[m]\in \Delta}(\bigwedge^i\L_{(S\otimes_{\Z_p\langle M\rangle}\Z_p\langle N^{\oplus_M^{\rm sat} (m+ 1)}\rangle,Q\oplus_M N^{\oplus_M^{\rm sat} (m + 1)})/(R,P)})_p^{\wedge}
\end{equation}
is an equivalence.
\begin{proof} 
    As in \cite[Remark 4.9]{BMS2}, it is enough to check that the natural map
\begin{equation}\label{eq:reduction-to-cotan}
    \L_{(S/p^n,Q)/(R,P)}\to \varlim_{[m]\in \Delta}   (\L_{(S/p^n\otimes_{\Z[M]}\Z[N^{\oplus_M^{\rm sat} (m + 1)}],Q\oplus_M N^{\oplus_M^{\rm sat} (m + 1)})/(R,P)})
\end{equation}
 is an equivalence for all $n$. We can thus assume that $p$ is nilpotent in $A$, so $A\to S$ is  flat and the map
 \begin{equation}\label{eq:clubsuit-pnilp}
  A\otimes^L_{
\Z[M]} \Z[N] \to A\otimes_{\Z[M]} \Z[N]
 \end{equation}
 is an equivalence.
Now, we have that for all $d$
\begin{align*}
S\otimes_{\Z[M]}\Z[N^{\oplus_M^{\rm sat} (d + 1)}]&\cong S\otimes_{A}A\otimes_{\Z[M]}\Z[N]\otimes_\Z\Z[(N^{\rm gp}/M^{\rm gp})^{\oplus d}]\\
&\overset{(*1)}{\simeq} S\otimes_{A}(A\otimes_{\Z[M]}^L\Z[N]\otimes_\Z^L\Z[(N^{\rm gp}/M^{\rm gp})^{\oplus d}])\\
&\overset{(*2)}{\simeq} S\otimes_{A}^L(A\otimes_{\Z[M]}^L\Z[N^{\oplus_M^{\rm sat} (d + 1)}])\\
&\simeq S\otimes_{\Z[M]}^L\Z[N^{\oplus_M^{\rm sat} (d + 1)}]
\end{align*}
where $(*1)$ is \eqref{eq:clubsuit-pnilp}
and the fact that $\Z[N]$ is flat over $\Z$ and $(*2)$ is Lemma \ref{lem:illusie-trick} together with the flatness of $S$ over $A$ (remember that now we are assuming that $p$ is nilpotent, so $p$-completely flat is flat). 
Observe that 
\[
S \otimes^L_{\Z[M]} \Z[N^{\oplus_M^{\rm sat} (d + 1)}] \simeq  S \otimes^L_{\Z[Q]} \Z[Q\oplus_M N^{\oplus_M^{\rm sat} (d + 1)}],
\]
so that we may consider the animated pre-log rings $(S \otimes^L_{{\Z}[M]} \Z[N^{\oplus_M^{\rm sat} (d + 1)}], Q \oplus_M N^{\oplus_M^{\rm sat} (d + 1)})$. 

By the transitivity sequence for the composites $(R,P) \to (S, M) \to (S, Q)$ and $(R,P)\to (S\otimes_{\Z[M]}\Z[N^{\oplus_M^{\rm sat} (d + 1)}], N^d)\to (S\otimes_{\Z[M]}\Z[N^{\oplus_M^{\rm sat} (d + 1)}], Q\oplus_M N^{\oplus_M^{\rm sat} (d + 1)})$, we have a commutative diagram:
\begin{equation}\label{eq:diagram-technical-descent-lqrsp}
\begin{tikzcd}
\L_{(S,M)/(R,P)}\arrow[r, "\simeq"] \arrow[d] & \varlim_{[m]\in \Delta} \L_{(S\otimes^L_{\Z[M]} \Z[N^{\oplus_M^{\rm sat} (m + 1)}], N^{\oplus_M^{\rm sat} (m + 1)})/(R,P) } \arrow[d] \\
\L_{(S,Q)/(R,P)}\arrow[r] \arrow[d]  & \varlim_{[m]\in \Delta}   \L_{(S\otimes^L_{\Z[M]} \Z[N^{\oplus_M^{\rm sat} (m + 1)}], Q\oplus_M N^{\oplus_M^{\rm sat} (m + 1)})/(R,P) } \arrow[d] \\
\L_{(S,Q)/(S,M)}\arrow[r] &\varlim_{[m]\in \Delta} \L_{(S\otimes^L_{\Z[M]} \Z[N^{\oplus_M^{\rm sat} (m + 1)}], Q\oplus_M N^{\oplus_M^{\rm sat} (m + 1)})/S\otimes^L_{\Z[M]} \Z[N^{\oplus_M^{\rm sat} (m + 1)}])}
\end{tikzcd}
\end{equation}
where the top horizontal arrow is an equivalence in light of Lemma \ref{lem:sat-descent-cotangent}, so it is enough to check that the bottom horizontal arrow of \eqref{eq:diagram-technical-descent-lqrsp} is an equivalence. By \cite[Corollary 3.9]{BLPO-HKR}, we have a commutative diagram where the columns are fiber sequences 
\begin{equation}\label{eq:diagram-technical-descent-lqrsp-2}
\begin{tikzcd}[column sep = tiny]
S\otimes_\Z Q^{\rm gp}/M^{\rm gp} \arrow[r]\ar[d] &\varlim_{[m]\in \Delta} S\otimes_\Z (Q\oplus_M N^{\oplus_M^{\rm sat} (m + 1)})^{\rm gp}/(N^{\oplus_M^{\rm sat} (m + 1)})^{\rm gp}\ar[d]\\
\L_{(S,Q)/(S,M)}\arrow[r]\ar[d] &\varlim_{[m]\in \Delta} \L_{(S\otimes^L_{\Z[M]} \Z[N^{\oplus_M^{\rm sat} (m + 1)}], Q\oplus_M N^{\oplus_M^{\rm sat} (m + 1)})/S\otimes^L_{\Z[M]} \Z[N^{\oplus_M^{\rm sat} (m + 1)}])}\ar[d]\\
\L_{S/S\otimes_{\Z[M]}^L\Z[Q]}\arrow[r] &\varlim_{[m]\in \Delta} \L_{(S\otimes^L_{\Z[M]} \Z[N^{\oplus_M^{\rm sat} (m + 1)}]/S\otimes^L_{\Z[M]} \Z[N^{\oplus_M^{\rm sat} (m + 1)}]\otimes^L_{\Z[N^{\oplus_M^{\rm sat} (m + 1)}]}\Z[Q\oplus_M N^{\oplus_M^{\rm sat} (m + 1)}])}
\end{tikzcd}
\end{equation}
The map $N^{\oplus_M^{\rm sat} (d + 1)} \to Q\oplus_M N^{\oplus_M^{\rm sat} (d + 1)}$ is constructed as  $M\oplus_M N^{\oplus_M^{\rm sat} (d + 1)} \to Q\oplus_M N^{\oplus_M^{\rm sat} (d + 1)}$, with the identity on the second factor, so by Lemma \ref{lem:illusie-trick} we have \begin{align*}
(Q\oplus_M N^{\oplus_M^{\rm sat} (d + 1)})^{\rm gp}/(N^{\oplus_M^{\rm sat} (d + 1)})^{\rm gp} &\cong (Q^{\rm gp}\oplus(N^{\oplus_M^{\rm sat} (d + 1)})^{\rm gp}/M^{\rm gp})/(0\oplus N^{\oplus_M^{\rm sat} (d + 1)})^{\rm gp} \\ &\cong (Q^{\rm gp}\oplus(N^{\oplus_M^{\rm sat} (d + 1)})^{\rm gp}/(0\oplus N^{\oplus_M^{\rm sat} (d + 1)})^{\rm gp}) /M^{\rm gp} \\ &\cong Q^{\rm gp}/M^{\rm gp},
\end{align*}
so the top horizontal map of \eqref{eq:diagram-technical-descent-lqrsp-2} is an equivalence. Finally, we have that the homotopy pushout diagram\[
\begin{tikzcd}
S\ar[r]\ar[d]&S\otimes_{Z[M]}^L\Z[Q]\ar[d]\\
S\otimes_{\Z[M]}^L\Z[N^{\oplus_M^{\rm sat} (d + 1)}]\ar[r]&S\otimes_{Z[M]}\Z[Q]\otimes_{\Z[M]}^L\Z[N^{\oplus_M^{\rm sat} (d + 1)}]
\end{tikzcd}
\]
implies that the map
 
\[
\L_{S\otimes_{\Z[M]}^L\Z[Q]/S}\to \varlim_{[m]\in \Delta} \L_{S\otimes^L_{\Z[M]} \Z[N^{\oplus_M^{\rm sat} (m + 1)}]\otimes^L_{\Z[N^{\oplus_M^{\rm sat} (m + 1)}]}\Z[Q\oplus_M N^{\oplus_M^{\rm sat} (m + 1)}]/S\otimes^L_{\Z[M]} \Z[N^{\oplus_M^{\rm sat} (m + 1)}]}
\]
equals the map\[
\L_{S\otimes_{\Z[M]}^L\Z[Q]/S}\to \varlim_{[m]\in \Delta} \Z[N^{\oplus_M^{\rm sat} (m + 1)}]\otimes^L_{\Z[M]}\L_{S\otimes_{\Z[M]}^L\Z[Q]/S}
\]
By Remark \ref{rmk:Niziol_contraction}, the map $\Z[M]\to \Z[N^{\oplus_M^{\rm sat} (\bullet + 1)}]$ is a homotopy equivalence, so the map above is an equivalence. By the transitivity sequences of the compositions\begin{align*}
S\to &S\otimes_{\Z[M]}^L\Z[Q]\to S\\ S\otimes_{\Z[M]}^L\Z[N^{\oplus_M^{\rm sat} (d + 1)}]\to &S\otimes_{\Z[M]}\Z[Q]\otimes_{\Z[M]}^L\Z[N^{\oplus_M^{\rm sat} (d + 1)}]\to S\otimes_{\Z[M]}^L\Z[N^{\oplus_M^{\rm sat} (d + 1)}],
\end{align*}
we conclude that the bottom horizontal arrow of \eqref{eq:diagram-technical-descent-lqrsp-2} is an equivalence too, so we conclude that the bottom horizontal arrow of \eqref{eq:diagram-technical-descent-lqrsp} is an equivalence too. This concludes the proof. 
\end{proof}
\end{thm}

\begin{lemma}
\label{lem:p-torsionfree abelian group}

Let $G$ be a $p$-torsionfree abelian group.
Then the Frobenius $\F_p[G]\to \F_p[G]$ is flat.
Hence the induced map $\F_p[G]\to \F_p[G_\perf]$ is flat.
\end{lemma}
\begin{proof}
Since a colimit of a directed system of flat modules is flat, and since $G$ is the filtered colimit of its finitely generated subgroups,
we may assume that $G$ is finitely generated.
We further reduce to the cases of $G=\Z$ and $G=\Z/n$ with $n$ prime to $p$.
The claim is clear for these cases.
\end{proof}

\begin{example}\label{ex:condition-clubsuit} Here we give some examples of $(A,M)$ and $M\to N$ such that $A\otimes^L_{
\Z[M]} \Z[N]$ is $p$-completely discrete.
\begin{enumerate}
\item    Let $R$ be a ring. Then for $M\to N$ any map of monoids, we have that the following squares are (homotopy) cocartesian:\[
\begin{tikzcd}
   \Z \ar[r] \ar[d] & \Z[M] \ar[r] \ar[d] & \Z[N] \ar[d] \\ R \ar[r] & R[M] \ar[r] & R[N].
\end{tikzcd}
    \]
    This implies that we have equivalences\[
    R[N]\simeq R[M]\otimes_{\Z[M]}\Z[N]\simeq R[M]\otimes_{\Z[M]}^L\Z[N]
    \]
In particular, $(R[M],M)$ is such that $R\otimes_{\Z[M]}^L \Z[N]$ has $p$-complete Tor-amplitude in degree $0$ for all $M\to N$.
\item Let $(A,M)$ be a saturated pre-log $R$-algebra
such that $M^\mathrm{gp}$ is $p$-torsionfree.
Then $M\to M_{\rm perf}$ is Kummer and $M_{\rm perf}$ is saturated, as observed in Remark \ref{rmk:perf-kummer}. 
Assume that there exists a section $\ol{M}\to M$ of the quotient map $M\to \ol{M}$ such that $R[\ol{M}]\to A$ is $p$-completely flat, where we write $\ol{M}$ for $M/M^*$ as usual. 
Then
\begin{align*}
A/p\otimes_{\Z[\ol{M}]}\Z[\ol{M}_{\rm perf}]&\simeq A/p\otimes_{R/p[\ol{M}]}R/p[\ol{M}_{\rm perf}] \\ &\simeq    A/p\otimes_{R/p[\ol{M}]}^LR/p[\ol{M}_{\rm perf}]\simeq A/p\otimes_{\Z[\ol{M}]}^L\Z[\ol{M}_{\rm perf}].
\end{align*}
This implies that $A\otimes^L_{\Z[\ol{M}]}\Z[\ol{M}_{\rm perf}]$ is $p$-completely discrete.

Assume further that $M^*$ is $p$-torsionfree.
We have isomorphisms
\begin{gather*}
A\otimes_{\Z[\ol{M}]}^L \Z[\ol{M}_\perf]\simeq 
A\otimes^L_{\Z[M]}\Z[\ol{M}_\perf \oplus M^*],
\\
A\otimes_{\Z[M]}^L \Z[M_\perf]
\simeq
A\otimes_{\Z[M]}^L \Z[\ol{M}_\perf \oplus (M^*)_\perf].
\end{gather*}
Together with Lemma \ref{lem:p-torsionfree abelian group},
we deduce that $A\otimes^L_{\Z[M]}\Z[M_\perf]$ is $p$-completely discrete.
\item   Let $(R,\N)$ be the pre-log ring with structure map $m_0\colon \N\to R$ given by \[
r\mapsto \begin{cases} 1 &\mathrm{if } \:\:\: r=0\\ 0&\mathrm{otherwise}.\end{cases}
\] This factors through a map $\N_{\rm perf}\to R$. Let $(A,M)$ be a saturated pre-log $(R,\N)$-algebra such that $A$ is nonzero. Notice that the commutativity of the diagram \[
\begin{tikzcd}
    \N\ar[r,"f"]\ar[d,"m_0"]&M\ar[d,"m"]\\
    {R}\ar[r,"\phi"]&A
\end{tikzcd}
\]
of monoids implies that the map $\N\to M$ is injective. Indeed, if $r\in \N$ such that $f(r)=1$ in $M$, then $1=mf(r)=\phi m_0(r)$, so $r=0$. 

Let us moreover assume that there exists a section $\ol{M}\to M$ of the quotient map $M\to \ol{M}$ factoring $f$ such that $R\otimes_{\Z[\N]}\Z[\ol{M}] \to A$ is $p$-completely flat and that $\mathbb{N}\to M$ is a saturated morphism of monoids. Let $M':= \ol{M}\oplus_{\N}\N_{\rm perf}\simeq \colim_{\id\otimes F_{\N}} \ol{M}\oplus_{\N}\N$. Since it is a filtered colimit of saturated monoids it is again saturated by \cite[Proposition I.1.3.6]{ogu}. Moreover, we have that  $M'_{\rm perf}\simeq \ol{M}_{\rm perf}$, which is saturated and $M'\to \ol{M}_{\rm perf}$ is Kummer by Remark \ref{rmk:perf-kummer}. Also, the map $\N_{\rm perf}\to M'$ is injective and saturated since the colimit is filtered, so by \cite[Proposition 4.1]{katolog} the map $\Z[\N_{\rm perf}]\to \Z[M']$ is flat. Finally, we have that $R\otimes_{\Z[\N]}\Z[\ol{M}]\cong R\otimes_{\Z[\N_{\rm perf}]}\Z[M']$, so $R\otimes_{\Z[\N_{\rm perf}]}\Z[M']\to A$ is also $p$-completely flat. Putting everything together, we have that \[
\begin{aligned}
A/p\otimes_{\Z[M']}\Z[\ol{M}_{\rm perf}] &\cong A/p\otimes_{R/p\otimes_{\Z[\N_{\rm perf}]}\Z[M']}R/p\otimes_{\Z[\N_{\rm perf}]}\Z[M']\otimes_{\Z[M']}\Z[\ol{M}_{\rm perf}] \\
&\overset{(*)}{\simeq}  A/p\otimes_{R/p\otimes^L_{\Z[\N_{\rm perf}]}\Z[M']}^LR/p\otimes_{\Z[\N_{\rm perf}]}^L\Z[M']\otimes_{\Z[M']}^L\Z[\ol{M}_{\rm perf}]\\
&\simeq A/p\otimes_{\Z[M']}^L\Z[\ol{M}_{\rm perf}],
\end{aligned}
\]
where the derived tensors that appear in $(*)$ follow respectively from the assumption that $A$ is $p$-completely flat over $R\otimes_{\Z[\N_{\rm perf}]}\Z[M']$, from the fact that $\Z[\N_{\rm perf}]\to \Z[M']$ is flat, and from (1) above. From this, we deduce that $A\otimes_{\Z[M']}^L\Z[\ol{M}_{\rm perf}]$ is $p$-completely discrete.

Assume further that $M^*$ is $p$-torsionfree.
As in (2),
using Lemma \ref{lem:p-torsionfree abelian group},
we can deduce that 
$A\otimes_{\Z[M\oplus_\N \N_\perf]}^L\Z[M_{\rm perf}]$ is $p$-completely discrete.

\end{enumerate}
\end{example}

\subsection{Saturated descent for prismatic cohomology}

\begin{prop}\label{prop:semipf-invariance-prismatic}
Let $A$ be a quasisyntomic ring. For all pre-log rings $(A, M)$ with $M$ semiperfect, the canonical map $A \to (A, M)$ induces an isomorphism \[
\Fil_\rN^{\geq m}\cPrism_{A}\simeq \Fil_\rN^{\geq m}\cPrism_{(A,M)}.
\]
\begin{proof}
 Consider a quasisyntomic cover $A\to S$ with $S\in \QRSPerfd$. As observed in Corollary \ref{semiperfect.5}, $(A,M)\to (S,M)$ is a quasisyntomic cover. It thus suffices to check that $\Fil_\rN^{\geq m}\cPrism_{S/R}\simeq \Fil_\rN^{\geq m}\cPrism_{(S,M)/R}$ for $R \to S$ a map with $R$ perfectoid. 
By the conjugate filtration, it is enough to check it on the graded pieces, so it is enough to check that the map\[
(\bigwedge^i\mathbb{L}_{S/R})_p^{\wedge}\to (\bigwedge^i\mathbb{L}_{(S,M)/R})_p^{\wedge}
\]
is an equivalence. This follows directly from Proposition \ref{prop:semipf-invariance-cotangent}.
\end{proof}
\end{prop}

\begin{lemma}\label{lem:free-quasismooth}
    Let $M$ be any monoid and let $A$ be $p$-complete with bounded $p^{\infty}$-torsion. Then $\L_{(A\langle M \rangle,M)/A}\simeq A\langle M\rangle \otimes_\Z^L M^{\rm gp}[0]$ after derived $p$-completion. In particular, the map $A\to (A\langle M \rangle,M)$ is log quasisyntomic.
\end{lemma}

\begin{proof}
Since $A$ has bounded $p^\infty$-torsion and $A[M]$ is a free $A$-module, $A\to A\langle M \rangle$ is $p$-completely faithfully flat by \cite[Lemma 4.4]{BMS2}, and it is classically $p$-complete concentrated in degree zero with bounded $p^\infty$-torsion by \cite[Lemma 4.7]{BMS2}. By \cite[Remark 4.8]{BLPO-BMS2}, we have \[
    \L_{(A\langle M \rangle,M)/A}\simeq A\langle M \rangle\otimes_{A[M]}^L\L_{(A[M],M)/A}
    \]
after derived $p$-completion. On the other hand, by \cite[Proposition 3.8]{BLPO-HKR} we have that 
\begin{align*}
 \L_{(A\langle M \rangle,M)/A} & \simeq A\langle M \rangle\otimes_{A[M]}^L(A[M]\otimes^L_\Z M^{\rm gp}[0]) \\
 &\simeq A\langle M \rangle\otimes_\Z^L M^{\rm gp}[0].
\end{align*}
after derived $p$-completion. Since $A\to A\langle M\rangle$ is $p$-completely faithfully flat, we have that 
$A\langle M\rangle\otimes_{A}^LA/p$ has $p$-complete Tor amplitude in degree $0$, so\[
(A\langle M\rangle\otimes^L_{\Z}M^{\rm gp})\otimes_{A}^LA/p \simeq (A\langle M\rangle\otimes_{A}A/p) \otimes^L_{\Z}M^{\rm gp}
\]
has Tor amplitude in $[-1,0]$ (as the Tor amplitude of any abelian group is $[-1,0]$), which concludes the proof.
\end{proof}

\begin{lemma}\label{lem:cotimes-qsyn} 
 
Let $(B,P)$ be a pre-log ring such that $B$ is $p$-complete with bounded $p^\infty$-torsion.
Let $(A,M)\in \lQSyn_{(B,P)}$ with $M$ saturated and $M^{gp}$ $p$-torsionfree.
    Assume that there is a section $\ol{M}\to M$ of the quotient map $M\to \ol{M}$ factoring $P\to M$,
    and let $\ol{M}\to N_0$ and $M\oplus_{\ol{M}} N_0\to N$ be maps of  saturated monoids   such that the following conditions hold:
    \begin{itemize}
        \item[(i)] $\Z_p\langle P\rangle \to \Z_p\langle N_0\rangle$ is $p$-completely flat.
        \item[(ii)] $B\cotimes_{\Z_p\langle P\rangle }\Z_p\langle \ol{M}\rangle \to A$ is $p$-completely flat.
        \item[(iii)] $A\cotimes_{\Z_p\langle M \rangle}^L \Z_p \langle N \rangle$ is $p$-completely discrete.
        \item[(iv)] $\Z_p\langle M\oplus_{\ol{M}} N_0\rangle \to \Z_p\langle N\rangle$ is $p$-completely flat.
    \end{itemize}
    Then $(A\cotimes_{\Z_p\langle M\rangle }\Z_p\langle N\rangle,N)\in \lQSyn_{(B,P)}$.
    \begin{proof}
 By $(i)$, we have that $B\to B\cotimes_{\Z_p\langle P\rangle} \Z_p\langle N_0\rangle$ is $p$-completely flat, so the latter has bounded $p^\infty$-torsion by \cite[Corollary 4.8 (1)]{BMS2}. Moreover, by $(ii)$ and $(iv)$, the composite map $B\cotimes_{\Z_p\langle P\rangle} \Z_p\langle N_0\rangle\to A\cotimes_{\Z_p\langle \ol{M}\rangle} \Z_p\langle N_0\rangle \to A\cotimes_{\Z_p\langle M\rangle} \Z_p\langle N\rangle$ is $p$-completely flat, so $A\cotimes_{\Z_p\langle M\rangle} \Z_p\langle N\rangle$ has bounded $p^\infty$-torsion.
    By $p$-completing the sequence of \cite[Corollary 3.10]{BLPO-HKR} and by \cite[Remark 4.8]{BLPO-BMS2}, we obtain a cofiber sequence of the form
    \[
    A\cotimes_{\Z_p}^L (M^{\rm gp})_p^\wedge\to (\L_{(A,M)/\Z_p})_p^\wedge\to (\L_{A/\Z_p\langle M\rangle})_p^\wedge.
    \]
    Since $(A,M)\in \lQSyn_{(B,P)}$, $(\L_{(A,M)/\Z_p})_p^\wedge$ has $p$-complete Tor amplitude $[-1,0]$ in $\cD(A)$, and since $M^{\rm gp}$ is $p$-torsionfree, the $A$-module $A\cotimes_{\Z_p}^L (M^{\rm gp})_p^\wedge$ is  $p$-completely flat.
    We conclude that $\L_{A/\Z_p\langle M\rangle}$ has $p$-completed Tor amplitude $[-1,0]$ in $\cD(A)$.
    Thus the same holds for the pushout $\L_{A/\Z_p\langle M \rangle}\cotimes^L_A (A\cotimes_{\Z_p\langle M \rangle}\Z_p\langle N\rangle)$
     along the map $A\to A\cotimes_{\Z_p\langle M\rangle}\Z_p\langle N\rangle$ by \cite[Lemma 4.5]{BMS2}.
     By condition $(iii)$ and base-change for the cotangent complex, this is equivalent to $\L_{A\cotimes_{\Z_p\langle M \rangle}\Z_p\langle N\rangle/\Z_p\langle N \rangle}$. Using the cofiber sequence \[
    (A\cotimes_{\Z_p\langle M \rangle}\Z_p\langle N\rangle)\cotimes_{\Z_p}^L (N^{\rm gp})_p^\wedge\to \L_{(A\cotimes_{\Z_p\langle M \rangle}\Z_p\langle N\rangle,N)/\Z_p} \to \L_{A\cotimes_{\Z_p\langle M \rangle}\Z_p\langle N\rangle/\Z_p\langle N \rangle},
    \]
    we find that $\L_{(A\cotimes_{\Z_p\langle M \rangle}\Z_p\langle N\rangle,N)/\Z_p}$ has $p$-complete Tor amplitude $[-1,0]$ in $\cD(A\cotimes_{\Z_p\langle M\rangle}\Z_p\langle N\rangle)$.
    \end{proof}
\end{lemma}

\begin{rmk}\label{rmk:sharp-p-torsionfree}
If $M$ is sharp and saturated, then $M^{\rm gp}$ is torsionfree by \cite[Proposition I.1.3.5 (2)]{ogu}.
\end{rmk}

\begin{rmk}\label{rmk:integral cover}
Let $(A,M)\in \lQSyn$. Then there exists a log quasisyntomic cover $(A,M)\to (S,Q)$ with $(S,Q)\in \lQRSPerfd$ such that $A\to S$ is a quasisyntomic cover and $M\to Q$ is integral.
Indeed, if $(S,Q)$ is chosen as \cite[Proof of Proposition 4.15]{BLPO-BMS2}, then $A\to S$ (resp.\ $M\to Q$) is a pushout of some quasisyntomic cover $F\to F_\infty$ (resp.\ integral map of integral monoids $P\to P_\infty$).
\end{rmk}

\begin{lemma}\label{lem:sat-descent-prismatic-general}
Let $(R, P)$ be a pre-log ring with $R$ perfectoid and $P$ semiperfect, valuative.
Let $(A,M)\in \lQSyn_{(R,P)}$ with $M$ saturated and $M^\mathrm{gp}$ $p$-torsionfree,  and $P\to M$ exact.  Assume that there is a section $\ol{M}\to M$ of the quotient map $M\to \ol{M}$ factoring $P\to M$ such that $R\cotimes_{\Z_p\langle P\rangle}\Z_p\langle \ol{M}\rangle \to A$ is $p$-completely flat.
Let
$\ol{M} \to N_0$ and $M\oplus_{\ol{M}}N_0 \to N$ be maps of saturated monoids such that $N_0$ and $N$ are semiperfect, 
$A\cotimes_{R\langle M\rangle}^LR\langle N\rangle$ is $p$-completely discrete,
and $\Z_p\langle M\oplus_{\ol{M}} N_0\rangle \to \Z_p\langle N\rangle$ is $p$-completely flat.
Assume also that the map $\ol{M}\to N_0$ is exact and the induced map $M\to N$ is Kummer.
Then the natural $\phi$-equivariant map of $E_\infty$-rings
\begin{equation}\label{eq:sat-descent-prismaic}
    \cPrism_{(A,M)/R}\to \varlim_{
    \Delta
    } (\cPrism_{(A\cotimes_{R\langle M\rangle }R\langle N^{\oplus_M^{\rm sat} (\bullet + 1)} \rangle,N^{\oplus_M^{\rm sat} (\bullet + 1)} )/R})
\end{equation}
is an equivalence, compatible with the Nygaard filtrations.
\end{lemma}

\begin{proof} 
 The map is one of $E_\infty$-rings, so it is enough to show that the map is an equivalence of objects of $\cD(\Lambda)$, as the forgetful functor from $E_\infty$-rings is conservative. We again write $N^\bullet$ for $N^{\oplus_M^{\rm sat} (\bullet + 1)}$ to shorten the notation. 
 Since $P\to \ol{M}$ and $\ol{M}\to N_0$ are exact, the composite $P\to N_0$ is exact.
Since $P$ is valuative, it is local by \cite[Proposition I.4.2.1]{ogu}, and so the map $\Z[P] \to \Z[N_0]$ is flat by \cite[Propositions I.4.6.3(5) and I.4.6.7]{ogu}. Combining this with the assumptions,
we conclude by Lemma \ref{lem:cotimes-qsyn} that $(A\cotimes_{R\langle M\rangle}R\langle N\rangle,N)\in \lQSyn_{(R,P)}$.
    Since $N$ is saturated and semiperfect, Corollary \ref{semiperfect.5} implies that $A\cotimes_{R\langle M\rangle}R\langle N\rangle$ is quasisyntomic.
    By Remark \ref{rmk:integral cover} and quasisyntomic descent \cite[Theorem 4.22 and \S 7.4]{BLPO-BMS2}, it is enough to show that for $(A,M)\to (S,Q)$ a log quasisyntomic cover with $(S,Q)\in \lQRSPerfd_R$ such that $A\to S$ is a quasisyntomic cover and $M\to Q$ is integral, the map
    \begin{equation}\label{eq:saturated-descent-Nygaard}
    \Fil_\rN^{\geq n}\cPrism_{(S,Q)/R}\to \varlim_{
\Delta^{\rm op}}(\Fil_\rN^{\geq n}\cPrism_{(S\cotimes_{R\langle M\rangle }R\langle N^{\bullet}\rangle, Q\oplus_M N^{\bullet})/R})     
    \end{equation}
    is an equivalence. By Lemma \ref{lem:illusie-trick}, we have that $N^d\simeq N\oplus (N^{\rm gp}/M^{\rm gp})^{\oplus d-1}$, which induces an isomorphism $R\langle N^d\rangle\simeq R\langle N\rangle \cotimes_R R\langle (N^{\rm gp}/M^{\rm gp})^{\oplus d-1}\rangle$. We thus have that
    \[
    A\cotimes_{R\langle M\rangle}R\langle N^{ d} \rangle \cong (A\cotimes_{R\langle M\rangle}R\langle N\rangle )\cotimes_R R\langle (N^{\rm gp}/M^{\rm gp})^{\oplus d-1}\rangle,
\]
which is quasisyntomic over $R\langle (N^{\rm gp}/M^{\rm gp})^{d-1}\rangle$ by \cite[Lemma 4.16(2)]{BMS2}. 
In particular, it is quasisyntomic over $R$ by Lemma \ref{lem:free-quasismooth}. 
 
Since $A\to S$ is a quasisyntomic cover, $S\cotimes_{R\langle M\rangle}R\langle N^d\rangle$ is quasisyntomic again by \cite[Lemma 4.16(2)]{BMS2}. Moreover, $S/p$ and $Q$ are semiperfect by the assumption that $(S,Q)\in \lQRSPerfd_R$ (see \cite[Definition 4.7]{BLPO-BMS2}), and $N^d\cong N\oplus (N^{\rm gp}/M^{\rm gp})^{\oplus d-1}$ is semiperfect since $N$ is semiperfect by assumption, so $N^{\rm gp}$ is also semiperfect. This implies that $R\langle N^d\rangle$ is quasiregular semiperfectoid. We thus have that $S\cotimes_{R\langle M\rangle }R\langle N^d\rangle/p$ and $Q \oplus_M N^d$ are semiperfect, so that
\[(S\cotimes_{R\langle M\rangle }R\langle N^d\rangle,Q \oplus_M N^d)\in \lQRSPerfd_R.\]
By \cite[Proposition 7.4]{BLPO-BMS2}, 
we are reduced to check \eqref{eq:saturated-descent-Nygaard} on the finite filtration of the graded pieces of $\Fil_\rN^{\geq n}$. That is, we are reduced to checking that the map \[
(\bigwedge^i\L_{(S,Q)/R})_p^{\wedge}\to \varlim_{\Delta}(\bigwedge^i\L_{(S\otimes_{R\langle M\rangle}R\langle N^{\bullet}\rangle,Q\oplus_M N^{\bullet})/R})_p^{\wedge},
\]
 is an equivalence for all $i$. But this is precisely Theorem \ref{thm:sat-descent-nonderived}.
\end{proof}
\begin{thm}\label{thm:main-sat-descent-prism}
Let $B$ be a $p$-complete ring with bounded $p^\infty$-torsion and let $(A,M)\in \lQSyn_B$ with $M$ saturated and $M^{\rm gp}$ $p$-torsionfree.
Assume that there exists a section $\ol{M}\to M$ of the quotient map $M\to \ol{M}$ such that $B\langle \ol{M}\rangle \to A$ is $p$-completely flat.
Then there is a $\phi$-equivariant equivalence  of $E_\infty$-rings:
\begin{equation}\label{eq:sat-descent-prismatic-nolog}
    \cPrism_{(A,M)}\simeq \varlim_{\Delta}
     (\cPrism_{A\cotimes_{B\langle M\rangle}B\langle M_\mathrm{perf}^{\oplus_M^{\rm sat} (\bullet + 1)}\rangle})
\end{equation}
compatible with the Nygaard filtrations.
\begin{proof}
As observed in Example \ref{ex:condition-clubsuit}, $M\to M_{\rm perf}$ is a Kummer map of saturated monoids and $A\cotimes^L_{\Z[M]}\Z[M_{\rm perf}]$ is $p$-completely discrete.
Furthermore, Lemma \ref{lem:p-torsionfree abelian group} implies that $\Z_p\langle M\oplus_{\ol{M}} \ol{M}_\perf\rangle \to \Z_p\langle M_\perf\rangle$ is $p$-completely flat.
 By Lemma \ref{lem:cotimes-qsyn}, since $P=\triv$ in this case, $(A\cotimes_{B\langle M\rangle}B\langle M_{\rm perf}\rangle,M_{\rm perf})\in \lQSyn_B$. As in the proof of Lemma \ref{lem:sat-descent-prismatic-general}, we consider $(A,M)\to (S,Q)$ a log quasisyntomic cover with $(S,Q)\in \lQRSPerfd$ and let $R\to S$ be a map with $R$ perfectoid. Then we have an equivalence \eqref{eq:saturated-descent-Nygaard} \[
\cPrism_{(S,Q)/R}\simeq \varlim_{\Delta} (\cPrism_{(S\cotimes_{R\langle M\rangle}R\langle M_\mathrm{perf}^{\oplus_M^{\rm sat} (\bullet + 1)}\rangle,Q\oplus_M M_\mathrm{perf}^{\oplus_M^{\rm sat} (\bullet + 1)})/R}),
\]
compatibly with the Nygaard filtration, which by quasisyntomic descent gives an equivalence\[
\cPrism_{(A,M)}\simeq \varlim_{\Delta} (\cPrism_{(A\cotimes_{B\langle M\rangle}B\langle M_\mathrm{perf}^{\oplus_M^{\rm sat} (\bullet + 1)}\rangle,M_\mathrm{perf}^{\oplus_M^{\rm sat} (\bullet + 1)})}).
\] 
By Proposition \ref{prop:semipf-invariance-prismatic} we have \[
\cPrism_{(A\cotimes_{B\langle M\rangle }B\langle M_\mathrm{perf}^{\oplus_M^{\rm sat} (m + 1)}\rangle,M_\mathrm{perf}^{\oplus_M^{\rm sat} (m + 1)})}\simeq \cPrism_{A\cotimes_{B\langle M\rangle }B\langle M_\mathrm{pref}^{\oplus_M^{\rm sat} (m + 1)}\rangle}
\]
 for all $m\in \Delta$, noting that each $M_\mathrm{perf}^{\oplus_M^{\rm sat} (m + 1)}$ is saturated (by construction) and semiperfect. This concludes the proof.
\end{proof}

\end{thm}

\begin{example}\label{ex:log-smooth-descent}
 Let $\kX$ be a log smooth formal scheme over $B$,
 where $B$ is a noetherian $p$-complete ring with bounded $p^\infty$-torsion. Then by \cite[Theorems III.2.5.5 and IV.3.3.3]{ogu} \'etale locally on $\ul{\kX}$ there is a chart $(\Spec(A),M_0)$ neat at a point $x$ with $M_0$ sharp and saturated (see Remark \ref{rmk:crys-global} later) and $B[M_0]\to A$ flat. Furthermore, working again locally,
 by \cite[Lemma 1.3.3]{Tsu99}, we may also assume that $M:=\Gamma(\kX,\cM_{\kX})$ is a chart of $\kX$ and the composite $M_0\to M\to \ol{M}$ is surjective. The composite $M_0\to \ol{M}\to \ol{\cM}_{\kX,x}$ is an isomorphism since the chart is neat at $x$ (see \cite[Definition II.2.3.1 (2-b)]{ogu}), so we have $M_0\cong \ol{M}$. In particular, locally $\kX$ satisfies the condition in Theorem \ref{thm:main-sat-descent-prism}. 
\end{example}

\begin{rmk}
    Notice that the right-hand side of \eqref{eq:sat-descent-prismatic-nolog}
    is the non-logarithmic Nygaard-completed prismatic cohomology, but the system is not a \v Cech nerve as the product is different!
\end{rmk}

\begin{rmk}
    A similar descent result has been proved in \cite{Koshikawa-Yao} for the Kummer \'etale cohomology with $p$ inverted. In particular one could deduce a comparison between syntomic and Kummer \'etale cohomology. We leave this to future work.
\end{rmk}

We finish by proving saturated descent for the non-completed version of prismatic cohomology. Let $R$ be a perfectoid ring and denote again by $\cPrism^{nc}_{-/R}$  the non-Nygaard complete prismatic cohomology of Construction \ref{constr:non-nygaard-complete}.

\begin{cor}\label{cor:sat-descent-de-Rham-prismnc}
   Let $R$ be a perfectoid ring and let $(A,M)\in \lQSyn_R$ with $M$ saturated and $M^{\rm gp}$ $p$-torsionfree.
   Assume that there exists a section $\ol{M}\to M$ of the quotient map $M\to \ol{M}$ such that $B\langle \ol{M}\rangle \to A$ is $p$-completely flat.
   Then there is an equivalence
   \[
    L\Omega_{(A,M)/R}\simeq \varlim_{[m]\in \Delta}L\Omega_{A\otimes_{R\langle M\rangle}R\langle M_\mathrm{perf}^{\oplus_M^{\rm sat} (m + 1)}\rangle/R}
    \]
    of $E_\infty$-rings, and a $\phi$-equivariant  equivalence \[
    \cPrism^{\rm nc}_{(A,M)/R}\simeq \varlim_{[m]\in \Delta}\Prism_{A\otimes_{\Z_p\langle M\rangle}\Z_p\langle M_\mathrm{perf}^{\oplus_M^{\rm sat} (m + 1)}\rangle/R},
    \]
    of $E_\infty$-rings, where the right-hand side is the prismatic cohomology of \cite{BSPrism}.
    \begin{proof}
    We begin by proving the statement for derived (log) de Rham cohomology. 
    
    Since $L\Omega_{(A,M)/R}$ is $p$-complete (by definition), it is enough to check that the canonical map to the limit is an equivalence modulo $p$. As in \cite[Example 5.12]{BMS2}, let $\Fil^{\rm conj}_{n}L\Omega_{-/R}\otimes^L_{\Z}\Z/p\Z$ denote the conjugate filtration (see \cite{Bhatt2012padicDD}, the definition is identical to the non-log case, and simply follows from the existence of a canonical filtration on the homotopy-colimit of a simplicial cosimplicial $A$-module). 
   We have that $\Fil^{\rm conj}_{-1}L\Omega_{-/R}\otimes^L_{\Z}\Z/p\Z = 0$, and the filtration is exhaustive, that is $\colim_n \Fil^{\rm conj}_{n}L\Omega_{-/R}\otimes^L_{\Z}\Z/p\Z\simeq L\Omega_{-/R}\otimes^L_{\Z}\Z/p\Z$. Moreover we have an identification $\mathrm{gr}_n^{\rm conj}L\Omega_{-/R}\otimes^L_{\Z}\Z/p\Z\simeq \bigwedge^n(\L_{-/R})[-n]\otimes^L_{\Z}\Z/p\Z$. By Proposition \ref{prop:semipf-invariance-cotangent} and Theorem \ref{thm:sat-descent-nonderived}, we deduce by induction that for all $n$ that\[
    \Fil^{\rm conj}_{n}L\Omega_{(A,M)/R}\otimes^L_{\Z}\Z/p\Z\simeq \lim_{[m]\in\Delta} \Fil^{\rm conj}_{n}L\Omega_{A\otimes_{R\langle M\rangle}R\langle M_\mathrm{perf}^{\oplus_M^{\rm sat} (m + 1)}\rangle/R}\otimes^L_{\Z}\Z/p\Z.
    \] 
    Since $A\otimes_{R\langle M\rangle}R\langle M_\mathrm{perf}^{\oplus_M^{\rm sat} (m + 1)}\rangle$ is quasisyntomic for all $m$ by Lemma \ref{lem:cotimes-qsyn}, we have that $\Fil^{\rm conj}_{n}L\Omega_{A\otimes_{R\langle M\rangle}R\langle M_\mathrm{perf}^{\oplus_M^{\rm sat} (m + 1)}\rangle/R}\otimes^L_{\Z}\Z/p\Z$ takes value in $\cD^{\geq -1}$ for all $n$,
 as each $\L_{-/R}$ has $p$-completed Tor amplitude $[-1,0]$: this implies that in this case, the direct product totalization coincides with the direct sum totalization (as in each cosimplicial degree there are only finitely many terms), and since the direct sum totalization commutes with filtered colimits we conclude.

    For the second isomorphism, we use the fact that $\cPrism^{\rm nc}_{-/R}$ is $\xi$-complete, so it is enough to check that the natural map\[
    \cPrism^{\rm nc}_{(A,M)/R}\longrightarrow \varlim_{[m]\in \Delta}\cPrism^{\rm nc}_{(A\otimes_{\Z_p\langle M\rangle}\Z_p\langle M_\mathrm{perf}^{\oplus_M^{\rm sat} (m + 1)}\rangle,M_\mathrm{perf}^{\oplus_M^{\rm sat} (m+ 1)})}
    \]
    is an equivalence modulo $\xi$, thus reducing to the case of $L\Omega$.  One only needs to notice that, since $M_{\rm perf}$ is perfect, we have 
    \[\cPrism^{\rm nc}_{(A\otimes_{\Z_p\langle M\rangle}\Z_p\langle M_\mathrm{perf}^{\oplus_M^{\rm sat} (m + 1)}\rangle,M_\mathrm{perf}^{\oplus_M^{\rm sat} (m + 1)})}\simeq \cPrism^{\rm nc}_{A\otimes_{\Z_p\langle M\rangle}\Z_p\langle M_\mathrm{perf}^{\oplus_M^{\rm sat} (m + 1)}\rangle}.\] This can also be checked by working modulo $\xi$, reducing the statement to derived de Rham cohomology, and then to the cotangent complex as before.
    Finally, the fact that (non logarithmic) $\cPrism^{\rm nc}_{-/R}$ agrees with $\Prism_{-/A_{\rm inf}(R)}$ is the content of  \cite[Theorem 13.1]{BSPrism}.
    
    \end{proof}
\end{cor}

\subsection{Saturated descent for de Rham--Witt cohomology and comparison}

For this subsection, we assume that $k$ is a perfect field equipped with an fs log structure $P$. Let $(A,M)\in \lSm_{(k,P)}$ with $P\to M$ saturated
(or equivalently by \cite[Corollary III.2.5.4]{ogu} $(k,P)\to (A,M)$ is of log-Cartier type).
The log de Rham--Witt complex $\WOmega_{(A,M)/(k,P)}$ is  equipped with the Nygaard filtration\[
    (\Fil_\rN^{\geq i}\WOmega_{(A,M)/(k,P)})^q :=\begin{cases}
p^{i-q-1}V\WOmega^q_{(A,M)/(k,P)}&{\rm if } \:\:\:q<i\\
\WOmega^q_{(A,M)/(k,P)}&{\rm if } \:\:\:q\geq i.
\end{cases}
    \]
By \cite[Proposition 8.2.1]{BLM}, since $\WOmega_{(A,M)/(k,P)}$ is a saturated Dieudonn\'e complex by \cite[Lemma 7.5]{Yao-logDRW},
there is an isomorphism
\begin{equation}
\label{eq:drw-graded}
        \gr^i_\rN(\WOmega_{(A,M)/(k,P)}) \simeq \tau^{\leq i}(\WOmega_{(A,M)/(k,P)}/p).
        \end{equation}

\begin{rmk}
As we will only be interested in the case of log smooth algebras of log-Cartier type, we will not use here the de Rham--Witt complex constructed in \cite{Yao-logDRW}, but just the fact that \eqref{eq:drw-graded} holds for the usual log de Rham--Witt complex. We note that  there is a gap in \cite[Lemma A.1]{Yao-logDRW}, as shown in Tsuji's counterexample in \cite[Section 3.6]{AHLS}. However, this lemma is only used in the globalization construction of the saturated de Rham--Witt complex, therefore it does not affect the validity of \cite[Lemma 7.5]{Yao-logDRW}.
\end{rmk}

Recall from \cite[\S 2.10]{BLPO-HKR} that the category $\mathrm{Poly}_{(k,P)}$ is defined as the category of polynomial pre-log rings over $(k,P)$. The objects are defined to be the free pre-log algebras: concretely, they are given as polynomial rings  $(k[T_{I\coprod J}],P\oplus \N^I)$, where $I$ and $J$ are finite sets and $\N^I$ is the pre-log structure sending the standard basis element $e_i\in \prod_{i\in I}\N$ to $T_{i}$ (see \cite[\S 2.2]{BLPO-BMS2}): by construction if $(A,M)\in \mathrm{Poly}_{(k,P)}$, then it is log smooth and saturated.

By left Kan extension from $\Poly_{(k,P)}$, we construct the derived de Rham--Witt complex with its Nygaard filtration on animated pre-log $(k,P)$-algebras similarly to \cite[Construction 9.2.5]{BLM}, so $\gr_N^{i}L\WOmega_{(-)/(k,P)}$ has a finite conjugate filtration with graded pieces $\bigwedge^{i}\L_{(-)/(k,P)}$. 
We immediately deduce the following result:

\begin{prop}\label{prop:semipf-invariance-dRW}
    Let $(A,M)$ be a quasisyntomic $(k,P)$-algebra with $M,P$ saturated and semiperfect. Then we have that\[
    \Fil_\rN^{\geq i}L\WOmega_{(A,M)/(k,P)}\simeq \Fil_\rN^{\geq i}L\WOmega_{A/k}.
    \]
    \begin{proof}
        As in Proposition \ref{prop:semipf-invariance-prismatic}, we reduce to the graded pieces and deduce the result from Proposition \ref{prop:semipf-invariance-cotangent}.    
        \end{proof}
\end{prop}
Let $\widehat{L\WOmega}_{(-)/(k,P)}$ be the Nygaard completion of $L\WOmega_{(-)/(k,P)}$. By reducing to the filtrations in the usual way, we have that $(\widehat{L\WOmega}_{(-)/(k,P)},\Fil_\rN^{\geq \bullet})$ is a quasisyntomic sheaf. The following is analogous to Lemma \ref{lem:sat-descent-prismatic-general}.
\begin{lemma}\label{lem:sat-descent-dRW-general}
    Let $(A,M)$ be a quasisyntomic saturated $(k,P)$-algebra and let $M\to N$ be a Kummer map of saturated monoids such that $N$ is semiperfect and $A\otimes^L_{k[M]}k[N]$ is discrete.
    Consider the \v Cech nerve $N^{\oplus_M^{\rm sat} (\bullet + 1)}$ of $M \to N$ in the category of saturated monoids.
Then the natural $\phi$-equivariant map of $E_\infty$-rings 
\begin{equation}\label{eq:sat-descent-drw}
    \widehat{L\WOmega}_{(A,M)/(k,P)}\to \varlim_{
    \Delta} (\widehat{L\WOmega}_{(A\cotimes_{k[M]}k[N^{\oplus_M^{\rm sat} (\bullet + 1)}],N^{\oplus_M^{\rm sat} (\bullet + 1)})/(k,P)})
\end{equation}
is an equivalence and compatible with the Nygaard filtrations.
\begin{proof}
    By reducing to the finite filtration of $\gr_N^{i}\widehat{L\WOmega}_{(A,M)/(k,P)}$, it boils down to \eqref{eq:reduction-to-cotan}. 
\end{proof}
\end{lemma}
Then the proof of the saturated descent follows immediately as in Theorem \ref{thm:main-sat-descent-prism}:
\begin{thm}\label{thm:main-sat-descent-dRW}
Let $k$ be a perfect field and let $(A,M)\in \lQSyn_k$ with $M$ saturated and $M^{\rm gp}$ $p$-torsionfree.
Assume that there is a section $\ol{M}\to M$ of the quotient map $M\to \ol{M}$ such that $k[\ol{M}]\to A$ is flat. Then there is a $\phi$-equivariant equivalence  of $E_\infty$-rings:
\begin{equation}\label{eq:sat-descent-drw-nolog}
    \widehat{L\WOmega}_{(A,M)/k}\simeq \varlim_{
    \Delta}  (\widehat{L\WOmega}_{A\otimes_{k[M]}k[M_\mathrm{perf}^{\oplus_M^{\rm sat} (\bullet + 1)}]/k})
\end{equation}
compatible with the Nygaard filtration. \qed
\end{thm}

We are now ready to prove our crystalline comparison results. We first deduce the comparison with $P=\triv$.

\begin{thm}\label{thm:crys_comparison}
Let $(A,M)\in \lQSyn_k$ with $M$ saturated and $M^{\rm gp}$ $p$-torsionfree.
Assume that there exists a section $\ol{M}\to M$ of the quotient map $M\to \ol{M}$ such that $k[\ol{M}]\to A$ is flat.
Then there is a $\phi$-equivariant equivalence of filtered $E_{\infty}$-rings
    \begin{equation}\label{eq:crystalline-comp}
        \widehat{L\WOmega}_{(A,M)/k}\simeq \cPrism_{(A,M)}
    \end{equation}
lifting the equivalence $\cPrism_{(A, M)}/\xi \simeq \widehat{L\Omega}_{(A,M)/k}$. This equivalence is functorial among maps $(A,M)\to (A',M')\in \lQSyn_k$ and $(A',M')$ as above. 
\end{thm}
\begin{proof}
By \cite[Theorem 8.17]{BMS2} (see also \cite[Theorem 9.1]{MathewTHH}), for every $A$ quasisyntomic over $k$ (with trivial log structure) there are functorial isomorphisms of filtered $E_\infty$-rings\[
\cPrism_A\simeq \widehat{L\WOmega}_{A/k}.
\]
in $\cD(\Z_p)$. By Theorems \ref{thm:main-sat-descent-prism} and \ref{thm:main-sat-descent-dRW}, since $A\otimes_{k[M]}k[M_\mathrm{perf}^{\oplus_M^{\rm sat} (m + 1)}]\in \lQSyn_k$ by Lemma \ref{lem:cotimes-qsyn}, the result follows from the equivalence
\begin{equation}\label{eq:reduction-to-BMS2-trivial}
    \cPrism_{A\otimes_{k[M]}k[M_\mathrm{perf}^{\oplus_M^{\rm sat} (m + 1)}]} \simeq \widehat{L\WOmega}_{A\otimes_{k[M]}k[M_\mathrm{perf}^{\oplus_M^{\rm sat} (m + 1)}]},
\end{equation}
    in $\cD(\Z_p)$, which is functorial since the resolution $A\to A\otimes_{k[M]}k[M_\mathrm{perf}^{\oplus_M^{\rm sat} (\bullet + 1)}]$ is.
\end{proof}

Moreover, we also deduce the following interesting comparison in case $P=\N$, i.e.\ $(k, {\N})$ is the standard log point.

\begin{thm}
\label{thm:main-sat-descent-dRW2}
Let $(A,M)\in \lQSyn_{(k,\N)}$ with $M$ saturated and $M^{\rm gp}$ $p$-torsionfree. Assume that there is a section $\ol{M}\to M$ of $M\to \ol{M}$ factoring $\N\to M$ such that $k\otimes_{k[\N]}k[\ol{M}]\to A$ is flat.
If $\N\to M$ is saturated,
then there is a $\phi$-equivariant equivalence\[
\widehat{L\WOmega}_{(A,M)/(k,\N)}\simeq \widehat{L\WOmega}_{(A,M\oplus_\N\N_{\rm perf})/k}\simeq \cPrism_{(A,M\oplus_\N\N_{\rm perf})}.
\]
This equivalence is functorial among maps $(A,M)\to (A',M')\in \lQSyn_{(k,\N)}$ and $(A',M')$ as above.
\begin{proof}
By Example \ref{ex:condition-clubsuit} (3), we have that\[
(A,M\oplus_\N\N_{\rm perf})\simeq (A,M)\otimes_{(k,\N)}^L(k,\N_{\rm perf})
\]
so since the map $\N\to k$ factors through $\N_{\rm perf}$ we have that\[
  \L_{(A,M)/(k,\N)}\simeq \L_{(A,M\oplus_\N\N_{\rm perf})/(k,\N_{\rm perf})}\simeq \L_{(A,M\oplus_\N\N_{\rm perf})/k},\]
  where the last equivalence follows from \cite[Corollary 4.13]{BLPO-BMS2}. Then the first equivalence follows from the conjugate filtration. 
By Example \ref{ex:condition-clubsuit} (3), we have that $A\otimes^L_{k[M\oplus_\N\N_{\rm perf}]}k[M_{\rm perf}]$ is discrete, so by Lemma \ref{lem:sat-descent-dRW-general} and Proposition \ref{prop:semipf-invariance-cotangent} we have that
\begin{align*}
    \widehat{L\WOmega}_{(A,M\oplus_\N\N_{\rm perf})/(k,\N_{\rm perf})} &\simeq \varlim_{
    \Delta} (\widehat{L\WOmega}_{(A\cotimes_{k[M\oplus_\N\N_{\rm perf}]} k[M_\mathrm{perf}^{\oplus_Q^{\rm sat} (\bullet + 1)}],M_\mathrm{perf}^{\oplus_Q^{\rm sat} (\bullet + 1)})/(k,\N_{\rm perf})})\\
    &\simeq \varlim_{
    \Delta} (\widehat{L\WOmega}_{(A\cotimes_{k[M\oplus_\N\N_{\rm perf}]} k[M_\mathrm{perf}^{\oplus_Q^{\rm sat} (\bullet + 1)}])/k}).
\end{align*}
Where the cosimplicial monoid given by the saturated \v Cech nerve $M_\mathrm{perf}^{\oplus_Q^{\rm sat} (\bullet + 1)}$ is computed relative to the map $Q:=M\oplus_\N\N_{\rm perf} \to M_{\rm perf}$.
Moreover, by \cite[Proposition I.4.2.1 (6)]{ogu}, the map $\N_{\rm perf}\to M\otimes_\N\N_{\rm perf}$ is exact. Finally, by direct computation we have that \[
\N_{\rm perf} = \Z[1/p^{\infty}]_{\geq 0}\subseteq \Z[1/p^{\infty}]=(\N_{\rm perf})^{\rm gp},
\] so it is sharp and valuative (in particular $p$-torsionfree as observed in Remark \ref{rmk:sharp-p-torsionfree}). 
We now check that $(M\oplus_\N\N_{\rm perf})^{\rm gp}$ is $p$-torsionfree. Since $\N\to M$ is injective as observed in Example \ref{ex:condition-clubsuit} (iii), we have an exact sequence\[
0\to (\N_{\rm perf})^{\rm gp}\to (M\oplus_{\N}\N_{\rm perf})^{\rm gp}\to M^{\rm gp}/\Z\to 0,
\]
so $(M\oplus_{\N}\N_{\rm perf})^{\rm gp}$ is $p$-torsionfree if and only if $M^{\rm gp}/\Z$ is $p$-torsionfree. Let $M^*\subseteq M$ be the group of units: it is also $p$-torsionfree, and let $\ol{M}:= M/M^*$. Let $x\in M^{\rm gp}$ such that $n\ol{x}=0$ in $\ol{M}^{\rm gp}$. Then $n x$ comes from $M^*$, and since $M$ is saturated $x$ comes from $M$, and $x + (n-1)x + (-nx) = 0$, so $x$ comes from $M^*$ therefore $\Z\to \ol{M}^{\rm gp}$ is also injective and we have an exact sequence\[
0\to M^* \to M^{\rm gp}/\Z\to \ol{M}^{\rm gp}/\Z\to 0.
\] 
By \cite[Proposition I.4.8.8]{ogu}, the map $\N\to \ol{M}$ is saturated, therefore since we assumed $\ol{M}$ fine, by \cite[Proposition I.4.8.11]{ogu} \footnote{In the statement of \emph{loc.\ cit.}\ we take the face $G$ to be zero. As the proof shows, the finiteness assumption in the statement is not necessary.} we have that $\ol{M}^{\rm gp}/\Z$ is torsionfree, so $M^{\rm gp}/\Z$ is $p$-torsionfree since $M^*\subseteq M^{\rm gp}$ is $p$-torsionfree: this implies that $(M\oplus_\N\N_{\rm perf})^{\rm gp}$ is $p$-torsionfree.

By Lemma \ref{lem:sat-descent-prismatic-general} and Proposition \ref{prop:semipf-invariance-prismatic}, we have equivalences
\begin{align*}
\cPrism_{(A,M\oplus_\N\N_{\rm perf})}&\simeq \varlim_{
    \Delta
    } (\cPrism_{(A\cotimes_{k[M\oplus_\N\N_{\rm perf}]} k[M_\mathrm{perf}^{\oplus_Q^{\rm sat} (\bullet + 1)}],M_\mathrm{perf}^{\oplus_Q^{\rm sat} (\bullet + 1)})})\\
    &\simeq \varlim_{\Delta^{\rm op}
    } (\cPrism_{A\cotimes_{k[M\oplus_\N\N_{\rm perf}]} k[M_\mathrm{perf}^{\oplus_Q^{\rm sat} (\bullet + 1)}}).
\end{align*}
From this, we obtain an isomorphism of simplicial $\Z_p$-modules 
\begin{equation}\label{eq:reduction-to-BMS2-logpoint}
\cPrism_{A\cotimes_{k[M\oplus_\N\N_{\rm perf}]} k[M_\mathrm{perf}^{\oplus_Q^{\rm sat} (\bullet + 1)}]}
\simeq \widehat{L\WOmega}_{A\cotimes_{k[M\oplus_\N\N_{\rm perf}]} k[M_\mathrm{perf}^{\oplus_Q^{\rm sat} (\bullet + 1)}]/k}
\end{equation}
by \cite[Theorem 8.17]{BMS2}. Functoriality follows from the resolution\[
A\to A\otimes_{k[M\oplus_{\N}\N_{\rm perf}]}k[M_\mathrm{perf}^{\oplus_Q^{\rm sat} (\bullet + 1)}]
\] being functorial.
\end{proof}
\end{thm}

We now aim for variants of the crystalline comparison. For this, we restrict to the case of $(A,M)\in \lSm_{(k,P)}$ with $P\to M$ saturated. Notice that the Nygaard filtration on $\WOmega_{(A,M)/(k,P)}$ is clearly complete as $p^{i-q-1}V\WOmega^q_{(A,M)/(k,P)}$ vanishes for $i\gg 0$. In particular the map of filtered $E_\infty$-rings $L\WOmega_{(A,M)/(k,P)}\to \WOmega_{(A,M)/(k,P)}$ induced by any choice of a polynomial resolution of $(A,M)$ factors through $\widehat{L\WOmega}_{(A,M)/(k,P)}$. In fact, we have the following:
\begin{lemma}\label{lem:LWOmega-vs-WOmega}
    Let $(A,M)\in \lSm_{(k,P)}$ with $P\to M$ saturated. Then the canonical map\[
    \widehat{L\WOmega}_{(A,M)/(k,P)}\xrightarrow{} \WOmega_{(A,M)/(k,P)}
    \]
    is an equivalence of complete filtered $E_\infty$-rings.
    \begin{proof}
It is enough to check it on the graded pieces of the conjugate filtration: in that case, it is the map\[
\bigwedge^i\L_{(A,M)/(k,P)}\to \Omega^i_{(A,M)/(k,P)}.
\]
By \cite[Proposition 4.6]{BLPO-HKR}, since $(k,P)\to (A,M)$ is log smooth and integral by hypothesis, we have that $\L_{(A,M)/(k,P)}\simeq \Omega^1_{(A,M)/(k,P)}$, so the claim follows. 
    \end{proof}
\end{lemma}

\begin{lemma}
\label{lem:reduced p-torsionfree}
Let $X$ be a noetherian fs log scheme over $\F_p$ such that $\ul{X}$ is reduced. Then $\Gamma(X,\cM_X)^\mathrm{gp}$ is $p$-torsionfree.
\end{lemma}
\begin{proof}

We set $M:=\Gamma(X,\cM_X)$.
We can work strict \'etale locally on $X$,
so we may assume that $\ul{X}$ is an affine scheme $\Spec(A)$ and the quotient map $M\to \ol{M}$ admits a section $\ol{M}\to M$ by \cite[Lemma 1.3.3]{Tsu99}.
Since $\ol{M}^\mathrm{gp}$ is torsionfree by \cite[Proposition I.1.3.5 (2)]{ogu},
it suffices to show that $M^*\cong A^\times$ is $p$-torsionfree.
If $x\in A^\times$ satisfies $x^p=1$,
then we have $(x-1)^p=0$.
Since $A$ is reduced,
we have $x=1$.
Hence $A^\times$ is $p$-torsionfree.
\end{proof}

For a noetherian pre-log ring $(R,P)$ such that $\ol{P}$ is fine and saturated and the induced map $P\to \Gamma(\Spec(R,P),\cM_{\Spec(R,P)})$ is an isomorphism,
let $\lSm_{(R,P)}^\mathrm{Aff}$ be the full subcategory of $\lSm_{(R,P)}$ spanned by those $X$ such that $\ul{X}$ is affine and $\Gamma(X,\cM_X)$ is a chart for $X$.

By \cite[Proposition 7.7]{BLPO-HKR},
$\lSm_{(R,P)}^\mathrm{Aff}$ is equivalent to the opposite category of the full subcategory of pre-log $(R,P)$-algebras spanned by those objects $(A,M)$ such that $\Spec(A,M)$ is log smooth over $\Spec(R,P)$, the induced map $M\to \Gamma(\Spec(A,M),\cM_{\Spec(A,M)})$ is an isomorphism, and $\ol{M}$ is fine and saturated.

\begin{lemma}
\label{lem:Tsuji globalization}
With the above notation,
let $\cC$ be an $\infty$-category with limits.
There is an equivalence of $\infty$-categories
\[
\Sh_\set(\lSm_{(R,P)}^\mathrm{Aff},\cC)
\simeq
\Sh_\set(\lSm_{(R,P)},\cC).
\]
The same holds for any basis $\cB$ for the strict \'etale topology on $\lSm_{(R,P)}^\mathrm{Aff}$.
\end{lemma}
\begin{proof}
By \cite[Lemmas 1.3.2, 1.3.3]{Tsu99},
$\lSm_{(R,P)}^\mathrm{Aff}$ is a basis for  the strict \'etale topology on $\lSm_{(R,P)}$.
Hence \cite[Corollary A.8]{zbMATH07654538} implies that 
we have an equivalence of $\infty$-categories
\[
\Sh_\set(\cB,\Spc)
\simeq
\Sh_\set(\lSm_{(R,P)},\Spc),
\]
where $\Spc$ denotes the $\infty$-category of spaces.
Use \cite[Proposition 1.3.1.7]{SAG} to conclude.
\end{proof}

\begin{rmk}\label{rmk:crys-global}
 
Let $R$ be a noetherian ring, let $P$ be an fs sharp monoid, and let $(R,P)$ be the standard log base with $P\to R$ the zero map. Write again $\lSm_{(R,P)}$ for the category of log smooth and separated fs log schemes of finite type over $(R,P)$. Let $X\in \lSm_{(R,P)}$ and assume that $X\to \Spec(R,P)$ is saturated.
By \cite[Theorem IV.3.3.1 (3)]{ogu}, for every geometric point $x$ of $\ul{X}$ there is an \'etale neighborhood $\widetilde{\ul{X}}$ and a chart $P\to Q$ neat at $x$.
 Since $f$ is saturated, $P\to Q$ is saturated too. 
 By \cite[Remark II.2.4.5]{ogu},
 $Q$ is a chart neat at $x$.
 In particular, $Q$ is sharp.
Furthermore,  by \cite[Lemma 1.3.3]{Tsu99}, we may also assume that $M:=\Gamma(\widetilde{X},\cM_{\widetilde{X}})$ is a chart of $\widetilde{X}$ and the composite $Q\to M\to \ol{M}$ is surjective. The composite $Q\to \ol{M}\to \ol{\cM}_{X,x}$ is an isomorphism since the chart $Q$ is neat at $x$, so we have $Q\cong \ol{M}$. Furthermore, $M^\mathrm{gp}$ is $p$-torsionfree by Lemma \ref{lem:reduced p-torsionfree}.
 In particular, we can globalize Theorems  \ref{thm:crys_comparison} and \ref{thm:main-sat-descent-dRW2}:
\end{rmk}
\begin{thm}\label{thm:crys_comp_triv}
For all $X\in \lSm_{(k,\triv)}$ we have an equivalence 
\begin{equation}\label{eq:global-crys-comp-trivial-base}
R\Gamma_{\cPrism}(X/k)\simeq R\Gamma_{\rm crys}(X/W(k)).
\end{equation} of filtered $E_\infty$-rings that depends functorially on $X$.
\begin{proof}
In light of Lemma \ref{lem:Tsuji globalization} and Remark \ref{rmk:crys-global} we only need to construct a natural equivalence
\begin{equation}\label{eq:global-crys-comp-trivial-base2}
R\Gamma_{\cPrism}(\Spec(A,M)/k)\simeq R\Gamma_{\rm crys}(\Spec(A,M)/W(k))
\end{equation}
for $X=\Spec(A,M)\in \lSm_{(k,k^\times)}^\mathrm{Aff}$
with $M$ saturated and $M^\mathrm{gp}$ $p$-torsionfree such that $M\cong \Gamma(X,\cM_X)$ and  that there is a section $\ol{M}\to M$ of the quotient map $M\to \ol{M}$ such that $k[M]\to A$ is flat. Note that we only require that \eqref{eq:global-crys-comp-trivial-base2} is natural in $(A,M)$ and not in $\ol{M}$.
So we conclude by combining Theorem \ref{thm:crys_comparison}, Lemma \ref{lem:LWOmega-vs-WOmega}, and quasisyntomic descent.
 
    \end{proof}
\end{thm}

\begin{thm}\label{thm:crys_comp_logpoint}
 
For all $X\in \lSm_{(k,\N)}$ with $\ul{X}$ reduced, we have an equivalence
\begin{equation}\label{eq:global-crys-comp-log-base}
R\Gamma_{\cPrism}((\ul{X},\partial X\oplus_\N\N_{\rm perf})/k)\simeq R\Gamma_{\rm crys}(X/W(k,\N))
\end{equation} of filtered $E_\infty$-rings that depends functorially on $X$.

\begin{proof}
    Since $\N$ is valuative, the map $X\to \Spec(k,\N)$ is integral, and it is saturated if and only if $\ul{X}$ is reduced by \cite[Theorem II.4.2]{tsuji}, so we deduce that $X\to \Spec(k,\N)$ is saturated. Then by \cite[Proposition I.4.7.5 and Theorem III.2.2.7]{ogu} the morphism is also s-integral. As before,
 we only need to construct a natural equivalence \eqref{eq:global-crys-comp-log-base} for $X=\Spec(A,M)\in \lSm_{(k,k^\times\oplus \N)}^\mathrm{Aff}$
with $M$ saturated and $M^\mathrm{gp}$ $p$-torsionfree such that $M\cong \Gamma(X,\cM_X)$ and that there is a section $\ol{M}\to M$ of the quotient map $M\to \ol{M}$ factoring $\N\to M$ such that $k\otimes_{k[\N]}k[\ol{M}]\to A$ is flat. 
 
So we conclude again by combining Theorem \ref{thm:main-sat-descent-dRW2}, Lemma \ref{lem:LWOmega-vs-WOmega}, and quasisyntomic descent.
\end{proof}
\end{thm}

\begin{rmk}
    As observed in \cite[Warning 4.6.2]{BhattLurie}, the equivalences \eqref{eq:global-crys-comp-trivial-base} and \eqref{eq:global-crys-comp-log-base} are \emph{not} $W(k)$-linear. Indeed, the maps are $F$-semilinear, inducing equivalences
    \begin{align*}
 &F^*R\Gamma_{\cPrism}(X)\simeq R\Gamma_{\rm crys}(X/W(k))\\
 &F^*R\Gamma_{\cPrism}(X\times_{(k,\N)}\Spec(k,\N_{\rm perf}))\simeq R\Gamma_{\rm crys}(X/W(k,\N))
\end{align*} in $\cD(W(k))$.
\end{rmk}

\subsection{Motivic crystalline comparison}

\begin{thm}
    When restricted to $\SmlSm_k$, the equivalence \eqref{eq:crystalline-comp} induces a $\Z_p$-linear equivalence \[
\mathbf{E}^{\cPrism}\simeq \mathbf{E}^{\rm crys}\textrm{ in }\logDA(k,\Z_p)
\]
of oriented ring spectra, which becomes $W(k)$-linear after twisting with Frobenius: \[
F^*\mathbf{E}^{\cPrism}\simeq \mathbf{E}^{\rm crys}\textrm{ in }\logDA(k,W(k))
\]
\end{thm}
\begin{proof}
 
    As recalled for example in \cite[Example 2.6.4]{BhattLurie}, since the base has characteristic $p$ the Breuil--Kisin twist is canonically trivialized, so for the graded commutative monoid in $\logDAeff(k,W(k))$ given by the collection $\{R\Gamma_{\cPrism}(-)\{i\}\}_{i\in \Z}$ is equivalent to the (constant) graded commutative monoid $\{R\Gamma_{\cPrism}(-)\}$. This implies that the equivalence \eqref{eq:global-crys-comp-trivial-base} induces an equivalence of constant graded commutative monoids, so by Proposition \ref{prop:build-spectra-maps} it is enough to show that for all $X\in \SmlSm_k$ the following diagram commutes in $\logDAeff(k,\Z_p)$:
    \[
    \begin{tikzcd}
        \Z_p(\P^1)\ar[r,"c_{\cPrism}"]\ar[dr,bend right = 20,"c_{\rm crys}"] &R\Gamma_{\cPrism}(-)\ar[d,"\gamma_{\cPrism}"]\\
        &R\Gamma_{\rm crys}(-)
    \end{tikzcd}
    \]
    In fact, $c_{\cPrism}$ is induced by the prismatic Chern class $c_{1}^{\cPrism}(\cO(1))$ and $c_{\rm crys}$ is induced by the crystalline Chern class $c_{1}^{\rm crys}(\cO(1))$: then the result follows from \cite[Proposition 7.5.5]{BhattLurie}.

    To get the second equivalence, just notice that the prismatic Chern class factors through $R\Gamma_{\syn}(-)\to \Fil^1_\rN R\Gamma_{\cPrism}(-)\xrightarrow{\rm can}  R\Gamma_{\cPrism}(-)$: in particular it factors through the $F$-semilinear map $\Fil^1_\rN R\Gamma_{\cPrism}(-)\xrightarrow{\rm can}  R\Gamma_{\cPrism}(-)$: this implies that there is a commutative diagram in $\logDAeff(k,W(k))$\[
     \begin{tikzcd}
        W(k)(\P^1)\ar[r,"c_{\cPrism}"]\ar[dr,bend right = 20,"c_{\rm crys}"] &F^*R\Gamma_{\cPrism}(-)\ar[d,"F^*\gamma_{\cPrism}"]\\
        &R\Gamma_{\rm crys}(-),
    \end{tikzcd}
    \]
    then the result is clear.
\end{proof}
\begin{rmk}
    This implies that the prismatic Gysin map of \ref{thm:gysin-prism} agrees with the crystalline Gysin map of \cite{Gros1985} (up to $F^*$), as both agree with the sequence induced by homotopy purity (see \cite[(4.5.2)]{mericicrys}) 
\end{rmk}

\section{A log Segal conjecture}

Let $k$ be a field of characteristic $p$. In this section, we will prove an analogue of \cite[Corollary 8.18]{BMS2} in the log setting, which is needed to define the Breuil--Kisin cohomology. First, we prove a statement for pre-log algebras over $(k,\triv)$, which is independent of the saturated descent of the previous section. To simplify the notation, we will write $\THH, \TC^{-}$ and $\TP$ for their respective $p$-completed versions.

By \cite[Proposition 6.3]{BMS2},
we have a commutative square of graded rings with vertical isomorphisms
\[
\begin{tikzcd}[column sep=huge]
\pi_*\TC^-(k)
\ar[d,"\simeq"']
\ar[r,"\varphi"]
&
\pi_*\TP(k)
\ar[d,"\simeq"]
\\
\Z_p[u,v]/(uv-p)
\ar[r,"{u\mapsto \sigma,v\mapsto p\sigma^{-1}}","\varphi-\text{linear}"']
&
{\Z_p[\sigma,\sigma^{-1}]}.
\end{tikzcd}
\]
As observed in the proof of \cite[Proposition 6.4]{BMS2}, for all $\bT$-equivariant $\THH(k)$-modules $\sF$,
we have a natural equivalence
\[
\sF^{h\bT}
\otimes_{\TC^-(k)}
\THH(k)
\simeq
\sF
\]
For a pre-log $k$-algebra $(A,M)$,
we apply this equivalence to $\cF=\THH(A,M)$ and $\cF=\THH(A,M)^{tC_p}$ to obtain a commutative square with vertical equivalences
\begin{equation}
\label{p-case.20.1}
\begin{tikzcd}
\TC^-(A,M)/v
\ar[r,"\varphi"]
\ar[d,"\simeq"']
&
\TP(A,M)/p
\ar[d,"\simeq"]
\\
\THH(A,M)
\ar[r,"\varphi"]
&
\THH(A,M)^{tC_p}.
\end{tikzcd}
\end{equation}
The equivalence $\THH(A,M)^{tC_p}\simeq \TP(A,M)/p$ is compatible with the motivic filtrations of \cite[Proposition 7.5]{BLPO-BMS2} by checking on log quasiregular semiperfect $k$-algebras, so it gives equivalences
\begin{align}
&\mathrm{gr}^i\THH(A,M)\simeq \mathrm{gr}^i\TCmin(A,M)/v\simeq \Fil^{\geq i}_\rN\cPrism_{(A,M)}/v[2i]\simeq \gr^{i}_N\cPrism_{(A,M)}[2i],\label{eq:general-THH-graded-prismatic} \\
&\mathrm{gr}^i\THH(A,M)^{tC_p}\simeq \mathrm{gr}^i\TP(A,M)/p\simeq \cPrism_{(A,M)}/p[2i]\simeq \widehat{L\Omega}_{(A,M)/k}[2i],
\label{eq:general-THHtCp-derived-de-Rham} 
\end{align}
where the last equivalences are stated in \cite[\S 7.2]{BLPO-BMS2}, analogous to \cite[Theorem 7.2]{BMS2} for the non-log case.

\begin{prop}\label{prop:log-segal}
For all $i\in\Z$, the above map gives an equivalence of sheaves of spectra on $\lSm(k)$:\[
\mathrm{gr}^i\THH(-)\simeq (\tau^{\leq i}\Omega)[2i],
\]
Moreover, under this equivalence, the map $\phi\colon \mathrm{gr}^i\THH(-)\to \mathrm{gr}^i\THH(-)^{tC_p}$ is the natural map $(\tau^{\leq i}\Omega)[2i]\to \Omega[2i]$. In particular, for all $(A,M)\in \lSm(k)$ the map\[
\phi\colon \THH(A,M)\to \THH(A,M)^{tC_p}
\]
is an equivalence in degrees $\geq \dim(A)$. 
\begin{proof}
Recall that by the conjugate filtration $\gr^i\THH(-)$ lives naturally in $\logDAeff(k,\Z_p)$ (via the Dold-Kan correspondence). Composing \eqref{eq:general-THHtCp-derived-de-Rham} with the map $\phi\colon \mathrm{gr}^i\THH(-)\to \mathrm{gr}^i\THH(-)^{tC_p}$ we obtain a map\[
\mathrm{gr}^i\THH(-) \to \Omega_{(-)/k}[2i].
\] of objects of $\logDAeff(k,\Z_p)$.
Let $X\in \lSm_k$ connected: by \cite[Corollary 8.18]{BMS2}, the result holds if $X$ has trivial log structure, in particular if $\eta_X$ is the generic point of $X$ we have that $\gr^i\THH(\eta_X)\simeq (\tau^{\leq i}\Omega_{\eta_X})[2i]$. This implies that $\mathrm{gr}^i\THH[-2i]$ is generically $(-i)$-connective, so by \cite[Corollary 4.6]{BindaMerici}, we conclude that $\mathrm{gr}^i\THH[-2i]$ is $(-i)$-connective for the homotopy $t$-structure of $\logDAeff(k)$, in particular $\phi[-2i]$ factors as \[
\begin{tikzcd}
{\rm gr}^i\THH(-)[-2i]\ar[rr,"{\phi[-2i]}"]\ar[dr,"\psi"]&&{\rm gr}^i\THH(-)^{tC_p}[-2i]\\
&\tau_{\geq-i}{\rm gr}^i\THH(-)^{tC_p}[-2i]\ar[ur].
\end{tikzcd}
\]
To conclude, it is enough to check that $\psi$ is an equivalence, and
again by \cite[Corollary 4.6]{BindaMerici}, it is enough to check it on $(\eta_X,\triv)$, which follows again from \cite[Corollary 8.18]{BMS2}.
\end{proof}
\end{prop}

Let us now move to the case of algebras over the log point $(k,\N)$. Here, let $(\Sph,\N)$ be the pre-log ring spectrum given by the quotient map $\Sph[\N]\to \Sph$: then if $(A,M)$ is a pre-log $(k,\N)$-algebra we have an induced map of pre-log ring spectra $(\Sph,\N)\to (A,M)$. Recall that in light of \cite[Remark 8.5]{BLPO-BMS2}, the spherical log point $(\Sph,\N)$ is a log cyclotomic base in the sense of \cite[Definition 8.3]{BLPO-BMS2}.

\begin{prop}\label{prop:log-segal2}
For every saturated log smooth map $(k,\N)\to (A,M)$ of fs pre-log rings
we have a natural equivalence
\[
\mathrm{gr}^i\THH((A,M)/(\Sph,\N))\simeq (\tau^{\leq i}\Omega_{(A,M)/(k,\N)})[2i].
\] 
Moreover,
under this equivalence the map \[\phi\colon \mathrm{gr}^i\THH((A,M)/(\Sph,\N))\to \mathrm{gr}^i\THH((A,M)/(\Sph,\N))^{tC_p}\] is the natural map $\tau^{\leq i}\Omega_{(A,M)/(k,\N)}\to \Omega_{(A,M)/(k,\N)}$.
In particular,
the map
\[
\varphi\colon \THH((A,M)/(\Sph,\N))
\to
\THH((A,M)/(\Sph,\N))^{tC_p}
\]
is an equivalence in degrees $\gg 0$. 
\end{prop}
\begin{proof}
The question is \'etale local on $A$.
Hence by \cite[Theorem 3.5]{katolog},
we may assume that
\begin{enumerate}
\item[(i)] $\N\to M$ is injective and saturated,
\item[(ii)] $k\otimes_{k[\N]}k[M]\to A$ is \'etale.
\end{enumerate}
Recall that \[
\THH((A,M)/(\Sph,\N))\simeq \THH((A,M))\otimes_{\THH(\Sph,\N)} \Sph
\]
is an equivalence of cyclotomic spectra, since $(\Sph, \N)$ is a log cyclotomic base.
Consider the pre-log ring $(A,M\oplus_\N \N_{\rm perf})=(A,M)\otimes^L_{(k,\N)}(k,\N_{\rm perf})$ as in the proof of Theorem \ref{thm:main-sat-descent-dRW2}. Then we have an equivalence \[
\THH(A,M\oplus_{\N}\N_{\rm perf})\simeq \THH(A,M)\otimes_{\THH(k,\N)}\THH(k,\N_{\rm perf}).
\] of cyclotomic spectra.
We shall now argue that there is an equivalence of cyclotomic spectra \begin{equation}\label{eq:wishingawayperfectiononceagain}
\THH(A,M\oplus_{\N}\N_{\rm perf}) \xrightarrow{\simeq} \THH((A,M)/(\Sph,\N)).
\end{equation}
For this, consider the commutative diagram \[\begin{tikzcd}[row sep = small]{\rm THH}({\Sph}, {\N}) \ar{r} \ar{d} & {\rm THH}({\Sph}, {\N}_{\rm perf}) \ar{d} \ar{r}{} & {\Sph} \ar{d} \\ {\rm THH}(k, {\N}) \ar{r} \ar{d} & {\rm THH}(k, {\N}_{\rm perf}) \ar{d} \ar{r} & {\rm THH}((k, {\N}_{\rm perf}) / ({\Sph}, {\N}_{\rm perf})) \ar{d} \\ {\rm THH}(A, M) \ar{r} & {\rm THH}(A, M \oplus_{\N} {\N}_{\rm perf}) \ar{r}{} & {\rm THH}((A, M) / ({\Sph}, {\N}))\end{tikzcd}\] in which every square is defined to be cocartesian. The identification of the lower right-hand corner is a consequence of the outer square being cocartesian.  Since both $k$ and ${\N}_{\rm perf}$ are perfect, the canonical map ${\rm THH}(k) \to {\rm THH}(k, {\N}_{\rm perf})$ is an equivalence (this follows from \cite[Corollary 4.13, Proposition 7.3]{BLPO-BMS2} by reduction to the corresponding statement for the cotangent complex), and the canonical map ${\rm THH}(k) \to {\rm THH}((k, {\N}_{\rm perf}) / ({\Sph}, {\N}_{\rm perf}))$ is also an equivalence since $({\Sph}, {\N}_{\rm perf}) \to (k, {\N}_{\rm perf})$ is strict. It follows that the map ${\rm THH}(k, {\N}_{\rm perf}) \to {\rm THH}((k, {\N}_{\rm perf}) / ({\Sph}, {\N}_{\rm perf}))$ is an equivalence,  and thus also the map ${\rm THH}(A, M \oplus_{{\N}} {\N}_{\rm perf}) \xrightarrow{} {\rm THH}((A, M) / ({\Sph}, {\N}))$ is an equivalence of cyclotomic spectra. The latter of these is \eqref{eq:wishingawayperfectiononceagain}.

By \eqref{eq:general-THHtCp-derived-de-Rham} and Theorem \ref{thm:main-sat-descent-dRW2},
we have an equivalence\[
\gr_N^i\THH((A,M)/(\Sph,\N))^{tC_p}\simeq \WOmega_{(A,M)/(k,\N)}/p[2i]\simeq \Omega_{(A,M)/(k,\N)}[2i]
\]
and the induced map $\gr^i\THH((A,M)/(\Sph,\N)\to \Omega_{(A,M)/(k,\N)}[2i]$ factors via the isomorphism $\gr^i\THH((A,M)/(\Sph,\N)\simeq \gr^i_N\cPrism_(A,M\oplus_\N \N_{\rm perf})\simeq \tau^{\leq i}\Omega_{(A,M)/(k,\N)}[2i]$ of \eqref{eq:drw-graded}. Finally, the map\[
\gr^i \TCmin(A,M\oplus_\N \N_{\rm perf}) \to \gr^i \TP(A,M\oplus_\N \N_{\rm perf})
\]
agrees with the inclusion $\Fil_\rN^{\geq i}\WOmega_{(A,M)/(k,\N)}[2i]\subseteq \WOmega_{(A,M)/(k,\N)}[2i]$: this implies that modulo $v$ and $p$ it is the natural map $\tau^{\leq i}\Omega_{(A,M)/(k,\N)}\to \Omega_{(A,M)/(k,\N)}$. The last statement follows from the fact that $\Omega^1_{(A,M)/(k,\N)}$ is locally free of finite type by \cite[Proposition IV.3.2.1]{ogu}.
\end{proof}

\section{The log Breuil--Kisin cohomology} \label{sec:BK_cohomology}
We fix the setup of \cite[\S 11]{BMS2}. 
Let $K$ be a discretely valued extension of $\Q_p$ with perfect residue field $k$, ring of integers $\cO_K$, and fix a uniformizer $\varpi \in \cO_K$. Let $K_\infty$ be the $p$-adic completion of $K(\varpi^{1/p^{\infty}})$ and let $C$ be the completion of an algebraic closure of $K_\infty$ and $A_{\rm inf} = A_{\rm inf}(\cO_C)$. 

Let $\mathfrak{S} = W(k)\llbracket z\rrbracket$. There is a surjective map $\tilde \theta: \mathfrak{S} \to \cO_K$ determined by the inclusion $W(k)\subseteq \cO_K$ and $z \mapsto \varpi$. The kernel of this map is generated by an Eisenstein polynomial $E=E(z) \in \mathfrak{S}$ for $\varpi$. Let $\phi$ be the endomorphism of $\mathfrak{S}$ determined by the Frobenius on $W(k)$ and $z\mapsto z^p$. We regard $\mathfrak{S}$ as a $\phi$-stable subring of $A_{\rm inf}(\cO_{K_\infty})$ or $A_{\rm inf}$ by the Frobenius on $W(k)$ and sending $z$ to $[\varpi^\flat]^p$ where $\varpi^\flat = (\varpi,\varpi^p,\varpi^{p^2},\ldots) \in \cO^\flat_{K_\infty}$ is
our chosen compatible system of $p$-power roots of $\varpi$; the resulting map $\mathfrak{S} \to A_{\rm inf}$ is faithfully flat and even topologically free (see \cite[Lemma 4.30 and its proof]{BMS2}). Write $\theta = \tilde \theta\circ \phi : \mathfrak{S} \to \cO_K$. The embedding $\mathfrak{S}\subseteq A_{\rm inf}$ is compatible with the $\theta$ and $\tilde \theta$ maps. 

This section aims to construct a cohomology theory for log smooth formal schemes that is valued in Breuil--Kisin $\mathfrak{S}$-modules, generalizing \cite[Theorem 11.2]{BMS2}. Namely, we will prove that:

\begin{thm}\label{thm:BK-horizontal}
To any log smooth affine formal log scheme $\Spf(A,M)/(\cO_K,\triv)$, one can functorially (for maps of formal log schemes) attach a $(p,z)$-complete $E_\infty$-algebra $\cPrism_{(A,M)/\mathfrak{S}}\in \cD(\mathfrak{S})$ together with a $\phi$-linear Frobenius endomorphism $\phi: \cPrism_{(A,M)/\mathfrak{S}} \to  \cPrism_{(A,M)/\mathfrak{S}}$ inducing an isomorphism $\cPrism_{(A,M)/\mathfrak{S}}\otimes_{\mathfrak{S},\phi}\mathfrak{S}[1/E]\simeq \cPrism_{(A,M)/\mathfrak{S}}[1/E]$, with the following features:
\begin{enumerate}
\item (de Rham comparison) After scalar extension along $\theta$, there is a functorial isomorphism
\[
\cPrism_{(A,M)/\mathfrak{S}}\otimes_{\mathfrak{S},\theta}^{L}\cO_K \simeq (\Omega_{(A,M)/\cO_K})_p^\wedge
\]
\item (crystalline comparison) After scalar extension along the map $\mathfrak{S}\to W(k)$ which is the Frobenius on $W(k)$ and sends $z$ to $0$, there is a functorial Frobenius equivariant isomorphism
\[
\cPrism_{(A,M)/\mathfrak{S}}\otimes_{\mathfrak{S}}^L W(k)\simeq R\Gamma_{\rm crys}((A_k,M)/W(k))
\]
\end{enumerate}
\end{thm}

\begin{thm}\label{thm:BK-semistable}
To any log smooth affine formal log scheme $\Spf(A,M)/(\cO_K,\langle \varpi\rangle)$ of log Cartier type, one can functorially (for maps of formal log schemes)  attach a $(p,z)$-complete $E_\infty$-algebra $\cPrism_{(A,M)/(\mathfrak{S},\langle z\rangle)}\in \cD(\mathfrak{S})$ together with a $\phi$-linear Frobenius endomorphism $\phi: \cPrism_{(A,M)/(\mathfrak{S},\langle z\rangle)} \to  \cPrism_{(A,M)/(\mathfrak{S},\langle z\rangle)}$ inducing an isomorphism $\cPrism_{(A,M)/(\mathfrak{S},\langle z\rangle)}\otimes_{\mathfrak{S},\phi}\mathfrak{S}[1/E]\simeq \cPrism_{(A,M)/(\mathfrak{S},\langle z\rangle)}[1/E]$, and having the following features:
\begin{enumerate}
\item (de Rham comparison) After scalar extension along $\theta$, there is a functorial isomorphism\[
\cPrism_{(A,M)/(\mathfrak{S},\langle z\rangle)}\otimes_{\mathfrak{S},\theta}^{L}\cO_K \simeq (\Omega_{(A,M)/(\cO_K,\langle \varpi\rangle)})_p^\wedge
\]
\item (crystalline comparison) After scalar extension along the map $\mathfrak{S}\to W(k)$ which is the Frobenius on $W(k)$ and sends $z$ to $0$, there is a functorial Frobenius equivariant isomorphism\[
\cPrism_{(A,M)/(\mathfrak{S},\langle z\rangle)}\otimes_{\mathfrak{S}}^L W(k)\simeq R\Gamma_{\rm crys}((A_k,M)/(W(k,\N)))
\]
\end{enumerate}
\end{thm}

The proof of these theorems will follow the path of \cite[Theorem 11.2]{BMS2}, by first constructing a Frobenius-twisted version. This was indeed first introduced in \cite[Theorem 8.1]{BLPO-BMS2}, analogously to \cite[\S 11.2]{BMS2}. Let us briefly recall this construction.

Recall that a \emph{log cyclotomic base} is a pre-log ring spectrum $(R,P)$ together with a commutative diagram \[
\begin{tikzcd}
    \THH((R,P) ; \Z_p)\ar[r,"\phi_p"]\ar[d]&\THH((R,P) ; \Z_p)^{tC_p}\ar[d]\\
    R\ar[r]&R^{tC_p}
\end{tikzcd}
\] of $S^1$-equivariant $E_\infty$-rings.
By \cite[Lemma 8.4]{BLPO-BMS2} the pre-log ring spectrum $(\Sph[z],\langle z\rangle)$ is a log cyclotomic base.
This implies that, for any $(\cO_K, \langle \varpi\rangle)$-algebra $(A, M)$ we have cyclotomic structure maps
\[
\varphi \colon {\THH}((A, M)/\Sph[z], \langle z \rangle) ; \Z_p) \to \THH((A, M) / \Sph[z], \langle z \rangle) ; \Z_p)^{tC_p}
\]
on relative log topological Hochschild homology ${\THH}((A, M) / (\Sph[z], \langle z \rangle) ; \Z_p)$. 

Moreover, by   
\cite[Proposition 8.6]{BLPO-BMS2}, we have an equivalence
\begin{equation}\label{eq:THH-base}
    \THH(\Sph[z^{1/p^{\infty}}],\langle z^{1/p^\infty}\rangle) \xrightarrow{\simeq} \Sph[z^{1/p^{\infty}}]
\end{equation}
after $p$-completion. These imply that for $(A,M)$ a pre-log $(\cO_K,\varpi)$-algebra  we have equivalences  
\begin{align*}
\THH((A,M)/(\Sph[z],\langle z \rangle))\otimes_{\Sph[z]}\Sph[z^{1/p^{\infty}}]&\xrightarrow{\simeq} \THH((A[\varpi^{1/p^{\infty}}],M_{\rm perf})/(\Sph[z^{1/p^{\infty}}],\langle z^{1/p^{\infty}}\rangle)\\
\THH((A[\varpi^{1/p^{\infty}}],M_{\rm perf})&\simeq \THH((A\otimes_{\cO_K}\cO_{K_{\infty}},M_{\rm perf})
\end{align*}
after $p$-completion (see \cite[Corollary 8.7]{BLPO-BMS2}). This implies by \cite[Proposition 8.8]{BLPO-BMS2} that the presheaf
\begin{align*}
(S,Q)\in \lQRSPerfd_{(\cO_K,\langle \varpi\rangle)} &\mapsto \pi_0 \TP((S,Q)/(\Sph[z],\langle z \rangle);\Z_p)
\end{align*}
is a log quasisyntomic sheaf. We define the sheaves of $\E_{\infty}$-$\mathfrak{S}^{(-1)}$-algebras 
\begin{align*}
\gr^0\TCmin(-/(\Sph[z],\langle z\rangle);\Z_p)\colon \lQSyn_{(\cO_K,\langle\varpi\rangle)}&\to \cD(\mathfrak{S})
\end{align*}
as the unfolding of    $\pi_0\TCmin(-/(\Sph[z],\langle z\rangle);\Z_p)$. We are now ready to state the analogue of \cite[Corollary 11.12]{BMS2}:

\begin{cor}\label{cor:frob-twisted}
    Let $\kX=\Spf(A,M)$ be  an affine log smooth formal scheme over $(\cO_K,\langle\varpi\rangle)$ of log Cartier type. Then the complex 
     $\cPrism_{\kX/(\mathfrak{S},\langle z\rangle)^{(-1)}}:= \gr_0\TCmin((A,M)/(\Sph[z],\langle z\rangle)$ is a $(p,z)$-complete $\E_{\infty}$-algebra object of $\cD(\mathfrak{S}^{(-1)})$ that admits a natural Frobenius endomorphism $\phi$ and has the following properties:
    \begin{enumerate}
    \item  
    There is a natural $\phi$-equivariant equivalence
    \begin{align*}
        \cPrism_{\kX/(\mathfrak{S},\langle z\rangle)^{(-1)}}\otimes_{\mathfrak{S}^{(-1)}}A_{\rm inf}(\cO_{K_{\infty}}) _{(p,z)}^\wedge&\simeq \cPrism_{\kX_{\cO_{K_{\infty}}}}
    \end{align*}
    \item There is a natural equivalence
        \begin{align*}
        \cPrism_{\kX/(\mathfrak{S},\langle z\rangle)^{(-1)}}\otimes_{\mathfrak{S}^{(-1)},\theta^{(-1)}}\cO_K &\simeq (\Omega_{\kX/(\cO_K),\langle \varpi\rangle})_p^\wedge
        \end{align*}
        of $\E_\infty$-$\cO_K$-algebras.
        \item After scalar extensions along the map $\mathfrak{S}^{(-1)}\to W(k)$ which is the identity on $W(k)$ and sends $z$ to $0$,
        there is a functorial Frobenius equivalence
        \begin{align*}
        \cPrism_{\kX/(\mathfrak{S},\langle z\rangle)^{(-1)}}\otimes_{\mathfrak{S}^{(-1)}}W(k) &\simeq R\Gamma_{\rm crys}(\kX_k/W(k,\N))
        \end{align*}
        of $\E_\infty$-$W(k)$-algebras.
        \item The Frobenius $\phi$ induces an equivalence
        \begin{align*}
        \cPrism_{\kX/(\mathfrak{S},\langle z\rangle)^{(-1)}}\otimes_{\mathfrak{S}^{(-1)},\phi}\mathfrak{S}^{(-1)}[1/\phi(E)]  &\simeq \cPrism_{\kX/(\mathfrak{S},\langle z\rangle)^{(-1)}}[1/\phi(E)]
        \end{align*}
    \end{enumerate}
\end{cor}
\begin{proof}
The proof is now completely analogous to \cite[Corollary 11.12]{BMS2}, we report it for completeness. Notice that $\cO_{K}\to \cO_{K_\infty}$ is $p$-completely flat, so 
\begin{align*}
    (A,M)\cotimes_{(\cO_K,\langle \varpi\rangle)}(\cO_{K_{\infty}},\N_\mathrm{perf})&\simeq (A\cotimes_{\cO_K}\cO_{K_{\infty}},M\oplus_{\N} \N_{\rm perf}).
\end{align*}
Part $(1)$ then follows from the equivalence:
\begin{align*}
    \gr^0\TP((A,M)/(\Sph[z],\langle z\rangle);\Z_p)\cotimes_{\Sph[z]}\Sph[z^{1/p^{\infty}}]&\overset{(*)}{\simeq}\gr^0\TP((A\cotimes_{\cO_K}\cO_{K_{\infty}},M\oplus_{\N} \N_{\rm perf});\Z_p)
\end{align*}
where $(*)$   follows by taking graded pieces and Tate construction on the analogous statement for $\THH$, which follows from the equivalences in \eqref{eq:THH-base}.
Part $(2)$ follows from the equivalence:
\begin{align*}
    \gr^0\TP((A,M)/(\Sph[z],\langle z\rangle);\Z_p)/E&\simeq \gr^0\HP((A,M)/(\cO_K,\langle \varpi\rangle);\Z_p),
\end{align*}
obtained by unfolding (see also \cite[Proof of Theorem 8.1]{BLPO-BMS2}), and \cite[Theorem 1.1]{BLPO-BMS2}.

Since $\N$ is valuative and the map $\N\to M$ is injective, $\Z[\N]\to \Z[M]$ is flat, in particular\[
\cO_K\to \cO_K\otimes_{\Z[\N]}\Z[M]\to A
\] 
is flat, so $(A,M)\otimes_{(\cO_K,\langle\varpi\rangle)}^L (k,\N_{\rm perf}) = (A\otimes^L_{\cO_K} k,M)$ is discrete.
Part $(2)$ follows from 
\begin{align*}
    \gr^0\TP((A,M)/(\Sph[z],\langle z\rangle))\otimes_{\Sph[z]}^L\Sph &\overset{(*)}{\simeq} \gr^0\TP((A,M)\otimes_{(\cO_K,\langle\varpi\rangle)} (k,\N_{\rm perf}))\simeq \cPrism_{(A/\varpi,M\oplus_\N \N_{\rm perf})},
\end{align*}
where $(*)$ follows from the equivalence after $p$-completion $\THH(\Sph,\N_{\rm perf})\to \Sph$, proved in the same way as \cite[Proposition 8.6]{BLPO-BMS2} using Proposition \ref{prop:semipf-invariance-cotangent}. Then Part $(3)$ follows now from  \eqref{eq:global-crys-comp-log-base}. Finally, Part $(4)$ follows from base change via the topological direct summand $\mathfrak{S}^{(-1)}\to A_{\rm inf}(\cO_C)$ and $(1)$, so similarly to \cite[Corollary 11.12 (4)]{BMS2} it is enough to show that $\phi^*\cPrism_{\kX_{\cO_{C}}/\cO_C}\to \cPrism_{\kX_{\cO_{C}}/\cO_C}$ induces an equivalence after inverting $\tilde\xi$: since $\kX$ is assumed to be log smooth over $\cO_K$, we have that $\cPrism_{\kX_{\cO_{C}}/\cO_C}\simeq \cPrism^{\rm nc}_{\kX_{\cO_{C}}/\cO_C}$, and by Corollary \ref{cor:sat-descent-de-Rham-prismnc} we reduce to the same question for the non-logarithmic non-Nygaard completed prismatic cohomology of \cite[Construction 7.12]{BMS2}, which in turns is equivalent by left Kan extension and \cite[Theorem 9.6]{BMS2} to the derived $A\Omega$-cohomology, for which the result follows by left Kan extension on $p$-completely smooth $\cO_K$-algebras.
\end{proof}

We are now ready to prove our version of the Frobenius descent, following \cite[Proposition 11.15]{BMS2}:

\begin{prop}\label{prop:frobdescent}

Let $(A, M)$ be an  $(\cO_K, \langle \varpi \rangle)$-pre-log algebra. 
\begin{enumerate}
\item The cyclotomic Frobenius
\begin{align*}
    \TCmin((A, M)/(\Sph[z], \langle z \rangle) ; \Z_p) &\xrightarrow{\phi} \TP((A, M)/(\Sph[z], \langle z \rangle) ; \Z_p)
\end{align*} 
extends naturally to a map
\begin{align*}
\TCmin((A, M)/({\Sph}[z], \langle z \rangle) ; \Z_p)[1/u] \otimes_{\Sph[z]} \Sph[z^{1/p^{\infty}}] &\to \TP((A, M)/(\Sph[z], \langle z \rangle) ; \Z_p).
\end{align*}
\item On log quasiregular semiperfectoids, the source of the above map is concentrated in even degrees, and the presheaf
\begin{align*}
(A, M)\in \lQRSPerfd_{(\cO_K,\langle \varpi\rangle)} &\mapsto \pi_0\TCmin((A, M)/({\Sph}[z], \langle z \rangle) ; \Z_p)[1/u]
\end{align*} 
is a sheaf with vanishing cohomology.
\item The unfoldings ${\gr}^0$ on   $\lQSyn_{(\cO_K,\langle \varpi\rangle)}$ gives a natural map
\begin{align*}
\gr^0\TCmin((A,M)/({\Sph}[z], \langle z \rangle) ; \Z_p)[1/u] \otimes_{\Z[z]} \Z[z^{1/p^{\infty}}] &\to \gr^0\TP((A,M)/(\Sph[z], \langle z \rangle) ; \Z_p).
\end{align*}
resulting from functoriality of unfolding. The map above is an equivalence if $(A, M)$   is the $p$-adic completion of a log smooth $({\mathcal O}_K, \langle \bar{\omega} \rangle)$-algebra of log Cartier type. 
\end{enumerate}
\end{prop}
\begin{proof}
Observe first that the map $({\Sph}[z], \langle z \rangle) \to ({\mathcal O}_K, \langle \varpi \rangle)$ is strict, so\[
\THH(\cO_K, \langle \varpi \rangle)/(\Sph[z], \langle z \rangle)\simeq \THH(\cO_K/\Sph[z]).
\] 
Part $(1)$ now follows like in \cite[Proposition 11.15 (1)]{BMS2} from the fact that the Frobenius in this case maps $u$ to $\sigma$ and the map is linear over $\Sph[z]\to \Sph[z]$, $z\mapsto z^p$.

Part $(2)$ follows again from the analogous statement with $u$ not inverted and by passage to filtered colimits. To prove part $(3)$, as in \cite[Proof of Proposition 11.15]{BMS2} it suffices to show that for $i\gg 0$ we have an equivalence
\begin{align*}
\gr^i\TCmin((A,M)/({\Sph}[z], \langle z \rangle) ; \Z_p) \otimes_{\Z[z]} \Z[z^{1/p^{\infty}}] &\to \gr^i\TP((A,M)/(\Sph[z], \langle z \rangle) ; \Z_p).
\end{align*}
For this, we can reduce modulo $z^{1/p}$, and using the equivalence
\[
\Sph_p^\wedge\simeq
\THH((\Sph,\N_{\rm perf});\Z_p),
\]
we reduce to showing that for any $(\ol{A},M)$  log smooth over $(k,\N)$ with $\N\to N$  saturated, the map
\begin{align*}
\gr^i\TCmin(\ol{A},M\oplus_\N \N_{\rm perf};\Z_p)&\to \gr^i\TP(\ol{A},M\oplus_{\N}\N_{\rm perf};\Z_p),
\end{align*}
is an equivalence for $i\gg 0$.
This follows from Proposition \ref{prop:log-segal2}, in place of \cite[Corollary 8.18]{BMS2}. 
\end{proof}

\begin{proof}[Proof of Theorems \ref{thm:BK-horizontal} and \ref{thm:BK-semistable}] 
 
Theorem \ref{thm:BK-semistable} includes Theorem \ref{thm:BK-horizontal} via base change.
We define
 $\widehat{\Prism}_{(A, M)/(\mathfrak{S},\langle z\rangle)} := \gr^0(\TCmin((A, M) / ({\Sph}[z],\langle z\rangle); \Z_p)[1/u])$. A \emph{verbatim} translation of the argument of \cite[Proof of Theorem 11.2]{BMS2} applies to conclude: replace the applications of \cite[Proposition 11.15, Corollary 11.12]{BMS2} with Proposition \ref{prop:frobdescent} and Corollary \ref{cor:frob-twisted}, respectively. 
\end{proof}
We finish this section by showing that the cohomology $\cPrism_{-/\mathfrak{S}}$ gives rise to a motivic spectrum in $\logFDA(\cO_K,\mathfrak{S})$. Let $\mathfrak{S}\{-1\}:=H^2_{\Prism}(\P^1_{\cO_K}/\mathfrak{S})$ (see \cite[Variant 9.1.6]{BhattLurie} for the bounded prism $(\mathfrak{S},E(z))$): it is an invertible $\mathfrak{S}$-module and by the crystalline comparison we get a canonical trivialization $\mathfrak{S}\{-1\}\otimes_{\mathfrak{S}}^L W(k)\simeq W(k)$ as $W(k)$-modules. Let\[
R\Gamma_{\cPrism}(-/\mathfrak{S})\{i\}:=R\Gamma_{\cPrism}(-/\mathfrak{S})\cotimes_{\mathfrak{S}}\mathfrak{S}\{-1\}^{\cotimes -i}.
\]
As in Construction \ref{constr:log-chern}, consider the prismatic Chern class $c_1^{\Prism}(\cO(1))$ considered in \cite[Variant 9.1.6]{BhattLurie}. Then, similarly to \eqref{eq:chern-map-fil}, for all $n,i\in \Z$ we have functorially in $\kX\in \FlQSm_{\cO_K}$ a map in $\widehat{\cD\cF(\mathfrak{S})}$:\[
R\Gamma_\cPrism(\kX/\mathfrak{S})\{n-i\}[-2i]\to R\Gamma_\cPrism(\P(\cE)/\mathfrak{S})\{n\}.
\]
Since $R\Gamma_\cPrism(\kX/\mathfrak{S})\{n-i\}$ is $(p,z)$-complete, it is in particular $z$-complete, so since the projection maps
\begin{equation}\label{eq:bcube-inv-BK}
R\Gamma_\cPrism(\kX\times_{\Spf (\cO_K)} (\P^n, \P^{n-1})/\mathfrak{S})\{n-i\}\to R\Gamma_\cPrism(\kX/\mathfrak{S})\{n-i\}  
\end{equation}
coincide mod $z$ (i.e. by applying $-\otimes^L_{\mathfrak{S}}W(k)$) with the maps\[
R\Gamma_{\rm crys}(\kX_k\times_{k} (\P^n, \P^{n-1})/W(k))\to R\Gamma_{\rm crys}(\kX_k/W(k)),
\]
which are equivalences by \cite[Theorem 1.3]{mericicrys}, so \eqref{eq:bcube-inv-BK} is an equivalence, implying that for all $n\in \Z$ we have graded commutative monoids in $\logFDAeff(R,\mathfrak{S})$:\[
E_*^{\mathrm{BK}}:=\{\cPrism_{-/\mathfrak{S}}\cotimes^L_{\mathfrak{S}}\mathfrak{S}\{-1\}^{\cotimes^L_{\mathfrak{S}} -i}\}_{i\in \Z},\] 
where the completion is the $(p,z)$-completion. 

We are left to show that the map
\[
R\Gamma_\cPrism(\kX/\mathfrak{S})\{i\}\to \widetilde{R\Gamma}_\cPrism(\kX\times \P^1/\mathfrak{S})\{i+1\}[2].
\]
induced by the prismatic Chern class is an equivalence.  

We argue as follows. For $\ul{\kX}$ a smooth formal scheme with trivial log structure, the following diagram commutes:\[
\begin{tikzcd}
R\Gamma_{\Prism}(\ul{\kX}/\mathfrak{S})\ar[r]\ar[dd]&R\Gamma_{\cPrism}(\ul{\kX}/\mathfrak{S})\ar[d] \\
&R\Gamma_{\cPrism}(\ul{\kX}/\mathfrak{S})\otimes^L_{\mathfrak{S}}W(k)\ar[d,no head,"\simeq"]\\
R\Gamma_{\Prism}(\ul{\kX}_k/W(k))\ar[r]&R\Gamma_{\cPrism}(\ul{\kX}_k/W(k)),
\end{tikzcd}
\] 
where the left-hand side is the absolute prismatic cohomology of \cite{BhattLurie} and the vertical map is the specialization, so, in particular, the prismatic Chern class of $R\Gamma_{\cPrism}(\ul{\kX}/\mathfrak{S})$ agrees modulo $z$ with the crystalline Chern class of $R\Gamma_{\rm crys}(\ul{\kX}_k/W(k))$ (using \cite[Proposition 7.5.5]{BhattLurie} again).

We can then reduce again to the case of the crystalline cohomology by $z$-completeness of \cite[(4.5.1)]{mericicrys}. 
Finally, the equivalences of Theorem \ref{thm:BK-horizontal} induce equivalences of graded commutative monoids:
\begin{align*}
E_*^{\mathrm{BK}}\otimes_{\mathfrak{S},\theta}^L \cO_K&\simeq (E_*^{\mathrm{dR}})_p^\wedge &\mathrm{in}\ \Fun(\Z^{\rm ds},\logFDAeff(\cO_K,\cO_K))\\
E_*^{\mathrm{BK}}\otimes_{\mathfrak{S}}^L W(k) &\simeq (i_* E_*^{\mathrm{crys}}) &\mathrm{in}\ \Fun(\Z^{\rm ds},\logFDAeff(\cO_K,W(k))),
\end{align*}
where $i_*E_n^{\rm crys} (\kX):= E_n^{\rm crys} (\kX_k) = R\Gamma_{\rm crys}(\kX_k)$, with Chern classes that agree with de Rham and crystalline Chern classes respectively again by \cite{BhattLurie}. 
By \ref{prop:build-spectra} and \ref{prop:build-spectra-maps}, we conclude that
    \begin{thm}\label{thm:BK-motivic}
    There is an oriented ring spectrum $\mathbf{E}^{\rm BK}$ in $\CAlg(\logFDA(\cO_K,\mathfrak{S}))$ such that for all $\kX\in \FlQSm(\cO_K)$ we have\[
\Map_{\logFDA(\cO_K,\mathfrak{S})}(\Sigma^{\infty}(\kX),\Sigma^{r,s}\mathbf{E}^{\rm BK}) \simeq R\Gamma_{\cPrism}(\kX/\mathfrak{S})\{s\}[r].
\]
Moreover, there are equivalences of oriented ring spectra:
\begin{align*}
\E^{\mathrm{BK}}\otimes_{\mathfrak{S},\theta}^L \cO_K&\simeq (\E^{\mathrm{dR}})_p^\wedge &\mathrm{in}\ \logFDA(\cO_K,\cO_K)\\
\E^{\mathrm{BK}}\otimes_{\mathfrak{S}}^L W(k) &\simeq (i_* \E^{\mathrm{crys}}) &\mathrm{in}\ \logFDA(\cO_K,W(k)).
\end{align*}
\end{thm}
Again we deduce from motivic properties similar results to Theorem \ref{thm:gysin-prism}, Theorem \ref{thm:coho_Grassmannian} and Theorem \ref{thm:blow-up_main}, which after base change agree with the usual properties of de Rham and crystalline cohomology. This has the following interesting application:
\begin{cor}\label{cor:gysin-Ainf}
Let $\kZ\to \kX$ be a morphism of $p$-adic proper smooth formal schemes over $\cO_K$ such that it is locally the $p$-completion of a pure codimension $d$ closed immersion of smooth schemes over $\cO_K$.
Let $\Bl_\kZ(\kX)$ denote the blow-up of $\kX$ in $\kZ$ and $E$ be the exceptional divisor, so that $(\Bl_\kZ(\kX),E)$ is log smooth over $\cO_K$.
Then for all $j$ there is a Gysin map in $\cD(A_{\inf})$ 
\[
R\Gamma_{A_{\inf}}(\kZ_{\cO_C})\{j-d\}[-2d]\to R\Gamma_{A_{\inf}}(\kX_{\cO_C})\{j\}\]
whose homotopy cofiber is\[
R\Gamma_{\cPrism}((\Bl_\kZ(\kX),E)/\mathfrak{S})\otimes_{\mathfrak{S}}A_{\inf}.
\]
There is an equivalence, functorial in $(\kX,\kZ)$,
    \[
 R\Gamma_{A_{\inf}}(\kX)\{j\} \oplus \bigoplus_{0<i< d} R\Gamma_{A_{\inf}}(\kZ_p)\{j-i\} [-2i] \xrightarrow{\sim} R\Gamma_{A_{\inf}}(\Bl_\kZ(\kX))\{j\}
 \]
For every $p$-adic log smooth formal scheme $\kY$ over $\cO_K$,
we have an equivalence of bigraded rings:\[
\phi^{\rm BK}_{n,r}\colon H^*_{A_{\inf}}(\kY_{\cO_C})\{\bullet\}\otimes_\Z Z_{n,r}\xrightarrow{\simeq} H^*_{A_{\inf}}((\mathrm{Gr}(n,r)\times \kY)_{\cO_C})\{\bullet\}
\]

\begin{proof}
    Apply \eqref{eq:htp-purity-oriented-spectra-formal}, \eqref{eq:sm-blow-up-oriented-spectra-formal} and \cite[Theorem 7.4.6]{BPO-SH} to $\E^{\rm BK}$ and then $-\otimes_{\mathfrak{S}}A_{\inf}$: the result follows from \cite[Theorem 1.2 (1)]{BMS2}.
\end{proof}
\end{cor}
\begin{rmk}
    An $A_{\inf}$-cohomology theory for general logarithmic $p$-adic formal schemes has been constructed by Diao and Yao in \cite{DiaoYao}, building on the construction of \v Cesnavi\v cius and Koshikawa in \cite{Cesnavicius-Koshikawa} for the semistable case. We expect that an analogue of \cite[Theorem 1.2 (1)]{BMS2} can be applied to compare our Breuil--Kisin cohomologies with these $A_{\inf}$-cohomologies.
\end{rmk}

\bibliographystyle{alpha}
\bibliography{bibMerici}

\end{document}

%% file: include.tex
\usepackage{amscd,verbatim}
\usepackage{amssymb}

\usepackage{bm}

\usepackage{amssymb, amsmath}
\usepackage[all]{xy}
\usepackage[inline]{enumitem}
\usepackage{hyphenat}
\usepackage[colorlinks,linkcolor=blue,citecolor=blue,urlcolor=red]{hyperref}
\usepackage[dvipsnames]{xcolor}
\usepackage{mathtools}
\usepackage{tikz-cd}
\usetikzlibrary{cd,decorations.pathreplacing}
\usepackage[labelformat=empty]{caption}
\usepackage{xfrac}
 \usepackage[ left=3cm, right=3cm, top = 2.8cm, bottom = 2.8cm ]{geometry}
\usepackage{stmaryrd}
\usepackage[normalem]{ulem}
\usepackage{bbm}

\makeatletter
\def\@tocline#1#2#3#4#5#6#7{\relax
  \ifnum #1>\c@tocdepth 
  \else
    \par \addpenalty\@secpenalty\addvspace{#2}%
    \begingroup \hyphenpenalty\@M
    \@ifempty{#4}{%
      \@tempdima\csname r@tocindent\number#1\endcsname\relax
    }{%
      \@tempdima#4\relax
    }%
    \parindent\z@ \leftskip#3\relax \advance\leftskip\@tempdima\relax
    \rightskip\@pnumwidth plus4em \parfillskip-\@pnumwidth
    #5\leavevmode\hskip-\@tempdima
      \ifcase #1
      \or\or \hskip 2em \or \hskip 2em \else \hskip 3em \fi%
      #6\nobreak\relax
    \dotfill\hbox to\@pnumwidth{\@tocpagenum{#7}}\par
    \nobreak
    \endgroup
  \fi}
\makeatother

\newcommand{\A}{\mathbf{A}}

\newcommand{\C}{\mathbf{C}}
\newcommand{\E}{\mathbf{E}}
\newcommand{\G}{\mathbf{G}}
\renewcommand{\L}{\mathbb{L}}
\newcommand{\N}{\mathbb{N}}
\renewcommand{\P}{\mathbf{P}}
\newcommand{\Q}{\mathbb{Q}}

\newcommand{\Z}{\mathbb{Z}}
\newcommand{\F}{\mathbb{F}}

\newcommand{\sF}{\mathcal{F}}

\newcommand{\sP}{\mathcal{P}}

\newcommand{\Sph}{\mathbb{S}}

\newcommand{\bT}{\mathbb{T}}

\newcommand{\ul}[1]{{\underline{#1}}}

\newcommand{\Map}{\operatorname{map}}
\newcommand{\uMap}{\operatorname{\underline{map}}}
\newcommand{\Hom}{\operatorname{Hom}}

\newcommand{\Pic}{\operatorname{Pic}}

\newcommand{\Spec}{\operatorname{Spec}}
\newcommand{\Spf}{\operatorname{Spf}}

\newcommand{\Comp}{\operatorname{Comp}}

\newcommand{\Shv}{\operatorname{\mathbf{Shv}}}

\newcommand{\pro}[1]{\text{\rm pro}_{#1}\text{\rm--}}

\newcommand{\eff}{{\operatorname{eff}}}

\newcommand{\op}{{\operatorname{op}}}

\newcommand{\Zar}{{\operatorname{Zar}}}
\newcommand{\Nis}{{\operatorname{Nis}}}
\newcommand{\et}{{\operatorname{\acute{e}t}}}

\newcommand{\id}{{\operatorname{Id}}}

\newcommand{\Sym}{{\operatorname{Sym}}}
\newcommand{\Tot}{{\operatorname{Tot}}}

\renewcommand{\lim}{\operatornamewithlimits{\varprojlim}}
\newcommand{\colim}{\operatornamewithlimits{\varinjlim}}

\newcommand{\ol}{\overline}

\renewcommand{\phi}{\varphi}
\renewcommand{\epsilon}{\varepsilon}

\newcommand{\cotimes}{\operatorname{\widehat{\otimes}}}

\newcommand{\CI}{\operatorname{\mathbf{CI}}}

\newcommand{\Bl}{{\mathbf{Bl}}}

\newcommand{\bcube}{{\ol{\square}}}

\newcommand{\M}{\mathbf{M}}

\newcommand{\perf}{\mathrm{perf}}

\newcounter{spec}
{\end{list}}%

\newtheorem{lemma}{Lemma}[section]
\newtheorem{thm}[lemma]{Theorem}

\newtheorem{prop}[lemma]{Proposition}

\newtheorem{cor}[lemma]{Corollary}

\theoremstyle{definition}
\newtheorem{defn}[lemma]{Definition}
\newtheorem{constr}[lemma]{Construction}
\newtheorem{definition}[lemma]{Definition}

\theoremstyle{remark}

\newtheorem{rmk}[lemma]{Remark}

\newtheorem{example}[lemma]{Example}

\numberwithin{equation}{section}

\numberwithin{equation}{lemma}

\colorlet{LightRubineRed}{RubineRed!70!}

\usepackage[color = blue!20, bordercolor = black, textsize = tiny]{todonotes}
\usepackage{relsize} 
\usepackage[bbgreekl]{mathbbol} 
\usepackage{amsfonts} 
\DeclareSymbolFontAlphabet{\mathbb}{AMSb} 
\DeclareSymbolFontAlphabet{\mathbbl}{bbold} 
\newcommand{\Prism}{{\mathlarger{\mathbbl{\Delta}}}}
\DeclareSymbolFontAlphabet{\mathbbl}{bbold} 
\newcommand{\cPrism}{{\widehat{\mathlarger{\mathbbl{\Delta}}}}}

%% file: defin.tex
\def\THH{\operatorname{THH}}
\def\TC{\operatorname{TC}}
\def\TCmin{\operatorname{TC}^-}
\def\TP{\operatorname{TP}}
\def\HH{\operatorname{HH}}
\def\HC{\operatorname{HC}}
\def\HCmin{\operatorname{HC}^-}
\def\HP{\operatorname{HP}}

\def\Fil{\operatorname{Fil}}
\def\Gr{\operatorname{Gr}}
\def\gr{\operatorname{gr}}

\def\QSyn{\operatorname{QSyn}}
\def\QRSPerfd{\operatorname{QRSPerfd}}
\def\lQSyn{\operatorname{lQSyn}}
\def\lQRSPerfd{\operatorname{lQRSPerfd}}

\def\syn{\mathrm{syn}}
\def\Fsyn{\mathrm{Fsyn}}
\def\Fet{\mathrm{F\acute{e}t}}

\def\LogRec{\operatorname{\mathbf{LogRec}}}
\def\Ch{\operatorname{\mathrm{Ch}}}

\def\ltr{\mathrm{ltr}}

\def\kX{\mathfrak{X}}
\def\kY{\mathfrak{Y}}
\def\kZ{\mathfrak{Z}}

\def\otCIsp{\otimes_{\CI}^{sp}}
\def\otCINissp{\otimes_{\CI}^{\Nis,sp}}

\def\tL{\tilde{L}}
\def\tX{\tilde{X}}
\def\tY{\tilde{Y}}
\def\tF{\widetilde{F}}
\def\tG{\widetilde{G}}

\def\Sh{\operatorname{\mathbf{Shv}}}
\def\PSh{\operatorname{\mathbf{PSh}}}
\def\Shltr{\operatorname{\mathbf{Shv}_{dNis}^{ltr}}}
\def\Shlog{\operatorname{\mathbf{Shv}_{dNis}^{log}}}
\def\Shvlog{\operatorname{\mathbf{Shv}^{log}}}
\def\Sm{\operatorname{\mathrm{Sm}}}
\def\SmlSm{\operatorname{\mathrm{SmlSm}}}
\def\lSm{\operatorname{\mathrm{lSm}}}
\def\FlSm{\operatorname{\mathrm{FlSm}}}
\def\FlQSm{\operatorname{\mathrm{FlQSm}}}
\def\FlQSyn{\operatorname{\mathrm{FlQSyn}}}
\def\lCor{\operatorname{\mathrm{lCor}}}
\def\SmlCor{\operatorname{\mathrm{SmlCor}}}
\def\PShltr{\operatorname{\mathbf{PSh}^{ltr}}}
\def\PShlog{\operatorname{\mathbf{PSh}^{log}}}
\def\logCI{\mathbf{logCI}} 

\def\Mod{\operatorname{Mod}}
\newcommand{\DM}[1][]{\operatorname{\mathcal{DM}_{#1}}}
\newcommand{\DMeff}[1][]{\operatorname{\mathcal{DM}^{\eff}_{#1}}}
\newcommand{\DA}[1][]{\operatorname{\mathcal{DA}_{#1}}}
\newcommand{\DAeff}[1][]{\operatorname{\mathcal{DA}^{\eff}_{#1}}}
\newcommand{\FDA}[1][]{\operatorname{\mathcal{DA}_{#1}}}
\newcommand{\FDAeff}[1][]{\operatorname{\mathcal{DA}^{\eff}_{#1}}}
\newcommand{\SH}[1][]{\operatorname{\mathcal{SH}_{#1}}}
\newcommand{\logSH}[1][]{\operatorname{\mathbf{log}\mathcal{SH}_{#1}}}
\newcommand{\logDA}[1][]{\operatorname{\mathbf{log}\mathcal{DA}_{#1}}}
\newcommand{\logDAeff}[1][]{\operatorname{\mathbf{log}\mathcal{DA}^{\eff}_{#1}}}
\newcommand{\logDM}[1][]{\operatorname{\mathbf{log}\mathcal{DM}_{#1}}}
\newcommand{\logDMeff}[1][]{\operatorname{\mathbf{log}\mathcal{DM}^{\eff}_{#1}}}
\newcommand{\logFDA}[1][]{\operatorname{\mathbf{log}\mathcal{FDA}_{#1}}}
\newcommand{\logFDAeff}[1][]{\operatorname{\mathbf{log}\mathcal{FDA}^{\eff}_{#1}}}
\newcommand{\WOmega}{\operatorname{\mathrm{W}\Omega}}
\def\Log{\operatorname{\mathcal{L}\textit{og}}}
\def\Rsc{\operatorname{\mathcal{R}\textit{sc}}}
\def\Pro{\mathrm{Pro}\textrm{-}}
\def\pro{\mathrm{pro}\textrm{-}}
\def\dg{\mathrm{dg}}
\def\plim{\mathrm{``lim"}}
\def\ker{\mathrm{ker}}
\def\coker{\mathrm{coker}}
\def\PrL{\mathcal{P}\mathrm{r^L}}
\def\PrLo{\mathcal{P}\mathrm{r^{L,\otimes}}}
\def\Spt{\mathcal{S}\mathrm{pt}}
\def\PSpt{\mathrm{Pre}\mathcal{S}\mathrm{pt}}
\def\Fun{\mathrm{Fun}}
\def\Sym{\mathrm{Sym}}
\def\CAlg{\mathrm{CAlg}}
\def\Poly{\mathrm{Poly}}
\def\Cat{\mathrm{Cat}}

\def\Alb{\operatorname{Alb}}
\def\bAlb{\mathbf{Alb}}
\def\Gal{\operatorname{Gal}}

\def\hofib{\mathrm{hofib}}
\def\fib{\mathrm{fib}}
\def\triv{\mathrm{triv}}
\def\ABl{\mathcal{A}\textit{Bl}}
\def\divsm#1{{#1_\mathrm{div}^{\mathrm{Sm}}}}

\def\cA{\mathcal{A}}
\def\cB{\mathcal{B}}
\def\cC{\mathcal{C}}
\def\cD{\mathcal{D}}
\def\cE{\mathcal{E}}
\def\cF{\mathcal{F}}
\def\cG{\mathcal{G}}
\def\cH{\mathcal{H}}
\def\cI{\mathcal{I}}
\def\cS{\mathcal{S}}
\def\cM{\mathcal{M}}
\def\cO{\mathcal{O}}
\def\cP{\mathcal{P}}

\def\rN{\mathrm{N}}

\def\tcA{\widetilde{\mathcal{A}}}
\def\tcB{\widetilde{\mathcal{B}}}
\def\tcC{\widetilde{\mathcal{C}}}
\def\tcD{\widetilde{\mathcal{D}}}
\def\tcE{\widetilde{\mathcal{E}}}
\def\tcF{\widetilde{\mathcal{F}}}

\def\one{\mathbbm{1}}

\def\XP{X \backslash \sP}
\def\M0a{{}^t\cM_0^a}
\newcommand{\Ind}{{\operatorname{Ind}}}

\def\Xkbar{\overline{X}_{\overline{k}}}
\def\dx{{\rm d}x}

\newcommand{\dNis}{{\operatorname{dNis}}}
\newcommand{\loget}{{\operatorname{l\acute{e}t}}}
\newcommand{\ket}{{\operatorname{k\acute{e}t}}}
\newcommand{\ABNis}{{\operatorname{AB-Nis}}}
\newcommand{\sNis}{{\operatorname{sNis}}}
\newcommand{\sZar}{{\operatorname{sZar}}}
\newcommand{\set}{{\operatorname{s\acute{e}t}}}
\newcommand{\sqsyn}{{\operatorname{sqsyn}}}
\newcommand{\cofib}{\mathrm{Cofib}}

\newcommand{\Gmlog}{\G_m^{\log}}
\newcommand{\Gmlogred}{\overline{\G_m^{\log}}}

\newcommand{\varcolim}{\mathop{\mathrm{colim}}}
\newcommand{\varlim}{\mathop{\mathrm{lim}}}
\newcommand{\tensor}{\otimes}
\newcommand{\Spc}{\mathrm{Spc}}

\newcommand{\eq}[2]{\begin{equation}\label{#1}#2 \end{equation}}
\newcommand{\eqalign}[2]{\begin{equation}\label{#1}\begin{aligned}#2 \end{aligned}\end{equation}}

\def\varplim#1{\text{``}\varlim_{#1}\text{''}}
\def\det{\mathrm{d\acute{e}t}}